\documentclass[12pt,a4paper,reqno,oneside]{amsart}

\usepackage{mathabx}
\usepackage{graphicx}
\usepackage{amsmath}
\usepackage{amsthm}
\usepackage{amssymb}
\usepackage{bbm}
\usepackage[sort,nocompress]{cite} 

\usepackage{amsfonts}
\usepackage{enumerate}

\usepackage{xcolor}
\usepackage{geometry}
\usepackage[colorlinks,citecolor=red]{hyperref}
\usepackage[notref,notcite,final]{showkeys}

\usepackage[capitalise]{cleveref}

\usepackage{courier}
\usepackage{listings} 
\lstdefinelanguage{Julia}%
  {morekeywords={abstract,break,catch,const,continue,do,else,elseif,%
      end,export,false,for,function,immutable,import,importall,if,in,%
      macro,module,otherwise,quote,return,switch,true,try,type,typealias,%
      using,while},%
   sensitive=true,%
   alsoother={$},%
   morecomment=[l]\#,%
   morecomment=[n]{\#=}{=\#},%
   morestring=[s]{"}{"},%
   morestring=[m]{'}{'},%
}[keywords,comments,strings]%

\newtheorem{lemma}{Lemma}[section]

\newtheorem{theorem}[lemma]{Theorem}
\newtheorem{example}[lemma]{Example}

\newtheorem{setting}[lemma]{Setting}
\theoremstyle{definition}

\newtheorem{listing}[lemma]{Listing}

\renewcommand{\gets}{\curvearrowleft}
\providecommand{\N}{{\ensuremath{\mathbbm{N}}}}
\providecommand{\Z}{{\ensuremath{\mathbbm{Z}}}}
\providecommand{\R}{{\ensuremath{\mathbbm{R}}}}

\renewcommand{\P}{{\ensuremath{\mathbbm{P}}}}
\providecommand{\E}{{\ensuremath{\mathbbm{E}}}}

\providecommand{\1}{{\ensuremath{\mathbbm{1}}}}
\providecommand{\F}{{\ensuremath{\mathbbm{F}}}}
\newcommand{\tnorm}[1]{{\left\vert\kern-0.25ex\left\vert\kern-0.25ex\left\vert #1     \right\vert\kern-0.25ex\right\vert\kern-0.25ex\right\vert}}

\newcommand{\supnorm}[1]{{\left\vert\kern-0.25ex\left\vert\kern-0.25ex\left\vert #1     \right\vert\kern-0.25ex\right\vert\kern-0.25ex\right\vert}}

\newcommand{\rdown}[1]{\lfloor #1\rfloor}
\providecommand{\var}{{\ensuremath{\mathbbm{V}}}}

\author[A. Neufeld]{Ariel Neufeld$^{1}$}
\address{$^1$  Division of Mathematical Sciences, School of Physical and Mathematical Sciences, Nanyang Technological University, Singapore}
\email{ariel.neufeld@ntu.edu.sg}

\author[T.A. Nguyen]{Tuan Anh Nguyen$^{2}$}
\address{$^2$  Faculty of Mathematics, Bielefeld University, Germany}
\email{tnguyen@math.uni-bielefeld.de}
\thanks{Financial support by the Nanyang Assistant Professorship Grant (NAP Grant) \textit{Machine Learning based Algorithms
in Finance and Insurance} is gratefully acknowledged.}
\title[MLP and DNN overcome the CoD when approximating PDEs]{Multilevel Picard approximations and deep neural networks with ReLU, leaky ReLU, and softplus activation  overcome the curse of dimensionality when approximating semilinear parabolic partial differential equations in $L^p$-sense}
\begin{document}\begin{abstract}
We prove that 
multilevel Picard approximations and
deep neural networks with ReLU, leaky ReLU, and softplus activation are capable of approximating solutions of semilinear Kolmogorov
PDEs 
in $L^\mathfrak{p}$-sense, $\mathfrak{p}\in [2,\infty)$,
in the case of gradient-independent, Lipschitz-continuous nonlinearities, while the computational effort of the 
multilevel Picard approximations
and
the required
number of parameters in the neural networks grow at most polynomially in both dimension $d\in \N$ and
 reciprocal of the
prescribed  accuracy $\epsilon$. 
\end{abstract}
\maketitle

\vspace{-0.15cm}
\section{Introduction}
Partial differential equations (PDEs) are important tools to analyze many real world phenomena, e.g., in financial engineering, economics, quantum mechanics, or statistical physics to name but a few. 
In most of the cases such high-dimensional nonlinear PDEs cannot be solved explicitly. It is one of
the most challenging problems in applied mathematics to approximately solve high-dimensional nonlinear PDEs.
In particular, it is very difficult to find approximation schemata for nonlinear PDEs for which one can
rigorously prove that they do overcome the so-called \emph{curse of dimensionality} in the sense that the computational complexity only grows polynomially in the space dimension $d$ of the PDE and the reciprocal
$\frac{1}{\varepsilon}$ of the accuracy $\varepsilon$.

In recent years, 
there are two types of approximation methods which are quite
successful in the numerical approximation of solutions of high-dimensional nonlinear  PDEs:  neural network based 
approximation methods for PDEs, cf.,
\cite{al2022extensions,beck2020deep,beck2021deep,beck2019machine,
beck2020overview,berner2020numerically,castro2022deep,CHW2022,
ew2017deep,ew2018deep,weinan2021algorithms,frey2022convergence,frey2022deep,GPW2022,
gnoatto2022deep,gonon2023random,gonon2021deep,gonon2023deep,GHJVW2023,
han2018solving,han2020convergence, 
han2019solving,hure2020deep,hutzenthaler2020proof,ito2021neural,
jacquier2023deep,jacquier2023random,JSW2021,
lu2021deepxde,nguwi2022deep,nguwi2022numerical,nguwi2023deep,
raissi2019physics,reisinger2020rectified,sirignano2018dgm,zhang2020learning,han2018solving,NSW2024}   
and multilevel Monte-Carlo based approximation methods for PDEs, cf.,
\cite{beck2020overcomingElliptic,beck2020overcoming,becker2020numerical,
hutzenthaler2019multilevel,hutzenthaler2021multilevel,
giles2019generalised,HJK2022,HJKNW2020,
HJKN2020,
HJvW2020,HK2020,hutzenthaler2022multilevel,
HN2022a,NW2022,NNW2023,NW2023,HJKP2021}.

For multilevel Monte-Carlo based algorithms it is often possible to provide a complete convergence and complexity analysis.
It has been proven that under some suitable assumptions, e.g., Lipschitz continuity
on the linear part, the nonlinear part, and the initial (or terminal) condition function of the PDE under consideration, 
the multilevel Picard approximation algorithms can overcome the curse of
dimensionality in the sense that the number of computational operations of
the proposed Monte-Carlo based approximation method grows at most polynomially in both the reciprocal
$\frac{1}{\varepsilon}$ of the prescribed approximation accuracy $\varepsilon\in (0,1)$ and the PDE dimension $d\in \N$.
More precisely, \cite{hutzenthaler2021multilevel} considers smooth semilinear parabolic heat equations. 
Later, \cite{HJKNW2020} extends \cite{hutzenthaler2021multilevel} to a more general setting, namely, semilinear heat equations which are not necessary smooth. \cite{beck2020overcoming} considers semilinear heat equation with more general nonlinearities, namely locally Lipschitz nonlinearities.
 \cite{HK2020,HJK2022} considers semilinear heat equations with gradient-dependent Lipschitz nonlinearities and \cite{NNW2023,NW2023} extends them to semilinear PDEs with general drift and diffusion coefficients.
\cite{HJvW2020} studies Black-Scholes-types semilinear PDEs.
\cite{HJKN2020} consider semilinear parabolic PDEs with nonconstant drift and diffusion coefficients.
\cite{HN2022a} considers a slightly more general setting than \cite{HJKN2020}, namely semilinear PDEs with locally monotone coefficient functions. 
\cite{HN2022b} introduced a schema with applications to forward backward stochastic differential equations under the assumption that 
both $\mu$ and $\sigma$ are $C^2$.
\cite{NW2022} studies semilinear partial integro-differential equations.
\cite{hutzenthaler2022multilevel} considers McKean-Vlasov stochastic differential equations (SDEs) with constant diffusion coefficients. \cite{beck2020overcomingElliptic} studies a special type of elliptic equations.
Almost all the works listed above prove $L^2$-error estimates except
 \cite{HJKP2021,HN2022b}, which draw their attention to $L^\mathfrak{p}$-error estimates, $\mathfrak{p}\in [2,\infty)$.

Numerical experiments indicate that 
deep learning methods work exceptionally well when approximating solutions of high-dimensional PDEs and that they do not suffer from the curse of dimensionality.
However, there exist only few theoretical results proving that deep learning based approximations of solutions of PDEs indeed do not suffer from the curse of dimensionality.
More precisely, 
\cite{BGJ2020}
 shows that empirical risk minimization  over deep neural network (DNN) hypothesis classes overcomes the curse
of dimensionality for the numerical solution of linear Kolmogorov equations with affine coefficients.
Next,
\cite{EGJS2022} considers the pricing problem  of a European
best-of-call option on a basket of $d$ assets within the Black–Scholes model  and  proves that the solution to the $d$-variate option pricing
problem can be approximated up to an $\epsilon$-error by a deep ReLU network with depth
$\mathcal{O}(\ln(d)\ln(\epsilon^{-1}) + (\ln(d))^2)$
and $\mathcal{O}(d^{2+\frac{1}{n}}\epsilon^{-\frac{1}{n}})$
nonzero weights, where $n\in \N$ is arbitrary (with the constant implied in $\mathcal{O}(\cdot)$ depending on $n$).
Furthermore, \cite{gonon2023random}
investigates the use of random  neural networks for learning Kolmogorov partial integro-differential equations (PIDEs) associated to Black-Scholes and more general exponential Lévy models. Here, random  neural networks are single-hidden-layer feedforward neural networks in which the input weights are randomly generated and only the output weights are trained.
In addition, \cite{NN2025} proves that
rectified deep neural networks overcome the curse of dimensionality when approximating solutions of McKean--Vlasov stochastic differential equations.
Moreover, \cite{gonon2021deep} studies the expression rates of DNNs  for option
prices written on baskets of $d$ risky assets whose log-returns are modelled by a multivariate L\'evy process with general correlation structure of jumps. 
Note that the PIDEs studied by \cite{gonon2021deep} are also Black-Scholes-type PIDEs (see \cite[Display~(2.3)]{gonon2021deep}). 
Next, 
\cite{gonon2023deep} proves that DNNs with ReLU activation function are able to
express viscosity solutions of Kolmogorov linear PIDEs on state spaces of
possibly high dimension $d$. Furthermore,
\cite{GHJVW2023} proves that DNNs overcome the curse of dimensionality when approximating the solutions to Black-Scholes PDEs and \cite{JSW2021} proves
that DNNs overcome the curse
of dimensionality in the numerical approximation of linear Kolmogorov PDEs with constant
diffusion and nonlinear drift coefficients. 
In addition, \cite{NS2023} proves that the  solution of the linear heat equation can be approximated by a random neural network whose amount of neurons only grow polynomially in the space dimension of the PDE and the reciprocal of the accuracy, hence overcoming the curse of dimensionality when approximating such an equation.
Moreover,
\cite{hutzenthaler2020proof}
proves that DNNs overcome the curse of dimensionality 
in the numerical approximation of semilinear heat equations and \cite{AJK+2023} extends \cite{hutzenthaler2020proof} to estimates with respect to $L^\mathfrak{p}$-norms, $\mathfrak{p}\in [2,\infty)$, when approximating the semilinear heat equation. Furthermore, \cite{AJKP2024} demonstrates  space-time $L^\mathfrak{p}$-error estimates, $\mathfrak{p}\in [2,\infty)$,
when approximating the semilinear heat equation.
Next, \cite{CHW2022} extends \cite{hutzenthaler2020proof} to semilinear PDEs with general drift and diffusion coefficients
and \cite{NNW2023} extends \cite{hutzenthaler2020proof} to semilinear PIDEs. Note that except \cite{AJK+2023,AJKP2024} all the  works mentioned in this paragraph  establish $L^2$-error estimates, but not $L^\mathfrak{p}$-estimates for general $\mathfrak{p}\in [2,\infty)$.
\\

\noindent
The main novelty of our paper is the following:
\vspace{-0.15cm}
\begin{enumerate}[(A)]
\item We extend the $L^2$-complexity analysis in \cite{HJKN2020} to an $L^\mathfrak{p}$-complexity analysis, $\mathfrak{p}\in [2,\infty)$. More precisely, in our first main result, \cref{c40} below, we prove that the MLP algorithms introduced by \cite{HJKN2020} overcome the curse of dimensionality when approximating semilinear parabolic 
PDEs in $L^\mathfrak{p}$-sense, $\mathfrak{p}\in [2,\infty)$.
\item We extend the result by \cite{CHW2022} from $L^\mathfrak{2}$  to an $L^\mathfrak{p}$-sense, $\mathfrak{p}\in [2,\infty)$, and from DNNs with ReLU activation to DNNs with more activation functions including now DNNs with ReLU, leaky ReLU, or softplus activation, see \cref{a08b} below, which is our second main result. 
More precisely, we 
show that for every $\mathfrak{p}\in [2,\infty)$ we
have that solutions of semilinear PDEs with Lipschitz continuous nonlinearities can be
approximated in the $L^\mathfrak{p}$-sense by DNNs with ReLU, leaky ReLU, or softplus activation
without the curse of dimensionality.
\end{enumerate}

\subsection{Notations}

Throughout this paper we use the following notations. 
Let 
$\R$ denote the set of all real numbers. Let
$\Z, \N_0, \N $ denote the sets which satisfy that $\Z=\{\ldots,-2,-1,0,1,2,\ldots\}$, $\N=\{1,2,\ldots\}$, $\N_0=\N\cup\{0\}$. Let $\nabla$ denote the gradient  and $\operatorname{Hess}$ denote the Hessian matrix. For every matrix $A$ let $A^\top$ denote the transpose of $A$ and let $\mathrm{trace}(A)$ denote the trace of $A$ when $A$ is a square matrix. 
For every probability space $(\Omega,\mathcal{F},\P)$, every random variable
$X\colon \Omega\to \R$, and every $s\in [1,\infty)$ let 
$\lVert X\rVert_s \in [0,\infty]$ satisfy that 
$\lVert X\rVert_s=(\E[\lvert X\rvert^s])^{\frac{1}{s}}$.
For every $d\in \N$ let $\lVert\cdot \rVert,\supnorm{\cdot}\colon \R^d\to [0,\infty)$ satisfy for all $x=(x_i)_{i\in [1,d]\cap\Z}\in \R^d$ that
$\lVert x\rVert=\sqrt{\sum_{i=1}^{d}\lvert x_i\rvert^2}$ and
$\supnorm{x}=\sup_{i\in [1,d]\cap\Z}\lvert x_i\rvert$.
For every $d\in \N$ let $\langle\cdot,\cdot\rangle\colon 
\R^d\times\R^d\to \R$
satisfy for all $x=(x_i)_{i\in [1,d]\cap\Z}$,
$y=(y_i)_{i\in [1,d]\cap\Z}$ that $\langle x,y\rangle=\sum_{i=1}^d x_iy_i$.
For every $d\in \N$ let $\lVert\cdot \rVert\colon \R^{d\times d}\to [0,\infty)$ satisfy for all
$a=(a_{ij})_{i,j\in [1,d]\cap \Z}\in \R^{d\times d}$ that
$\lVert a\rVert=\sqrt{\sum_{i=1}^{d}\sum_{i=1}^{d}\lvert a_{ij}\rvert^2  }$.
 When applying a result we often use a phrase like ``Lemma 3.8 with $d\gets (d-1)$''
that should be read as ``Lemma 3.8 applied with $d$ (in the notation of Lemma 3.8) replaced
by $(d-1 )$ (in the current notation)''.

\subsection{MLP approximations overcome the curse of dimensionality when approximating semilinear parabolic PDEs in $L^\mathfrak{p}$-sense}
\begin{theorem}\label{c40}
Let $T,\mathbf{k}\in (0,\infty)$, $\mathfrak{p}\in [2,\infty)$, $c\in [\mathfrak{p}^2,\infty)$.
Let $M\colon \N\to \N$ satisfy for all $n\in\N $ that
$M_n=\max \{k\in \N\colon k\leq\exp (\lvert\ln(n)\rvert^{1/2})\}$.
 For every $d\in \N$ let $g^d\in C(\R^d,\R)$, $f\in C(\R,\R)$, $\mu^d\in C(\R^d,\R^d)$, $\sigma^d\in C(\R^d,\R^{d\times d})$. Assume for all
$x,y\in \R^d$, $v,w\in \R^d$ that
\begin{align}
\max \{
\lvert Tf(0)\rvert, \lvert g^d(0)\rvert, \lVert\mu^d(0)\rVert, \lVert\sigma^d(0)\rVert
\}\leq cd^c,\quad \lvert g(x)\rvert\leq c(d^c+\lVert x\rVert^2)^\frac{1}{2},\label{t23}
\end{align}
\begin{align}
\max \{\sqrt{T}\lvert g^d(x)-g^d(y)\rvert, \lVert\mu^d(x)-\mu^d(y)\rVert,
\lVert\sigma^d(x)-\sigma^d(y)\rVert
\}\leq c\lVert x-y\rVert,\label{t24}
\end{align}
\begin{align}
\lvert
f(w)-f(v)\rvert\leq c\lvert w-v\rvert.\label{t25}
\end{align}
Let $(\Omega,\mathcal{F},\P, (\F_t)_{t\in[0,T]})$ be a filtered probability space which satisfies the usual conditions\footnote{Let $T \in [0,\infty)$ and let ${\bf \Omega} = (\Omega,\mathcal{F},\P, (\F_t)_{t\in[0,T]})$ be 
a filtered probability space. 
Then we say that ${\bf \Omega}$
satisfies the usual conditions if and only if 
it holds that $\{ A\in \mathcal F: \P(A)=0 \} \subseteq \F_0$ and $\forall \, t\in [0,T) \colon \mathbb{F}_t 
= \cap_{ s \in (t,T] } \F_s$.}.
Let 
$  \Theta = \bigcup_{ n \in \N }\! \Z^n$.
Let $\mathfrak{t}^\theta\colon \Omega\to[0,1]$, $\theta\in \Theta$, be identically distributed and independent random variables. Assume for all $t\in(0,1)$ that $\P(\mathfrak{t}^0\leq t)=t$. For every $d\in \N$ let $W^{d,\theta}\colon [0,T]\times\Omega \to \R^{d}$, $\theta\in\Theta$, be independent standard $(\F_{t})_{t\in[0,T]}$-Brownian motions. 
Assume that
$(\mathfrak{t}^\theta)_{\theta\in\Theta}$ and
$(W^{d,\theta})_{d\in \N,\theta\in\Theta}$ are independent. For every $K\in \N$ let
$\rdown{\cdot}_K\colon \R\to \R$ satisfy for all $t\in \R$
that $\rdown{t}_K=\max ( \{0,\frac{T}{K},\ldots,\frac{(K-1)T}{T},T\}\cap ((-\infty,t)\cup\{0\}) )$. For every 
$d,K\in \N$,
$\theta\in \Theta$, $t\in[0,T]$, $x\in \R^d$
let $Y^{d,\theta,K,t,x}= (Y^{d,\theta,K,t,x}_s)_{s\in [t,T]}\colon [t,T]\times \Omega\to \R^d$ satisfy for all $s\in [t,T]$ that
$Y^{d,\theta,K,t,x}_t=x$ and
\begin{align}
Y^{d,\theta,K,t,x}_s=Y^{d,\theta,K,t,x}_{\max\{t,\rdown{s}_K\}}
+\mu^d(Y^{d,\theta,K,t,x}_{\max\{t,\rdown{s}_K\}})
(s-\max \{t,\rdown{s}_K\})
+
\sigma^d(Y^{d,\theta,K,t,x}_{\max\{t,\rdown{s}_K\}})
(W^{d,\theta}_s-W^{d,\theta}_{\max \{t,\rdown{s}_K\}}).\label{c09c}
\end{align}
Let $U^{d,\theta,K}_{n,m}\colon [0,T]\times \R^d\times \Omega\to \R$, $n\in \Z$, $d,K,m\in \N$, $\theta\in \Theta$, satisfy for all
$\theta\in \Theta$, $d,K,m\in \N$, $n\in \N_0$, $t\in [0,T]$, $x\in \R^d$ that
$
{U}_{-1,m}^{d,\theta,K}(t,x)={U}_{0,m}^{d,\theta,K}(t,x)=0$ and
\begin{align} \begin{split} 
\label{t27f}
&  {U}_{n,m}^{d,\theta,K}(t,x)
=  \frac{1}{m^n}\sum_{i=1}^{m^n}
g^d(Y^{d,(\theta,0,-i),K,t,x}_T)\\
&+
\sum_{\ell=0}^{n-1}\frac{T-t}{m^{n-\ell}}
    \sum_{i=1}^{m^{n-\ell}}
     \bigl( f\circ {U}_{\ell,m}^{d,(\theta,\ell,i),K}
-\1_{\N}(\ell)
f\circ {U}_{\ell-1,m}^{d,(\theta,-\ell,i),K}
\bigr)\!
\left(t+(T-t)\mathfrak{t}^{(\theta,\ell,i)},Y_{t+(T-t)\mathfrak{t}^{(\theta,\ell,i)}}^{d,(\theta,\ell,i),K,t,x}\right).\end{split}
\end{align}
Let $(C^{d,K}_{n,m})_{d,K\in \N,n,m\in \Z}\subseteq \N_0$ satisfy for all $d,K\in \N$,
$m,n\in \N$ that
\begin{align}\label{c29}
C_{0,m}^{d,K}=0,\quad 
C_{n,m}^{d,K}\leq (cd^c+cd^cK)
m^n
+ \sum_{\ell=0}^{n-1}m^{n-\ell}\left(2cd^c+cd^cK+
C_{\ell,m}^{d,K}+
C_{\ell-1,m}^{d,K}\right).
\end{align}
Then the following items are true.
\begin{enumerate}[(i)]\item \label{t30}
For every $d\in \N$ there exists a unique at most polynomially growing viscosity solution $u^d$ of 
\begin{align}
\frac{\partial u^d}{\partial t}(t,x)
+\frac{1}{2}\mathrm{trace}(\sigma^d(x)(\sigma^d(x))^\top
(\mathrm{Hess}_x u^d(t,x) ))
+\langle \mu^d(x),(\nabla_xu^d) (t,x)\rangle
+f(u^d(t,x))=0\label{t30b}
\end{align}
with $u^d(T,x)=g^d(x)$ for $(t,x)\in (0,T)\times \R^d$.
\item \label{t31}There exist $(C_\delta)_{\delta\in (0,1)}\subseteq (0,\infty)$, $(n(d,\epsilon))_{d\in \N,\epsilon\in (0,1)}\subseteq \N $ such that for all $d\in \N$, $\epsilon\in (0,1)$ it holds that
\begin{align}
\sup_{t\in [0,T],x\in [0,\mathbf{k}]^d}\left\lVert
{U}_{n(d,\varepsilon),M_{n(d,\varepsilon)}}^{d,0,(M_{n(d,\varepsilon)})^{{n(d,\varepsilon)}}}(t,x)-u^d(t,x)\right\rVert_\mathfrak{p}
\leq \epsilon 
\quad\text{and}\quad  C^{d,(M_{n(d,\epsilon)})^{n(d,\epsilon)}}_{n(d,\epsilon),M_{n(d,\epsilon)}}\leq \eta d^\eta \epsilon^{-(4+\delta)}.
\end{align}
\end{enumerate}

\end{theorem}
The proof of \cref{c40} is presented directly after the proof of \cref{c02}.
Let us comment on the mathematical objects in \cref{c40}.
Our goal here in \cref{c40} is to approximately solve the family of semilinear parabolic PDEs in \eqref{t30b} indexed by $d\in \N$.
The functions $\mu^d$ and $\sigma^d$
are the drift and diffusion coefficients of the linear part of the PDEs. The function $f$ is the nonlinear part of the PDEs. The functions $g^d$ is the terminal condition at time $T$ of the PDEs.
Next, \eqref{t23}--\eqref{t25} are usual regularity properties for the coefficients of the PDEs, which assure that the PDEs has  unique viscosity solutions. Note that the existence and uniqueness of viscosity solutions of semilinear PDEs of the form \eqref{t30b} are not a new result, see
\cite[Theorem 1.1]{beck2021nonlinear}. We
still state it in 
\cref{c40}\eqref{t30} and \cref{a08b}\eqref{c41} so that the 
statements of the main results in 
\cref{c40}\eqref{t31} and \cref{a08b}\eqref{c42} 
are well-posed.
Next, the filtered probability space $(\Omega,\mathcal{F}, \P, (\F_{t})_{t\in [0,T]} )$ in \cref{c40} above is the probability space on
which we introduce the stochastic MLP approximations which we employ to approximate the
solutions $u^d$ of the PDEs in \eqref{t30b}. The set $\Theta$ in \cref{c40} is used as an index set to introduce sufficiently many independent random variables.
The functions $\mathfrak{t}^\theta$ are independent random variables which are uniformly distributed on $[0,1]$.
The functions $W^{d,\theta}$ describe
 independent standard Brownian motions which we use as random input
sources for the MLP approximations.
The functions
$Y^{d,\theta,K,t,x}$ in \eqref{c09c}
 above describe Euler-Mayurama approximations which we use in the MLP approximations in \eqref{t27f} above as discretizations of the underlying Itô processes associated to the linear parts of the
PDEs in \eqref{t30b}.
The function $U^{d,\theta,K}_{n,m}$ in \eqref{t27f} describe the MLP approximations which we employ to approximately compute the solutions $u^d$ to the PDEs \eqref{t30b}.
Let us discuss the computational effort of the MLP approximations in \eqref{t27f}. Each $C^{d,K}_{n,m}$ in \eqref{c29} is the computational effort to compute a realization of $U^{d,\theta,K}_{n,m}(t,x,\omega)$. 
Here, we assume that the computational effort of $f$, $g^d$,
$(\mu^d,\sigma^d)$ plus the effort to simulate an arbitrary $d$-dimensional Brownian increments is bounded by $cd^c$, which is a polynomial of $d$. 
Due to \eqref{c09c} and
\eqref{t27f} the family $(C^{d,K}_{n,m})$ satisfies the recursive inequality \eqref{c29} above. 
\cref{c40} establishes that the solutions $u^d$ of the PDEs in \eqref{t30b}
can be approximated by the MLP approximations $U^{d,\theta,K}_{n,m}$
 in \eqref{t27f} with the number of involved function evaluations  and
the number of involved scalar random variables growing at most polynomially in the reciprocal
$1/\epsilon$ of the prescribed approximation accuracy $\epsilon\in (0,1)$ and at most polynomially in the PDE
dimension $d\in \N$. In other words, \cref{c40} states that MLP approximations overcome the curse of dimensionality when approximating the semilinear parabolic PDEs in \eqref{t30b}.

\subsection{Numerical example} We present a numerical example to illustrate the result of \cref{c40}.
\begin{example}\label{e02a}
Assume that
$T=1 $, 
$d=100$, and
 assume for $w\in \R$,
$x\in \R^d$ that $f(w)=\sin(w)$, 
$\mu^d(x)=\cos(\lVert x\rVert)x$,
$\sigma^d(x)=\mathrm{Id}_{d}$, and $g^d(x)={1}/(2+\frac{2}{5}\lVert x\rVert^2)$. In this example the PDE \eqref{t30b} is
\begin{align}
(\tfrac{\partial }{\partial t}u^d)(t,x)+\tfrac{1}{2}(\Delta u^d)(t,x)
+\cos (\lVert x\rVert)\langle x, (\nabla u^d)(t,x)\rangle+\sin(u^d(t,x))=0
\end{align} for $t\in [0,T]$, $x\in \R^d$.    
\end{example}
As we do not know the exact solution, we use the MLP approximations \eqref{t27f} with $n = m = 4$ on the uniform grid of mesh $T /10000$ as
reference solutions. The $L^4$-distance between the reference solution and the multilevel Picard approximation is approximated by taking averages over $5$ runs.
\begin{figure}
    \centering
    \includegraphics[width=0.5\linewidth]{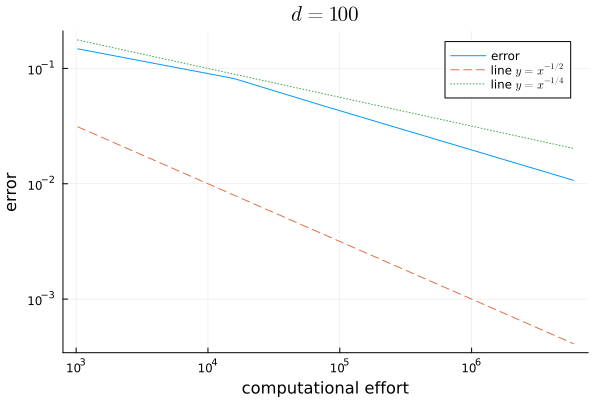}
    \caption{Numerical result for \cref{e02a}}
    \label{f02}
\end{figure}
In this example 
the polyline representing the relative $L^4$-error  tends to stay between 
the reference lines $y = x^{-1/2}$ and $y=x^{-1/4}$. This indicates a convergence rate of order $\epsilon^{-(4+\delta)}$ with respect to the $L^4$-error. The code for this example is presented in Listing~\ref{l03} in the appendix.

\subsection{A mathematical framework for DNNs}
In order to formulate our second main result, \cref{a08b},  we first need to introduce a mathematical frame work for DNNs.
\begin{setting}[A mathematical framework for DNNs]\label{m07}
Let $a\in C(\R,\R)$. Let
 $\mathbf{A}_{d}\colon \R^d\to\R^d $, $d\in \N$, satisfy for all $d\in\N$, $x=(x_1,\ldots,x_d)\in \R^d$ that
\begin{align}
\mathbf{A}_{d}(x)= \left(a(x_1),a(x_2),\ldots,a(x_d)\right).
\end{align}
Let $\mathbf{D}=\cup_{H\in \N} \N^{H+2}$.
Let
\begin{align}
\mathbf{N}= \bigcup_{H\in  \N}\bigcup_{(k_0,k_1,\ldots,k_{H+1})\in \N^{H+2}}
\left[ \prod_{n=1}^{H+1} \left(\R^{k_{n}\times k_{n-1}} \times\R^{k_{n}}\right)\right].
\end{align} Let $\mathcal{D}\colon \mathbf{N}\to\mathbf{D}$, 
$\mathcal{P}\colon \mathbf{N}\to \N$,
$
\mathcal{R}\colon \mathbf{N}\to (\cup_{k,l\in \N} C(\R^k,\R^l))$
satisfy that
for all $H\in \N$, $k_0,k_1,\ldots,k_H,k_{H+1}\in \N$,
$
\Phi = ((W_1,B_1),\ldots,(W_{H+1},B_{H+1}))\in \prod_{n=1}^{H+1} \left(\R^{k_n\times k_{n-1}} \times\R^{k_n}\right), 
$
$x_0 \in \R^{k_0},\ldots,x_{H}\in \R^{k_{H}}$ with the property that
$\forall\, n\in \N\cap [1,H]\colon x_n = \mathbf{A}_{k_n}(W_n x_{n-1}+B_n )
$ we have that
\begin{align}
\mathcal{P}(\Phi)=\sum_{n=1}^{H+1}k_n(k_{n-1}+1),
\quad 
\mathcal{D}(\Phi)= (k_0,k_1,\ldots,k_{H}, k_{H+1}),
\end{align}
$
\mathcal{R}(\Phi )\in C(\R^{k_0},\R ^ {k_{H+1}}),
$
and
\begin{align}
 (\mathcal{R}(\Phi)) (x_0) = W_{H+1}x_{H}+B_{H+1}.
\end{align}
\end{setting}
Let us comment on the mathematical objects  in  \cref{m07}.
The function $a$ is called the activation function. 
An example of $a$ is the ReLU function $\R\ni x\mapsto \max \{x,0\}\in \R$. However, in this paper we do not restrict ourselves to this function.
For all $ d\in \N $, $ \mathbf{A}_d\colon\R^d\to\R^d$  refers to the componentwise activation function. 
By $ \mathbf{N} $ we denote the set of all
parameters characterizing artificial feed-forward DNNs.
For every $H\in \N$, $k_0,k_1,\ldots,k_H,k_{H+1}\in \N$,
$
\Phi = ((W_1,B_1),\ldots,(W_{H+1},B_{H+1}))\in \prod_{n=1}^{H+1} \left(\R^{k_n\times k_{n-1}} \times\R^{k_n}\right)\subseteq \mathbf{N}
$ the natural number $H$ can be interpreted as the depth of the parameters characterizing artificial feed-forward DNN $\Phi$ and
$
(W_1,B_1),\ldots,(W_{H+1},B_{H+1})
$
can be interpreted as the parameters of $\Phi$.
By $ \mathcal{R} $   we denote the operator that maps each parameters characterizing a DNN to its corresponding function. By $ \mathcal{P} $ we denote the function that counts the number of parameters of the corresponding DNN. By $ \mathcal{D} $ we denote the function that maps the parameters characterizing a DNN to the vector of its layer dimensions.

\subsection{DNNs overcome the curse of dimensionality when approximating semilinear parabolic PDEs in $L^\mathfrak{p}$-sense}
\begin{theorem}\label{a08b}
Assume \cref{m07}. Let $\alpha\in [0,\infty)\setminus\{1\}$, $\mathfrak{a}_0,\mathfrak{a}_1\in C(\R,\R)$ satisfy for all $x\in \R$ that $\mathfrak{a}_0=\max\{x,\alpha x\}$ and
$\mathfrak{a}_1=\ln(1+e^x)$. Assume that $a\in \{\mathfrak{a}_0,\mathfrak{a}_1\}$.
Let $\beta,\mathfrak{p}\in [2,\infty) $, $c\in [1,\infty)$.
For every $d\in \N$, $\varepsilon\in (0,1)$, $v\in\R^d$
let
$\Phi_{\mu_\varepsilon^d},\Phi_{\sigma_\varepsilon^d,v}, \Phi_{g^d_\varepsilon}
\in \mathbf{N}
$,
$f\in C(\R,\R)$, $g^d,g_\varepsilon^d\in C(\R^d,\R) $,
$\mu^d,\mu^d_\varepsilon\in C(\R^d,\R^d)$,
$\sigma^d,\sigma^d_\varepsilon\in C(\R^{d\times d},\R^d)$
satisfy 
for all $v\in \R^d$
that
$\mu_\varepsilon^d=\mathcal{R}(\Phi_{\mu_\varepsilon^d}) $, 
$\sigma_\varepsilon^d (\cdot)v=\mathcal{R}(\Phi_{\sigma_\varepsilon^d,v})$,
$g^d_\varepsilon=\mathcal{R}(\Phi_{g^d_\varepsilon})$. Assume for all
 $d\in \N$, $\varepsilon\in (0,1)$,
$v\in\R^d$ that
$\mathcal{D}(\Phi_{\sigma_\varepsilon^d,v})=\mathcal{D}(\Phi_{\sigma_\varepsilon^d,0})$.
Assume for all 
$d\in \N$, $\varepsilon\in (0,1)$,
$v,w\in \R$,
$x,y\in \R^d$ that
\begin{align}
\max\{
\lVert
\mu_\varepsilon^d(x)-\mu_\varepsilon^d(y)
\rVert,
\lVert
\sigma_\varepsilon^d(x)-\sigma_\varepsilon^d(y)
\rVert
\}
\leq c\lVert x-y\rVert,\label{d04b}
\end{align}
\begin{align}
\lvert
g_\varepsilon^d(x)-g_\varepsilon^d(y)
\rvert\leq c\frac{(d^c+\lVert x\rVert)^\beta+(d^c+\lVert y\rVert)^\beta}{2\sqrt{T}}\lVert x-y\rVert,\label{d09b}
\end{align}
\begin{align}
\lvert g^d_\varepsilon(x)\rvert \leq c(d^c +\lVert x\rVert)^\beta,\quad \max\!\left \{
\lVert\mu^d_\varepsilon(0)\rVert,
\lVert\sigma^d_\varepsilon(0)\rVert,
\lvert Tf(0)\rvert, \lvert g^d_\varepsilon(0)\rvert
\right\}\leq cd^c,\label{d05b}
\end{align}
\begin{align}
\max\{
\lVert
\mu^d_\varepsilon(x)-\mu^d(x)
\rVert,
\lVert
\sigma^d_\varepsilon(x)-\sigma^d(x)
\rVert,
\lVert
g^d_\varepsilon(x)-g^d(x)
\rVert
\}\leq \varepsilon cd^c (d^c+\lVert x\rVert)^\beta,\label{d08c}
\end{align}
\begin{align}
\max\!\left\{
\mathcal{P}(\Phi_{g^d_\varepsilon}),
\mathcal{P}(\Phi_{\mu^d_\varepsilon})
,
\mathcal{P}(\Phi_{\sigma^d_\varepsilon,0})\right\}\leq cd^c\varepsilon^{-c}
.\label{c01eb}
\end{align}
Then the following items are true.
\begin{enumerate}[(i)]
\item \label{c41}For every $d\in \N$ there exists a unique at most polynomially growing viscosity solution $u^d$ of 
\begin{align}\label{c11b}
\frac{\partial u^d}{\partial t}(t,x)
+\frac{1}{2}\mathrm{trace}(\sigma^d(x)(\sigma^d(x))^\top
(\mathrm{Hess}_x u^d(t,x) ))
+\langle \mu^d(x),(\nabla_xu^d) (t,x)\rangle
+f(u^d(t,x))=0
\end{align}
with $u^d(T,x)=g^d(x)$ for $t\in (0,T)\times \R^d$.
\item \label{c42}There exists $(C_{\delta})_{\delta\in (0,1)}\subseteq(0,\infty)$, $\eta\in (0,\infty)$, $(\Psi_{d,\epsilon})_{d\in \N, \epsilon\in(0,1)}\subseteq \mathbf{N}$ such that for all 
$d\in \N$, $\epsilon\in(0,1)$ we have that 
$\mathcal{R}(\Psi_{d,\epsilon})\in C(\R^d,\R)$,  
\begin{align}\label{k22}
\mathcal{P}(\Psi_{d,\epsilon})\leq C_\delta \eta d^\eta\epsilon^{-(4+\delta)-6c},\quad \text{and}\quad 
\left(
\int_{[0,1]^d}
\left\lvert
(\mathcal{R}(\Psi_{d,\epsilon}))(x)
-
u^d(0,x)
\right\rvert^\mathfrak{p}
dx\right)^\frac{1}{\mathfrak{p}}<\epsilon.
\end{align}
\end{enumerate}
\end{theorem}

Let us make some comments on the mathematical objects  in  \cref{a08b}. First of all, in \cref{a08b} we consider different types of activation functions.
The activation function $\mathfrak{a}_\nu$ for $\nu\in \{0,1\}$ is the ReLU activation if $\nu=\alpha=0$, the leaky ReLU activation if
$\nu=0$ and $\alpha\in (0,1)$, or the softplus activation if $\nu=1$. Next,
the assumptions above \eqref{d04b} ensure that the functions
$g^d_\varepsilon,\mu^d_\varepsilon,\sigma^d_\varepsilon,
$
which approximate the terminal condition and the linear part of 
the PDE  are DNNs.
The bound $cd^c\varepsilon^{-c}$ in \eqref{c01eb}, 
which is a polynomial of $d$ and $\varepsilon^{-1}$,
ensures that the functions
$\mu^d_\varepsilon,\sigma^d_\varepsilon, g^d_\varepsilon $ 
are DNNs whose corresponding numbers of parameters grow 
without the curse of dimensionality. 
Under these assumptions  \cref{a08b} states that, roughly speaking, if DNNs can approximate the terminal condition and the linear part of the PDE in \eqref{c11b} 
without the curse of dimensionality, 
then they can also approximate its solution without the curse of dimensionality. 
More precisely, we show in \eqref{k22} that for every dimension $d\in \N$ and for every accuracy $\epsilon \in (0,1)$ the $L^\mathfrak{p}(dx)$-expression error of the unique viscosity solution of the nonlinear PDE~\eqref{c11b} is $\epsilon$ and the number of parameters of the DNNs is upper bounded polynominally in $d$ and $\epsilon^{-1}$. Therefore, the approximation rates are free from the curse of dimensionality.
We refer to \cite{CR2023,CHW2022,GHJVW2023,hutzenthaler2020proof,JSW2021,
neufeld2024rectified} 
for similar results in $L^2$-sense.

\subsection{Sketch of the proofs}
Since \eqref{t31} in \cref{c40} contains an $L^\mathfrak{p}$-estimate we first need to prove  $L^\mathfrak{p}$-estimates for MLP approximations (cf.\ \cref{k09} and \cref{c02}), which, to the best of our knowledge, still do not appear in the scientific literature for general $\mathfrak{p}\in [2,\infty)$.
The main tool to get $L^\mathfrak{p}$-estimates is the Marcinkiewicz-Zygmund inequality (see \cite[Theorem~2.1]{Rio2009}). In addition, \cref{c02} is the $L^\mathfrak{p}$-version of \cite[Proposition~4.1]{HJKN2020}. Our first main result,
\cref{c40}, is a direct consequence of \cref{k09} and \cref{c02} and its proof is presented directly after the proof of \cref{c02}. From the technical point of view, the main novelty of \cref{c40} is the sequence $(M_n)_{n\in \N}$. 
An analysis of the error  and computational complexity (see the proof of Theorem~\ref{c40}) shows that, in order to make the MLP algorithm \eqref{t27f} converge in $L^\mathfrak{p}$, we can choose $m=M_n$ such that
\begin{align}\label{a02a}
\liminf_{j\to\infty}M_j=\infty, \quad \limsup_{j\to\infty} \frac{(M_j)^{\mathfrak{p}/2}}{j}<\infty, \quad\text{and}\quad \sup_{k\in \N}\frac{M_{k+1}}{M_k}<\infty .
\end{align}In the $L^2$-case we can simply choose $M_n=n$.
In order to have \eqref{a02a} we choose $M_n=\max \{k\in \N\colon k\leq\exp (\lvert\ln(n)\rvert^{1/2})\}$ as done in, e.g., \cite[Theorem~1.1]{HJKP2021}. This choice of $(M_n)_{n\in \N}$ is convenient for the simulation since it does not depend on the exponent $\mathfrak{p}$.

\cref{a08b}  follows from \cref{a08} and Lemmas \ref{a10}, \ref{a10b}, \ref{p10}, and \ref{p10a}. We present the proof of \cref{a08b} after the proof of \cref{a08}. Let us sketch the proof of \cref{a08}.
Although the result presented in \cref{a08} is purely deterministic, we use probabilistic arguments to prove its statement. More precisely, we 
employ the theory of full history recursive MLP approximations, which are numerical approximation methods
for which we know now (cf. \cref{c40}) that   they overcome the curse of dimensionality.
We refer to \cite{NW2022} for the convergence analysis of MLP algorithms for semilinear PIDEs and  to \cite{
beck2020overcomingElliptic,
beck2020overcoming,
hutzenthaler2019multilevel,
giles2019generalised,
hutzenthaler2021multilevel,
HJK2022,
HJKN2020,
HJKNW2020,
HJvW2020,
HK2020,
HN2022a, NNW2023} 
for corresponding results proving that MLP algorithms overcome the curse of dimensionality for PDEs without any non-local/ jump term, all in $L^2$-sense.

The main strategy of the proof of \cref{a08}, roughly speaking, 
is to demonstrate that these MLP approximations can be  represented by DNNs, 
if the coefficients determining the linear part, 
the terminal condition, and the nonlinear part are corresponding DNNs 
(cf.\  \cref{k04c}). 
Such ideas have been successfully applied to prove that
DNNs overcome the curse of dimensionality in the numerical approximations of
 {semilinear} heat equations (see  \cite{hutzenthaler2020proof,AJK+2023})
as well as {semilinear} Kolmogorov
PDEs (see \cite{CHW2022}). We also refer to 
\cite{GHJVW2023,JSW2021} for results proving that
DNNs
overcome the curse of dimensionality when approximating {linear} PDEs.

More precisely, 
we represent $u^d$ as solution of 
the stochastic fixed point equation (SFPE) \eqref{c31} where the forward processes $(X^{d,\theta,t,x}_{s})$ are defined by \eqref{c32} with drift $\mu^d$ and diffusion $\sigma^d$.
We define the MLP approximations in \eqref{t27d} involving the Euler-Maruyama approximations in \eqref{c09a}. Each $U_{n,m}^{d,\theta,K,\varepsilon}$ can be considered as approximation of the solution $u^d$ to the  PDE \eqref{c11b}. In order to estimate the approximation error
$U_{n,m}^{d,\theta,K,\varepsilon}-u^d$
 we decompose
$U_{n,m}^{d,\theta,K,\varepsilon}-u^d=U_{n,m}^{d,\theta,K,\varepsilon}-u^{d,\varepsilon}+u^{d,\varepsilon}-u^d$ where $u^{d,\varepsilon}$ is defined by SFPE \eqref{c30} where the forward processes $X^{d,\theta,\varepsilon,t,x}$ here are defined by \eqref{c33} with drift $\mu^d_\varepsilon$ and diffusion $\sigma^d_\varepsilon$, which are DNNs. The error $U_{n,m}^{d,\theta,K,\varepsilon}-u^{d,\varepsilon}$ is the error bound for an MLP approximation involving Euler-Maruyama approximations and therefore can be established in \cref{c02} (see \eqref{c34}). The error $u^{d,\varepsilon}-u^d$ can be estimated, as in the $L^2$ case, by the perturbation result in \cite[Lemma~2.3]{CHW2022}. The main difficulty in the case of leaky ReLU and softplus activation is, compared to the case with ReLU, that here we have another definition of the operator $\odot$ than that in, e.g., \cite{CHW2022} (see \cref{m08,m11b}) and as a consequence we need to rebuild the whole DNN calculus.

The paper is organized as follows. In \cref{c35} we establish $L^\mathfrak{p}$-estimates for MLP approximations and prove our first main result, \cref{c40}.
In \cref{c36}
we study DNN representations for MLP approximations for PDEs of the
form \eqref{c11b}. 
In \cref{c37} we use the main representations in \cref{c36} to prove our second main result, \cref{a08b}.

\subsection{Future works}
In future work, we aim to extend our analysis to the gradient-dependent case, that is, to
situations in which the nonlinearity $f$ also depends on the gradient $\nabla u$. The
main challenge in this setting lies in the fact that the gradient of the solution to the
corresponding PDE is represented via the Bismut--Elworthy--Li formula.
Moreover, the associated stochastic fixed-point equation, which serves as the key tool
for establishing convergence of the corresponding MLP scheme, must simultaneously
capture both the solution and its gradient. Consequently, since the MLP scheme can be
interpreted as a multilevel Monte Carlo approximation of the unique solution to this
stochastic fixed-point equation, its structure in the gradient-dependent case differs
significantly from that in the gradient-independent setting (see
\cite{neufeld2024multilevel} and \cite{NW2023}).
In contrast, for the special case of the semilinear heat equation with gradient-dependent
nonlinearity, the corresponding MLP scheme becomes substantially simpler. In this
setting, it has been shown that the MLP approach can overcome the curse of
dimensionality also in the $L^{\mathfrak{p}}$-framework for $\mathfrak{p} \in [2,\infty)$
(see \cite{NGUYEN2025116834}).

\section{MLP approximations}\label{c35}

\subsection{Error bounds for abstract MLP approximations}
In this section we establish $L^\mathfrak{p}$-estimates for MLP approximations. More precisely, we extend \cite[Corollary~3.12]{HJKN2020} and \cite[Proposition~4.1]{HJKN2020} to $L^\mathfrak{p}$-estimates. First of all, we work with an abstract MLP setting, \cref{s01}, and prove 
$L^\mathfrak{p}$-error estimates, see \cref{k09}. The main difference between the general $L^\mathfrak{p}$-case and the $L^2$-case is that in the $L^\mathfrak{p}$-case we appeal to 
the Marcinkiewicz-Zygmund inequality (see \cite[Theorem~2.1]{Rio2009}). Having proven the $L^\mathfrak{p}$-error estimate we easily prove
the $L^\mathfrak{p}$-error estimate for MLP approximations involving Euler-Maruyama approximations, see \cref{c02}.
\begin{setting}\label{s01}
Let $d\in \N$, $p_v\in [1,\infty)$,
$c,T \in (0,\infty)$, 
$f \in C( \R,\R)$, 
$g\in C(\R^d,\R)$,
$V\in C([0,T]\times\R^d,[1,\infty))$,  $  \Theta = \bigcup_{ n \in \N } \Z^n$.
Let
$\mathfrak{t}^\theta\colon \Omega\to[0,1]$,
$\theta\in \Theta$, be identically distributed and independent random variables which satisfy for all
$t\in [0,1]$
that
$\P(\mathfrak{t}^0 \leq t)= t$.
Let
$(X^{\theta,s,x}_{t})_{s\in[0,T],t\in[s,T],x\in\R^d}
\colon \{(\sigma,\tau)\in [0,T]^2\colon \sigma\leq \tau \} \times\R^d\times \Omega\to\R^d
$, $\theta\in \Theta$, be measurable as well as independent and identically distributed. Assume that
$(X^{\theta,s,x}_{t})_{s\in[0,T],t\in[s,T],x\in\R^d,\theta\in \Theta} 
$ and $(\mathfrak{t}^\theta)_{\theta\in \Theta}$ are independent.
Assume for all 
$s\in [0,T]$, $t\in[s,T]$,
$x\in \R^d$,
$w_1,w_2\in \R$
that
\begin{equation}
\label{a14}
  \lvert g(x)\rvert \leq  V(T,x), \quad 
\lvert Tf(0)\rvert\leq V(s,x ), 
\end{equation}%
\begin{equation}\label{a14c}
\lvert  f(w_1)-f(w_2)\rvert\leq c \lvert w_1-w_2\rvert,
\end{equation}
\begin{equation}
 \bigl\lVert V(t,X_{t}^{0,s,x})\bigr\rVert_{p_v}
\leq V(s,x).
\label{k04b}
\end{equation}
Let
$ 
  {U}_{ n,m}^{\theta } \colon [0, T] \times \R^d \times \Omega \to \R
$,
$n,m\in\Z$, $\theta\in\Theta$, satisfy for all 
$n,m\in \N$, $\theta\in\Theta$, $t\in[0,T]$, $x\in\R^d$ that
$
{U}_{-1,m}^{\theta}(t,x)={U}_{0,m}^{\theta}(t,x)=0$ and
\begin{align} \begin{split} 
\label{t27}
&  {U}_{n,m}^{\theta}(t,x)
=  \frac{1}{m^n}\sum_{i=1}^{m^n}
g(X^{(\theta,0,-i),t,x}_T)\\
&+
\sum_{\ell=0}^{n-1}\frac{T-t}{m^{n-\ell}}
    \sum_{i=1}^{m^{n-\ell}}
     \bigl( f\circ {U}_{\ell,m}^{(\theta,\ell,i)}
-\1_{\N}(\ell)
f\circ {U}_{\ell-1,m}^{(\theta,-\ell,i)}
\bigr)\!
\left(t+(T-t)\mathfrak{t}^{(\theta,\ell,i)},X_{t+(T-t)\mathfrak{t}^{(\theta,\ell,i)}}^{(\theta,\ell,i),t,x}\right).\end{split}
\end{align}
\end{setting}

\begin{lemma}[Independence and distributional properties]
\label{ppt0}
Assume \cref{s01}. 
Then the following items are true.
\begin{enumerate}[(i)]
\item  \label{ppt1}
We have
for all $n \in \N_0$, $m \in \N$, $\theta\in\Theta $ that
$
  {U}_{ n,m}^{\theta }
$ and $f\circ {U}_{ n,m}^{\theta }$
are 
measurable.
\item  \label{ppt2}
We have\footnote{Let $( \Omega, \mathcal F, \mathbb P )$ be a probability space, let $n \in \N$, 
and let $( S_k, \mathcal{S}_k )$, $k \in \{ 1, 2, ..., n \}$, be measurable spaces.
Note that for all 
$ X_k \colon \Omega \to S_k $, $ k \in \{ 1, 2, ..., n \} $,
we have that 
$\mathfrak{S}( X_1, X_2, ..., X_n )$ is the smallest sigma-algebra 
on $\Omega$ with respect to which $X_1, X_2, ..., X_n$ are measurable.}
 for all $n \in \N_0$, $m \in \N$, $\theta \in \Theta$ that
\begin{align}\begin{split}
&\mathfrak{S}( ( U^{ \theta }_{ n, m }( t, x ) )_{t\in [0,T],x\in \R^d})
\subseteq
  \mathfrak{S}
( ( \mathfrak{t}^{(\theta,\vartheta)} )_{ \vartheta \in \Theta }, (  X_{t}^{(\theta,\vartheta),s,x})_{\vartheta\in \Theta, s\in [0,T], t\in [s,T],x\in \R^d } ).
\end{split}
\end{align}

\item  \label{ppt3}
We have 
for all $n \in \N_0$, $m \in \N$, $\theta\in\Theta $ that
$(U_{ n,m}^{\theta}(t,x))_{t\in [0,T],x\in \R^d}$, 
$(X^{\theta,s,x}_{t})_{s\in [0,T],t\in [s,T],x\in \R^d}$,  and $\mathfrak{t}^\theta$ are independent.
\item  \label{ppt4}
We have
for all $n, m \in \N_0$, $m \in \N$, $i,j,k,\ell,\nu \in \Z$, $\theta \in \Theta$ with $(i,j) \neq (k,l)$ 
that
$
 (U^{(\theta,i,j)}_{n,m}(t,x))_{t\in[0,T],x\in\R^d}
$,
$ ( 
  U^{(\theta,k,\ell)}_{n,m}(t,x))_{t\in[0,T],x\in\R^d}
$,
$ \mathfrak{t}^{(\theta,i,j)}$, and $(X^{(\theta,i,j),s,x}_{t})_{s\in[0,T],t\in [s,T],x\in\R^d}$
are independent.
\item  \label{ppt5}
We have
for all $n \in \N_0$, $m \in \N$, $t\in[0,T]$, $x\in\R^d$ that 
$  U^\theta_{n, m}(t,x)$, $\theta \in \Theta$,
are identically distributed.
\end{enumerate}
\end{lemma}
\begin{proof}[Proof of \cref{ppt0}]See
\cite[Lemma~3.2]{HJKN2020}.
\end{proof}
\begin{theorem}
[$L^p$-error estimates, $p\in [2,\infty)$, for MLP approximations]\label{k09}Assume \cref{s01}. Let $p\in [2,\infty)$, $q_1\in [1,\infty)$ satisfy that $pq_1\leq p_v$. Then the following items hold.
\begin{enumerate}[(i)]
\item\label{k09a} There exists a unique measurable $u\colon[0,T]\times \R^d\to \R$ which satisfies for all
$t\in [0,T]$, $x\in \R^d$ that
$\E[\lvert g(X^{0,t,x}_T)\rvert]
+ \int_{t}^{T} \E [\lvert f(u(s,X^{0,t,x}_s))\rvert]\,ds
+\sup_{y\in \R^d,s\in [0,T]}\frac{\lvert u(s,y)\rvert}{V(s,y)}<\infty
$ and
\begin{align}
u(t,x)= \E [g(X^{0,t,x}_T)]+\int_{t}^{T} \E[f(u(s,X^{0,t,x}_s))]\,ds.
\end{align} 
 \item \label{k09b}
We have for all 
$m,n\in \N$,
$t\in [0,T]$, $x\in \R^d$
\begin{align}
\left\lVert
{U}_{n,m}^{0}(t,x)-u(t,x)\right\rVert_p
\leq 2(p-1)^{\frac{n}{2}}e^{5cTn}e^{m^{p/2}/p}m^{-n/2}V^{q_1}(t,x)
.
\end{align}\end{enumerate}
\end{theorem}\begin{proof}
[Proof of \cref{k09}]
For every random field $H\colon [0,T]\times \R^d\times\Omega\to\R$ and every $s\in [0,T]$ let $\tnorm{H}_s\in [0,\infty]$ satisfy that
\begin{align}
\tnorm{H}_s=\sup_{t\in [s,T],x\in\R^d}\frac{\lVert H(t,x)\rVert_p}{(V(t,x))^{q_1}}.\label{k01}
\end{align}Furthermore, for every 
random variable $\mathfrak{X}\colon \Omega\to \R$ with $\E [\lvert \mathfrak{X}\rvert]<\infty$ let $\var_{p}(\mathfrak{X})\in[0,\infty]$ satisfy that
\begin{align}
\var_{p}(\mathfrak{X})=\lVert\mathfrak{X}-\E[\mathfrak{X}]\rVert_{p}^2.
\end{align}

First,
measurability and
\cite[Proposition~2.2]{HJKN2020} (applied with $d\gets d$,
$T\gets T$,
$L\gets c$,
$\mathcal{O}\gets \R^d$,
$(X^x_{t,s})_{t\in [0,T],s\in [t,T],x\in \R^d}\gets 
(X^{0,t,x}_{s})_{t\in [0,T],s\in [t,T],x\in \R^d}
$,
$f\gets ([0,T]\times \R^d\times \R\ni (t,x,w)\mapsto f(w)\in \R)$, $g\gets g$, $V\gets V$
in the notation of 
\cite[Proposition~2.2]{HJKN2020})
show that there exists a unique measurable $u\colon[0,T]\times \R^d\to \R$ which satisfies for all
$t\in [0,T]$, $x\in \R^d$ 
that 
$\E[\lvert g(X^{0,t,x}_T)\rvert]
+ \int_{t}^{T} \E [\lvert f(u(s,X^{0,t,x}_s))\rvert]\,ds
+\sup_{y\in \R^d,s\in [0,T]}\frac{\lvert u(s,y)\rvert}{V(s,y)}<\infty
$ and 
\begin{align}
u(t,x)= \E [g(X^{0,t,x}_T)]+\int_{t}^{T} \E[f(u(s,X^{0,t,x}_s))]\,ds
\end{align} 
and we have for all $t\in [0,T]$, $x\in \R^d$ that  $\frac{\lvert u(t,x)\rvert}{V(t,x)}\leq 2e^{c(T-t)}$.
This, the fact that $q_1\geq 1$, and \eqref{k01} imply for all $s\in [0,T]$ that
\begin{align}
\tnorm{u}_s\leq 2e^{cT}.\label{k04}
\end{align}
This proves \eqref{k09a}.

Next, Jensen's inequality, the fact that $p\leq p_v$, \eqref{k04b}, and the fact that $V\leq V^{q_1}$ show for all $t\in [0,T]$, $x\in \R^d$ that
\begin{align}
\lVert g(X^{0,t,x}_T) \rVert_{p}\leq 
\lVert V(T,X^{0,t,x}_T) \rVert_{p}
\leq \lVert V(T,X^{0,t,x}_T) \rVert_{p_v}
\leq V(t,x)\leq V^{q_1}(t,x).\label{k03} 
\end{align}
Next, the disintegration theorem, the measurability and independence properties, the fact that
$pq_1\leq p_v$,
and Jensen's inequality
prove for all $t\in [0,T]$, $\ell,\nu \in \N_0$, $x\in \R^d$,
$H\in \mathrm{span}_\R (\{f\circ U^{\nu}_{\ell,m},f\circ u\})$ that
\begin{align} \begin{split} 
\left\lVert
(T-t)H(t+(T-t) \mathfrak{t}^0, X^{0,t,x}_{t+(T-t)\mathfrak{t}^0} )
\right\rVert_p&=
(T-t)\left\lVert\left\lVert
\left\lVert
H(r,y)
\right\rVert_p\big|_{y=X^{0,t,x}_r}\right\rVert_p\Bigr|_{r=t+(T-t)\mathfrak{t}^0}
\right\rVert_p\\
&\leq 
(T-t) \left\lVert\left[\tnorm{H}_{r}\left\lVert
V^{q_1}(r,X_r^{0,t,x})
\right\rVert_p\right]\Bigr|_{r=t+(T-t)\mathfrak{t}^0}\right\rVert_p\\
&\leq (T-t) \left\lVert\tnorm{H}_{t+(T-t)\mathfrak{t}^0}\right\rVert_p V^{q_1}(t,x)\label{k02}.\end{split}
\end{align}
Moreover, \eqref{k01} and \eqref{a14c} show for all $t\in [0,T]$ and all random fields $H,K\colon [0,T]\times\R^d\times\Omega\to \R$ that
$\tnorm{(f\circ H)-(f\circ K)}_t\leq c\tnorm{H-K}_t$. This,  
\eqref{k02}, and the independence and distributional properties imply for all $t\in [0,T]$, $\nu,\ell\in \N_0$, $m,n\in \N$, $x\in \R^d$ that
\begin{align} \begin{split} 
&\left\lVert
(T-t)((f\circ U^\nu_{\ell,m})-(f\circ u))
(t+(T-t)\mathfrak{t}^0, X^{0,t,x}_{t+(T-t)\mathfrak{t}^0})
\right\rVert_p
\\
&
\leq (T-t)\left\lVert\tnorm{(f\circ U^{\nu}_{\ell,m}) -(f\circ u)}_{t+(T-t)\mathfrak{t}^0}\right\rVert_pV^{q_1}(t,x)\\
&\leq (T-t)c\left\lVert\tnorm{U^0_{\ell,m}-u}_{t+(T-t)\mathfrak{t}^0}\right\rVert_pV^{q_1}(t,x).
\end{split}\label{k06}
\end{align}
This and the triangle inequality
show for all $t\in [0,T]$, $x\in \R^d$, $m,\ell\in \N$ that
\begin{align} \begin{split} 
&
\left\lVert
(T-t)\left[
((f\circ U^0_{\ell,m})-(f\circ U^1_{\ell-1,m}))
(t+(T-t)\mathfrak{t}^0, X^{0,t,x}_{t+(T-t)\mathfrak{t}^0})
\right]
\right\rVert_p\\
&
\leq 
\left\lVert
(T-t)\left[
((f\circ U^0_{\ell,m})-(f\circ u))
(t+(T-t)\mathfrak{t}^0, X^{0,t,x}_{t+(T-t)\mathfrak{t}^0})
\right]
\right\rVert_p\\
&\quad 
+
\left\lVert
(T-t)\left[
((f\circ U^1_{\ell-1,m})-(f\circ u))
(t+(T-t)\mathfrak{t}^0, X^{0,t,x}_{t+(T-t)\mathfrak{t}^0})
\right]
\right\rVert_p\\
&\leq 
\sum_{j=\ell-1}^\ell \left[
(T-t)c\left\lVert \tnorm{U_{j,m}^0-u}_{t+(T-t)\mathfrak{t}^0}\right\rVert_p
\right]
V^{q_1}(t,x).
\end{split}\label{k05}
\end{align}
This, \eqref{t27}, the triangle inequality,
the fact that $\forall\,m\in \N\colon U^{0}_{0,m}=0$, the independence and distributional properties, 
 \eqref{k03}, \eqref{k04}, \eqref{k05},
and induction prove for all $n\in \N_0$, $m\in \N$, $x\in \R^d$, $t\in [0,T]$, $\theta\in \Theta$ that
\begin{align}
\tnorm{U^\theta_{n,m}}_{t}+
\left\lVert
(T-t)(f\circ U^\theta_{n,m})(t+(T-t)\mathfrak{t}^\theta, X^{\theta,t,x}_{t+(T-t)\mathfrak{t}^\theta})
\right\rVert_p<\infty.
\end{align}
Next, linearity, the independence and distributional properties, and a telescoping sum argument prove for all
$n,m\in \N$, $t\in [0,T]$, $x\in \R^d$ that
\begin{align} \begin{split} 
\label{t27b}
& \E [ {U}_{n,m}^{0}(t,x)]
=  \frac{1}{m^n}\sum_{i=1}^{m^n}
\E [g(X^{(0,0,-i),t,x}_T)]\\
&+
\sum_{\ell=0}^{n-1}\frac{T-t}{m^{n-\ell}}
    \sum_{i=1}^{m^{n-\ell}}\E
\biggl[
     \bigl( f\circ {U}_{\ell,m}^{(0,\ell,i)}
-\1_{\N}(\ell)
f\circ {U}_{\ell-1,m}^{(0,-\ell,i)}
\bigr)\!
\left(t+(T-t)\mathfrak{t}^{(0,\ell,i)},X_{t+(T-t)\mathfrak{t}^{(0,\ell,i)}}^{(0,\ell,i),t,x}\right)\biggr]\\
&= \E [g(X_T^{0,t,x})]+\sum_{\ell=0}^{n-1}(T-t)\Biggl[
\E\!\left[ (f\circ U^{0}_{\ell,m})(t+(T-t)\mathfrak{t}^0, X^{0,t,x}_{t+(T-t)\mathfrak{t}^0})\right]\\
&\qquad\qquad\qquad\qquad\qquad\qquad-\E\!\left[
(f\circ U^{0}_{\ell-1,m})(t+(T-t)\mathfrak{t}^0, X^{0,t,x}_{t+(T-t)\mathfrak{t}^0})\right]\Biggr]\\
&= \E [g(X_T^{0,t,x})]+(T-t)
\E\!\left[ (f\circ U^{0}_{n-1,m})(t+(T-t)\mathfrak{t}^0, X^{0,t,x}_{t+(T-t)\mathfrak{t}^0})\right]
.\end{split}
\end{align}
Moreover, the disintegration theorem and the independence and distributional properties show for all $t\in [0,T]$, $x\in \R^d$ that
\begin{align}
u(t,x)= \E\!\left [g(X^{0,t,x}_T)\right]+ (T-t)\E \!\left[(f\circ u) (t+(T-t)\mathfrak{t}^0, X^{0,t,x}_{t+(T-t)\mathfrak{t}^0})\right].
\end{align}
This, the triangle inequality, \eqref{t27b}, Jensen's inequality, and \eqref{k06} prove for all 
$n,m\in \N$, $t\in [0,T]$, $x\in \R^d$ 
that
\begin{align}
\frac{\left\lvert\E [U^0_{n,m}(t,x) ]-u(t,x)\right\rvert}{V^{q_1}(t,x)}\leq (T-t)c\left\lVert
\tnorm{U^0_{n-1,m}-u}_{t+(T-t)\mathfrak{t}^0}
\right\rVert_p.\label{k07}
\end{align}
Moreover, 
the Marcinkiewicz-Zygmund inequality (see \cite[Theorem~2.1]{Rio2009}), the fact that $p\in [2,\infty)$, the triangle inequality, and  Jensen's inequality show  for all $n\in\N$ and all identically distributed and independent random variables $\mathfrak{X}_k$, $k\in [1,n]\cap\Z$, with $\E [\lvert\mathfrak{X}_1\rvert]<\infty$ that
\begin{align} \begin{split} 
\left(\var_{p}\left[\frac{1}{n}\sum_{k=1}^{n}\mathfrak{X}_k\right]\right)^{1/2}&=\frac{1}{n}
\left\lVert \sum_{k=1}^{n}(\mathfrak{X}_k-\E[\mathfrak{X}_k])\right\rVert_{p}\leq \frac{\sqrt{p-1}}{n}
\left(\sum_{k=1}^{n}\left\lVert \mathfrak{X}_k-\E[\mathfrak{X}_k]\right\rVert_p^2\right)^{\frac{1}{2}}\\&\leq  \frac{2\sqrt{p-1}\lVert \mathfrak{X}_1\rVert_{p}}{\sqrt{n}}.\end{split}
\end{align} This, \eqref{t27},
the triangle inequality, the independence and distributional properties, \eqref{k03},  and \eqref{k05} show for all
$n,m\in \N$, $t\in [0,T]$, $x\in \R^d$ that
\begin{align}\label{k08}
 \begin{split} 
&
\frac{\left\lVert
U^0_{n,m}(t,x)-\E[U^0_{n,m}(t,x)]
\right\rVert_p}{V^{q_1}(t,x)}= 
\frac{\left(\var_p( U^0_{n,m}(t,x) )\right)^{\frac{1}{2}}}{V^{q_1}(t,x)}\\
&\leq \frac{\left(
\var_p \!\left[\frac{1}{m^n}\sum_{i=1}^{m^n}
g(X^{(0,0,-i),t,x}_T)\right]\right)^\frac{1}{2}}{V^{q_1}(t,x)}
\\
&\quad +
\sum_{\ell=0}^{n-1}\frac{\left(\var_p\!\left[\frac{T-t}{m^{n-\ell}}
    \sum_{i=1}^{m^{n-\ell}}
     \bigl( f\circ {U}_{\ell,m}^{(0,\ell,i)}
-\1_{\N}(\ell)
f\circ {U}_{\ell-1,m}^{(0,-\ell,i)}
\bigr)\!
\left(t+(T-t)\mathfrak{t}^{(0,\ell,i)},X_{t+(T-t)\mathfrak{t}^{(0,\ell,i)}}^{(0,\ell,i),t,x}\right)
\right]\right)^\frac{1}{2}}{V^{q_1}(t,x)}\\
&\leq \frac{\frac{2\sqrt{p-1}\left\lVert g(X^{0,t,x}_T)\right\rVert_p}{\sqrt{m^n}}
+\sum_{\ell=1}^{n-1}
\frac{2\sqrt{p-1}\left\lVert(T-t)
 \bigl( f\circ {U}_{\ell,m}^{0}
-\1_{\N}(\ell)
f\circ {U}_{\ell-1,m}^{1}
\bigr)\!
\left(t+(T-t)\mathfrak{t}^{0},X_{t+(T-t)\mathfrak{t}^{0}}^{0,t,x}\right)
\right\rVert_p}{\sqrt{m^{n-\ell}}}
}{V^{q_1}(t,x)}\\
&
\leq \frac{2\sqrt{p-1}}{\sqrt{m^n}}
+\sum_{\ell=1}^{n-1}
\sum_{j=\ell-1}^\ell\frac{2\sqrt{p-1}}{\sqrt{m^{n-\ell}}} \left[
(T-t)c\left\lVert \tnorm{U_{j,m}^0-u}_{t+(T-t)\mathfrak{t}^0}\right\rVert_p
\right].
\end{split}
\end{align}
In addition,
the fact that $\mathfrak{t}^0$ is uniformly distributed on $[0,1]$ and  the substitution rule imply for all
$s\in[0,T]$, $t\in [0,T]$, and all measurable $h\colon [0,T]\to \R$ that
\begin{align}\label{b04} \begin{split} 
&(T-t)\left\lVert h(t+(T-t)\mathfrak{t}^0)\right\rVert_{p}
= (T-t)^{1-\frac{1}{p}}\left[\int_{0}^{1}(T-t)
\lvert h(t+(T-t)\lambda)\rvert^{p}\,d\lambda\right]^{\frac{1}{p}}  \\
&= 
 (T-t)^{1-\frac{1}{p}}
\left[\int_{t}^{T}
\lvert h(\zeta)\rvert^{p}\,d\zeta\right]^{\frac{1}{p}}\leq 
 (T-s)^{1-\frac{1}{p}}
\left[\int_{s}^{T}
\lvert h(\zeta)\rvert^{p}\,d\zeta\right]^{\frac{1}{p}} 
.\end{split}
\end{align}
This, \eqref{k01}, the triangle inequality, \eqref{k07},
\eqref{k08}, and
the fact that $\forall\,n,m\in \N$, $a_0,a_1,\ldots,a_{n-1}\in[0,\infty)\colon 
(
\sum_{\ell=1}^{n-1}\sum_{j=\ell-1}^\ell\frac{ a_j}{\sqrt{ m^{n-\ell}}}
)
+a_{n-1}
\leq \sum_{\ell=0}^{n-1}\frac{2a_\ell}{\sqrt{ m^{n-\ell-1}}}
$ prove for all $n,m\in \N$, $s\in [0,T]$ that
\begin{align} \begin{split} 
\tnorm{{U}_{n,m}^{0}-u}_{s}
&=
\sup_{ t\in[s,T]}\frac{
\left\lVert {U}_{n,m}^{0}(t,x)-u(t,x)\right\rVert_{p}}{V^{q_1}(t,x)} 
\\
&\leq \sup_{ t\in[s,T]}
\frac{\left\lVert 
{U}_{n,m}^{0}(t,x)-
\E [{U}_{n,m}^{0}(t,x)]
\right\rVert_{p}+
\bigl\lvert\E \bigl[  {U}_{n,m}^{0}(t,x)\bigr]-u(t,x)\bigr\rvert}{V^{q_1}(t,x)}
 \\
&\leq \sup_{t\in [s,T]}\Biggl[
\frac{2\sqrt{p-1}}{\sqrt{m^n}}+
\sum_{\ell=1}^{n-1}\sum_{j=\ell-1}^{\ell}\left[
\frac{2\sqrt{p-1}}{\sqrt{m^{n-\ell}}}
(T-t) c\left\lVert \tnorm{ {U}_{j,m}^{0}- u}_{t+(T-t)\mathfrak{t}^{0}}\right\rVert_{p}
\right] 
\\
&\qquad\qquad\qquad+
 (T-t)c\left\lVert \tnorm{ {U}_{n-1,m}^{0}- u}_{t+(T-t)\mathfrak{t}^{0}}\right\rVert_{p}
\Biggr] \\
&\leq 
\sup_{t\in [s,T]}\left[
\frac{2\sqrt{p-1}}{\sqrt{m^n}}+
\sum_{\ell=0}^{n-1}\left[
\frac{4\sqrt{p-1}}{\sqrt{m^{n-\ell-1}}}
(T-t) c\left\lVert \tnorm{ {U}_{\ell,m}^{0}- u}_{t+(T-t)\mathfrak{t}^{0}}\right\rVert_{p} 
\right]\right] \\
&\leq 
\frac{2\sqrt{p-1}}{\sqrt{m^n}}+
\sum_{\ell=0}^{n-1}\left[
\frac{4\sqrt{p-1}}{\sqrt{m^{n-\ell-1}}}
(T-s)^{1-\frac{1}{p}} c\left[\int_{s}^{T} \tnorm{ {U}_{\ell,m}^{0}- u}_{\zeta}^{p}\, d\zeta\right]^{\frac{1}{p}} 
\right].\end{split}
\end{align}
Next, 
\cite[Lemma~3.11]{HJKN2020} (applied
for every $s\in[0,T] $, $n,m\in\N $ with
$M\gets m$,
$N\gets n$, $\tau\gets s$, 
$a\gets 2\sqrt{p-1}$,
$b\gets 4(T-s)^{1-\frac{1}{p}} c\sqrt{p-1}$,  
$(f_j)_{j\in\N_0}\gets 
([s,T]\ni t\mapsto \tnorm{{U}_{j,m}^{0}-u}_{t}\in[0,\infty])_{j\in\N_0}
$ 
in the notation of \cite[Lemma~3.11]{HJKN2020}), \eqref{k04}, and
the fact that $\forall\,m\in \N\colon U_{0,m}^0=0$
prove for all $m,n\in\N$, $s\in[0,T]$ that
\begin{align} \begin{split} 
\tnorm{{U}_{n,m}^{0}-u}_{s} &\leq  
\left(
2\sqrt{p-1}+
4(T-s)^{1-\frac{1}{p}} c\sqrt{p-1}\cdot (T-s)^{\frac{1}{p}}\cdot \sup_{t\in [s,T]}\tnorm{u}_{t}
\right)  \\&\qquad\qquad\qquad\qquad\cdot 
e^{m^{p/2}/p}m^{-n/2}
 \left(1+4(T-s)^{1-\frac{1}{p}} c\sqrt{p-1}\cdot (T-s)^{\frac{1}{p}}\right)^{n-1}  \\
&\leq \sqrt{p-1}\left(2+4cT\cdot 2e^{cT}\right)e^{m^{p/2}/p}m^{-n/2} \left(\sqrt{p-1}(1+4cT)\right)^{n-1}  \\
&\leq 2(p-1)^{\frac{n}{2}}e^{cT}(1+4cT)e^{m^{p/2}/p}m^{-n/2} (1+4cT)^{n-1}\\
&\leq 2(p-1)^{\frac{n}{2}}e^{5cTn}e^{m^{p/2}/p}m^{-n/2}.\end{split}
\end{align}
This and \eqref{k01} imply
for all 
$m,n\in \N$,
$t\in [0,T]$, $x\in \R^d$ that
\begin{align}
\left\lVert
{U}_{n,m}^{0}(t,x)-u(t,x)\right\rVert_p
\leq 2(p-1)^{\frac{n}{2}}e^{5cTn}e^{m^{p/2}/p}m^{-n/2}V^{q_1}(t,x)
.\end{align}
This completes the proof of \cref{k09}.
\end{proof}

\subsection{Error bounds for MLP approximations involving Euler-Maruyama approximations}
\cref{c02} below extends \cite[Proposition~4.1]{HJKN2020} to an $L^\mathfrak{p}$-estimate, $\mathfrak{p}\in [2,\infty)$. Its proof can be easily adapted from that of 
\cite[Proposition~4.1]{HJKN2020}. However, we present it here for convenience of the reader.
\begin{lemma}\label{c02}
Let $d,K\in\N$,
$T\in (0,\infty)$, $\mathfrak{p}\in [2,\infty)$,
  $\beta,b,c\in [1,\infty)$,  
$p\in[\mathfrak{p}\beta,\infty)$, 
$\varphi\in C^2(\R^d, [1,\infty))$,
$g\in C( \R^d,\R)$, $f\in C(\R,\R)$,
$\mu\in C( \R^{d} ,\R^{d})$, $\sigma\in C( \R^{d},\R^{d\times d})$. 
Assume for all $x,y\in \R^d$, $z\in \R^d\setminus\{0\}$, $t\in [0,T]$, $v,w\in \R$ that
\begin{align}
\max\!\left\{
\frac{\lvert(\varphi'(x))(z)\rvert}{(\varphi(x))^\frac{p-1}{p}\lVert z\rVert},
\frac{(\varphi''(x))(z,z)}{(\varphi(x))^\frac{p-2}{p}\lVert z\rVert^2},
\frac{c\lVert x\rVert+\lVert\mu(0)\rVert}{(\varphi(x))^\frac{1}{p}},
\frac{c\lVert x\rVert+\lVert\sigma(0)\rVert}{(\varphi(x))^\frac{1}{p}}
\right\}\leq c,\label{c08}
\end{align}
\begin{align}
\max \{\lvert Tf(0)\rvert, \lvert g(x)\rvert\}\leq b (\varphi(x))^\frac{\beta}{p},\label{v04b}
\end{align}
\begin{align}
\lvert
g(x)-g(y)\rvert\leq b\frac{(\varphi(x)+\varphi(y))^\frac{\beta}{p}}{\sqrt{T}}
\lVert x-y\rVert
,\quad \lvert f(v)-f(w)\rvert\leq c\lvert v-w\rvert,\label{v07}
\end{align}
\begin{align}\label{c03}
\max\{ \lVert \mu(x)-\mu(y)\rVert,
\lVert \sigma(x)-\sigma(y)\rVert\}\leq c\lVert x-y\rVert.
\end{align}
Let $(\Omega,\mathcal{F},\P, (\F_t)_{t\in[0,T]})$ be a filtered probability space which satisfies the usual conditions.
Let 
$  \Theta = \bigcup_{ n \in \N }\! \Z^n$.
Let $\mathfrak{t}^\theta\colon \Omega\to[0,1]$, $\theta\in \Theta$, be identically distributed and independent random variables. Assume for all $t\in(0,1)$ that $\P(\mathfrak{t}^0\leq t)=t$. Let $W^\theta\colon [0,T]\times\Omega \to \R^{d}$, $\theta\in\Theta$, be independent standard $(\F_{t})_{t\in[0,T]}$-Brownian motions. 
Assume that
$(\mathfrak{t}^\theta)_{\theta\in\Theta}$ and
$(W^\theta)_{\theta\in\Theta}$ are independent. Let
$\rdown{\cdot}_K\colon \R\to \R$ satisfy for all $t\in \R$
that $\rdown{t}_K=\max ( \{0,\frac{T}{K},\ldots,\frac{(K-1)T}{T},T\}\cap ((-\infty,t)\cup\{0\}) )$. For every $\theta\in \Theta$, $t\in[0,T]$, $x\in \R^d$
let $Y^{\theta,t,x}= (Y^{\theta,t,x}_s)_{s\in [t,T]}\colon [t,T]\times \Omega\to \R^d$ satisfy for all $s\in [t,T]$ that
$Y^{\theta,t,x}_t=x$ and
\begin{align}
Y^{\theta,t,x}_s=Y^{\theta,t,x}_{\max\{t,\rdown{s}_K\}}
+\mu(Y^{\theta,t,x}_{\max\{t,\rdown{s}_K\}})
(s-\max \{t,\rdown{s}_K\})
+
\sigma(Y^{\theta,t,x}_{\max\{t,\rdown{s}_K\}})
(W^\theta_s-W^\theta_{\max \{t,\rdown{s}_K\}}).\label{c09}
\end{align}
Let $U^\theta_{n,m}\colon [0,T]\times \R^d\times \Omega\to \R$, $n\in \Z$, $m\in \N$, $\theta\in \Theta$, satisfy for all
$\theta\in \Theta$, $m\in \N$, $n\in \N_0$, $t\in [0,T]$, $x\in \R^d$ that
$
{U}_{-1,m}^{\theta}(t,x)={U}_{0,m}^{\theta}(t,x)=0$ and
\begin{align} \begin{split} 
\label{t27c}
&  {U}_{n,m}^{\theta}(t,x)
=  \frac{1}{m^n}\sum_{i=1}^{m^n}
g(Y^{(\theta,0,-i),t,x}_T)\\
&+
\sum_{\ell=0}^{n-1}\frac{T-t}{m^{n-\ell}}
    \sum_{i=1}^{m^{n-\ell}}
     \bigl( f\circ {U}_{\ell,m}^{(\theta,\ell,i)}
-\1_{\N}(\ell)
f\circ {U}_{\ell-1,m}^{(\theta,-\ell,i)}
\bigr)\!
\left(t+(T-t)\mathfrak{t}^{(\theta,\ell,i)},Y_{t+(T-t)\mathfrak{t}^{(\theta,\ell,i)}}^{(\theta,\ell,i),t,x}\right).\end{split}
\end{align}
Then the following items are true.
\begin{enumerate}[(i)]
\item \label{c04}For every $t\in [0,T]$, $\theta\in \Theta$ there exists an up to indistinguishability unique continuous
random field
$X^{\theta,t,\cdot }=(X^{\theta,t,x}_s)_{s\in [t,T],x\in \R^d}
\colon [t,T]\times \R^d\times \Omega\to \R^d
$ which satisfies  for all $x\in \R^d$
that 
$
(X^{\theta,t,x}_s)_{s\in [t,T]}
$
is $(\F_s)_{s\in [t,T]}$-adapted and which satisfies  for all $s\in [t,T]$, $x\in \R^d$  that $\P$-a.s.\
\begin{align}
X^{\theta,t,x}_s=x+\int_{t}^{s}\mu (X^{\theta,t,x}_r)\,dr
+\int_t^s\sigma(X^{\theta,t,x}_r)\,dW^\theta_r.
\end{align}
\item\label{c05} For all $\theta\in \Theta$, $t\in [0,T]$, $s\in [t,T]$, $r\in [s,T]$, $x\in \R^d$
that
$\P(X^{\theta,s,X^{\theta,t,x}_s}_r= X^{\theta,t,x}_r)=1.$
\item \label{k21}There exists a unique
measurable
$u\colon [0,T]\times\R^d\to\R$ which satisfies
for all $t\in[0,T]$, $x\in \R^d$ that
$
\bigl(\sup_{s\in[0,T],y\in\R^d} [{\lvert u(s,y)\rvert}{(\varphi(y))^{-\beta /p}}]\bigr)+\int_{t}^{T}\E\bigl[\lvert f(u(s,X_{s}^{0,t,x}))\rvert \bigr]\,ds
+
\E\bigl[\lvert g(X^{0,t,x}_{T})\rvert\bigr]  < \infty$
and\begin{align}
u(t,x)=\E\!\left[g(X^{0,t,x}_{T})\right]+\int_{t}^{T}\E\!\left[f(u(s,X_{s}^{0,t,x}))\right]ds .
\end{align}
\item\label{c06} For all $t\in [0,T]$, $x\in \R^d$, $n\in \N_0$, $m\in \N$ we have that $U^\theta_{n,m}(t,x)$ is measurable.
\item\label{c07} For all
$t\in [0,T]$, $x\in \R^d$, $m,n\in \N$ we have that
\begin{align} \begin{split} 
&\left\lVert
{U}_{n,m}^{0}(t,x)-u(t,x)\right\rVert_\mathfrak{p}\leq  12bc^2e^{9c^3T}
(\varphi(x))^{\frac{\beta+1}{p}}\left[
2\mathfrak{p}^{\frac{n}{2}}e^{5cTn}e^{m^{\mathfrak{p}/2}/\mathfrak{p}}m^{-n/2}+\frac{1}{\sqrt{K}}
\right].
\end{split}
\end{align}

\end{enumerate}
\end{lemma}
\begin{proof}[Proof of \cref{c02}]
Observe that \eqref{c03} prove \eqref{c04} and \eqref{c05}. For the rest of the proof let $\Delta=\{(t,s)\in [0,T]^2\colon t\leq s \}$ and 
$\mathfrak{X}^k=(\mathfrak{X}^{k,t,x}_s)_{t\in [0,T],s\in [t,T],x\in \R^d}\colon \Delta\times \R^d\times \Omega\to \R^d$
satisfy for all $t\in [0,T]$, $s\in [t,T]$, $x\in \R^d$ that $ \mathfrak{X}^{0,t,x}_s=X^{0,t,x}_s$ and 
$ \mathfrak{X}^{1,t,x}_s=Y^{0,t,x}_s$. 
For every $x\in \R^d$ let $\mathfrak{Y}^x=(\mathfrak{Y}^x_t)_{t\in [0,T]}\colon [0,T]\times \Omega\to \R^d$ satisfy for all $t\in [0,T]$ that
$\mathfrak{Y}^x_t=x+\mu(x)t+\sigma(x)W_t$. 
For every $x\in \R^d$, $n \in \N$ let $\tau^x_n\colon \Omega\to [0,T]$ satisfy that 
$\tau^x_n=\inf(\{T\}\cup \{t\in [0,T]\colon [\sup_{s\in[0,t]}\varphi(\mathfrak{Y}^x_s)] + \int_0^t\sum_{i=1}^d\lvert (\varphi' (\mathfrak{Y}^x_s))(\sigma_i(x)) \rvert^2\,ds\geq n \})$.
Next, the triangle inequality, \eqref{c03}, and \eqref{c08} prove for all $x\in \R^d$ that
\begin{align} \begin{split} 
\max \{\lVert\mu (x)\rVert, \lVert\sigma (x)\rVert\}
&\leq \max \{\lVert\mu (x)-\mu(0)\rVert+\lVert\mu(0)\rVert, \lVert\sigma (x)-\sigma(0)\rVert+\lVert\sigma(0)\rVert\}\\
&\leq \max \{c\lVert x\rVert+\lVert\mu(0)\rVert,
c\lVert x\rVert+\lVert\sigma(0)\rVert
\}\leq c(\varphi(x))^{\frac{1}{p}}.\end{split}\label{c10}
\end{align}
This, \eqref{c08}, and the fact that
$\forall\,a,b\in [0,\infty),\lambda\in (0,1)\colon a^\lambda b^{1-\lambda}\leq \lambda a+(1-\lambda)b$ imply  for all $x,y\in\R^d$  that 
\begin{align}\begin{split}
&\left\lvert( \varphi'(y))(\mu(x))\right\rvert+\frac{1}{2}
\left\lvert\sum_{k=1}^{d}
(\varphi''(y))(\sigma_k(x),\sigma_k(x))\right\rvert\\
&\leq
c(\varphi(y))^{1-\frac{1}{p}}
\lVert\mu(x)\rVert+\frac{c}{2} 
(\varphi(y))^{1-\frac{2}{p}}\sum_{k=1}^{d}
\lVert\sigma_k(x)\rVert^2\\&=
c(\varphi(y))^{1-\frac{1}{p}}
\lVert\mu(x)\rVert+\frac{c}{2} 
(\varphi(y))^{1-\frac{2}{p}}\lVert\sigma(x)\rVert^2
\\
&
\leq  c(\varphi(y))^{1-\frac{1}{p}}
c(\varphi(x))^{\frac{1}{p}}+\frac{c}{2}(\varphi(y))^{1-\frac{2}{p}}c^2(\varphi(x))^\frac{2}{p}
\\
&
\leq
c^2\left[\left(1-\frac{1}{p}\right)\varphi(y)+\frac{1}{p}\varphi(x)\right]+ \frac{c^3}{2}\left[\left(1-\frac{2}{p}\right)\varphi(y)+\frac{2}{p}\varphi(x)\right]\\
&\leq \left[c^3\left(1-\frac{1}{p}\right)+\frac{c^3}{2}\left(1-\frac{2}{p}\right)\right]\varphi(y)
+
\left[\frac{c^3}{p}+\frac{2c^3}{2p}
\right]\varphi(x)\\
&
= \left(\frac{3c^3}{2}-\frac{2c^3}{p}\right)\varphi(y)+\frac{2c^3}{p}\varphi(x).
\end{split}\label{r04}\end{align}
Combining this and, e.g., \cite[Lemma~2.2]{CoxHutzenthalerJentzen2014} (applied
for every $t\in[0,T)$,
$s\in[t,T]$,  $x\in\R^d$, $\theta\in\Theta$
with
$T\gets T-t$,
$O\gets\R^d$,
$V\gets( [0,T-t]\times\R^d\ni(s,x)\mapsto\varphi(x)\in[0,\infty))$, $\alpha\gets ([0,T-t]\ni s\mapsto 2c^3 \in [0,\infty) )$, 
$\tau\gets s-t$,
$X\gets (X^{\theta,t,x}_{t+r})_{r\in[0,T-t]}$  in the notation of \cite[Lemma~2.2]{CoxHutzenthalerJentzen2014}) demonstrates  for all
$\theta\in\Theta$, $x\in\R^d$, $t\in[0,T]$, $s\in [t,T]$ 
  that 
\begin{align}\label{r05}
\E \bigl[\varphi(X_{s}^{\theta,t,x})\bigr]\leq e^{2c^3(s-t)}\varphi(x).
\end{align}
It\^o's formula, \eqref{r04}, and the fact that $\varphi\ge 1$ imply  for all $x\in\R^d$, $t\in[0,T]$  that
\begin{equation}
\begin{split}
&\E[\varphi(\mathfrak{Y}^x_{\min\{\tau_n^x,t\}})]\\
&=
\varphi(x)+\E\!\left[\int_0^{\min\{\tau_n^x,t\}}
( \varphi'(\mathfrak{Y}^x_s))(\mu(x))+\frac{1}{2}\sum_{k=1}^{m}
(\varphi''(\mathfrak{Y}^x_s))(\sigma_k(x),\sigma_k(x))\,ds\right]\\
&\le 
\varphi(x)+\E\!\left[\int_0^{\min\{\tau_n^x,t\}}
\left(\frac{3c^3}{2}-\frac{2c^3}{p}\right)\varphi(\mathfrak{Y}_s^x)+\frac{2c^3}{p}\varphi(x)\,ds\right]\\
&\le 
\varphi(x)\left(1+\frac{2c^3t}{p}\right)+\left(\frac{3c^3}{2}-\frac{2c^3}{p}\right)\E\!\left[\int_0^{t}
 \varphi(\mathfrak{Y}_s^x)\1_{[0,\tau_n^x]}(s) \,ds\right]\\
&\le 
\varphi(x)\left(1+\frac{2c^3t}{p}\right)+\left(\frac{3c^3}{2}-\frac{2c^3}{p}\right)\int_0^{t}
\E[\varphi(\mathfrak{Y}_{\min\{\tau_n^x,s\}}^x)]\,ds.
\end{split}
\end{equation}
Gronwall's inequality and
the fact that for all $a\in \R$ it holds that $ 1+ a\leq e^{a}$ therefore assure  for all $x\in\R^d$, $t\in[0,T]$ that 
\begin{align}
\E[\varphi(\mathfrak{Y}^x_{\min\{\tau_n^x,t\}})]
\leq \exp\!\left( \left[ \frac{ 3c^3}{ 2 } - \frac{ 2c^3}{ p } \right] t \right) 
\left[ 1 + \frac{ 2c^3t }{p } \right] \varphi(x)
\leq e^{2c^3t}\varphi(x).
\end{align}
Fatou's lemma hence proves  for all $x\in\R^d$, $t\in[0,T]$  that
\begin{align}
\E\bigl[
\varphi(x+\mu(x)t+\sigma(x)W_{t})\bigr]
=\E[\varphi(\mathfrak{Y}^x_{t})]
\leq e^{2c^3t}\varphi(x).
\end{align}%
The tower property for conditional expectations, 
the fact that for all $t\in[0,T]$, $s\in[t,T]$, $\theta\in\Theta$ it holds that
$W^\theta_{s}-W^\theta_{t}$ and $\F_t$ are independent,
 and the fact that for all $t\in[0,T]$, $s\in[t,T]$, $\theta\in\Theta$, $B\in\mathcal{B}(\R^d)$ it holds that $\P((W^\theta_{s}-W^\theta_{t})\in B)=\P( W^\theta_{s-t}\in B)$ hence prove  for all
$\theta\in\Theta$,
$x\in\R^d$, $t\in[0,T]$, $s\in [t,T]$ 
  that
\begin{align}\begin{split}
\E \bigl[\varphi(Y_{s}^{\theta,t,x})\bigr]
&=\E\biggl[\E \Bigl[\varphi\bigl(Y_{\max\{t,\rdown{s}_{K}\}}^{\theta,t,x}+
\mu(Y_{\max\{t,\rdown{s}_{K}\}}^{\theta,t,x})(s- \max\{t,\rdown{s}_{K}\})\\
&\quad\qquad\qquad+
\sigma(Y_{\max\{t,\rdown{s}_{K}\}}^{\theta,t,x})(W^{ {\theta}}_{s} -W^{ {\theta}}_{ \max\{t,\rdown{s}_{K}\}} )\bigr)\Big|\F_{\rdown{s}_{K}}\Bigr] \biggr]
\\
&
=\E\!\left[\E \Bigl[\varphi \bigl(z+
\mu(z)(s- \max\{t,\rdown{s}_{K}\})+
\sigma(z)(W^{ {\theta}}_{s- \max\{t,\rdown{s}_{K}\}} ) \bigr) \Bigr]\Bigr|_{ z=Y_{t,\max\{t,\rdown{s}_{K}\}}^{\theta,x}}\right]\\
&\leq e^{2c^3(s-\max\{t,\rdown{s}_{K}\})}
\E\!\left[
\varphi\bigl(Y_{\max\{t,\rdown{s}_{K}\}}^{\theta,t,x}\bigr)\right].
\end{split}\end{align}
Induction and \eqref{c09} hence show 
for all 
$\theta\in\Theta$,
$x\in\R^d$, $t\in[0,T]$, $s\in [t,T]$ 
  that 
$
\E \bigl[\varphi(Y_{s}^{\theta,t,x})\bigr]\leq e^{2c^3(s-t)}\varphi(x).
$ 
Jensen's inequality and \eqref{r05} therefore prove  
for all $q\in[0,p]$,
$\theta\in\Theta$,
$x\in\R^d$, $t\in[0,T]$, $s\in [t,T]$ 
that
\begin{align}\label{r06}\begin{split}
&\max\!\big\{\E \bigl[(\varphi(Y_{s}^{\theta,t,x}))^{\frac{q}{p}}\bigr],
\E \bigl[(\varphi(X_{s}^{\theta,t,x}))^{\frac{q}{p}}\bigr]\big\}\\
&
\leq 
\max\!\left\{
\left(
\E \bigl[\varphi(Y_{s}^{\theta,t,x})\bigr]\right)^{\frac{q}{p}},
\left(\E \bigl[\varphi(X_{s}^{\theta,t,x})\bigr]\right)^{\frac{q}{p}}\right\}
\leq e^{2qc^3(s-t)/p}(\varphi(x))^{\frac{q}{p}}.
\end{split}\end{align}
Moreover, observe that the fact that $\mu$ is continuous,
the fact that $\sigma$ is continuous, the fact that for all 
$\theta\in\Theta$, $\omega \in \Omega$ we have that
$[0,T]\ni t\mapsto W^\theta_t(\omega)\in \R^d$ is continuous,
and Fubini's theorem imply  for all
$\theta\in\Theta$ and all
measurable
$\eta\colon  [0,T]\times\R^d\to[0,\infty)$ 
that  
\begin{equation}\label{v09d}
\Delta\times\R^d \ni (t,s,x)\mapsto
\E\bigl[\eta\bigl(s,Y_{s}^{\theta,t,x}\bigr)\bigr]\in[0,\infty]
\end{equation}
is measurable.
Furthermore, note that \eqref{c08}, \eqref{c03}, \eqref{r04}, and, e.g., \cite[Lemma~3.7]{beck2019existence} (applied with $\mathcal{O}\gets\R^d$,  
$V\gets
([0,T]\times\R^d\ni (t,x)\mapsto e^{-2c^3 t/p}\varphi(x) \in(0,\infty))
$
 in the notation of \cite[Lemma~3.7]{beck2019existence})
imply that
$\Delta\times\R^d\times\R^d \ni (t,s,x,y)\mapsto
\bigl(s,X_{s}^{\theta,t,x},X_{s}^{\theta,t,y}\bigr)\in\mathcal{L}^0(\Omega;\R\times\R^d\times\R^d )$ is continuous.
This and the dominated convergence theorem prove  for all $\theta\in\Theta$ and all bounded and
continuous $\eta\colon  [0,T]\times\R^d\times\R^d\to[0,\infty)$  that
$ 
\Delta\times\R^d\times\R^d \ni (t,s,x,y)\mapsto
\E\bigl[\eta\bigl(s,X_{s}^{\theta,t,x},X_{s}^{\theta,t,y}\bigr)\bigr]\in[0,\infty]$
is continuous. Hence, we obtain  
for all $\theta\in\Theta$ and all bounded and
continuous $\eta\colon  [0,T]\times\R^d\times\R^d\to[0,\infty)$ that
$ 
\Delta\times\R^d\times\R^d \ni (t,s,x,y)\mapsto
\E\bigl[\eta\bigl(s,X_{s}^{\theta,t,x},X_{s}^{\theta,t,y}\bigr)\bigr]\in[0,\infty]$
is measurable.
This implies for all $\theta\in\Theta$
and all
measurable
$\eta\colon  [0,T]\times\R^d\times\R^d\to[0,\infty)$ 
that  
\begin{align}\label{v09e}
\Delta\times\R^d\times\R^d \ni (t,s,x,y)\mapsto
\E\bigl[\eta\bigl(s,X_{s}^{\theta,t,x},X_{s}^{\theta,t,y}\bigr)\bigr]\in[0,\infty]
\end{align}
is measurable.
Combining
 \eqref{v09d}, \eqref{r06},
\eqref{v04b}, \eqref{v07}, and
\cite[Proposition~2.2]{HJKN2020} (applied
for every 
$ k\in \{0,1\}$
with $L\gets c$,
$\mathcal{O}\gets \R^d$, 
$ (X_{t,s}^x)_{(t,s,x)\in \Delta \times \R^d}\gets  (\mathfrak{X}^{k,t,x}_{s})_{(t,s,x)\in\Delta \times \R^d}$, 
$
V\gets ([0,T]\times\R^d\ni(s,x)\mapsto e^{2c^3  \beta(T-s)/p}(\varphi (x))^{\beta /p}\in(0,\infty))
$
 in the notation of \cite[Proposition~2.2]{HJKN2020}) hence
 establishes that 
\begin{enumerate}[a)]
\item there exist unique
measurable
$u_k \colon [0,T]\times\R^d\to\R$, $k\in \{0,1\}$,
which satisfy for all $k\in \{0,1\}$, $t\in[0,T]$, $x\in\R^d$ 
that
$\sup_{s\in[0,T]}\sup_{x\in\R^d} \bigl[\lvert u_k(s,x)\rvert{(\varphi(x))^{-\beta /p}}\bigr]+\E\bigl[\bigl\lvert g\bigl(\mathfrak{X}^{k,t,x}_{T}\bigr)\bigr\rvert+\int_{t}^{T}\bigl\lvert f\bigl(s,\mathfrak{X}_{s}^{k,t,x},u_k\bigl(s,\mathfrak{X}_{s}^{k,t,x}\bigr)\bigr)\bigr\rvert\,ds \bigr]<\infty$
and 
\begin{align}\begin{split}
&u_k(t,x)=\E\!\left[g\bigl(\mathfrak{X}^{k,t,x}_{T}\bigr)+\int_{t}^{T}f\bigl(s,\mathfrak{X}_{s}^{k,t,x},u_k(s,\mathfrak{X}_{s}^{k,t,x})\bigr)\,ds \right]
\end{split}\label{v15}\end{align}
and
\item 
we have for all $k\in \{0,1\}$ that 
\begin{align}\label{v14}\begin{split}
\sup_{t\in[0,T]}\sup_{x\in\R^d} \left[\frac{|u_k(t,x)|}{e^{2c^3 \beta  (T-t)/p}(\varphi(x))^{\beta /p}}\right]\leq 
\sup_{t\in[0,T]}\sup_{x\in\R^d}\left[\left[\frac{|g(x)|}{(\varphi(x))^{\beta/p}}
+\frac{|Tf(t,x,0)|}{(\varphi(x))^{\beta /p}}\right]e^{cT}\right]
\leq
2be^{cT}.
\end{split}\end{align}
\end{enumerate}
This proves \eqref{k21}. 
Moreover, note that \cite[Lemma~3.2]{HJKN2020} establishes \eqref{c06}.
Next observe that \eqref{c10} and \eqref{r06} demonstrate  for all 
$\theta\in\Theta$, $t\in[0,T]$, $r\in[t,T]$, $x\in\R^d$ that
\begin{align}\label{r12}\begin{split}
&\max\!\left\{\E\!\left[\bigl\lVert\mu(Y_{\max\{t,\rdown{r}_K\}}^{\theta,t,x})\bigr\rVert^2\right],
\E\!\left[\lVert\sigma(Y_{\max\{t, \rdown{r}_K\}}^{\theta,t,x})\rVert\right]\right\}\\&\leq c^2
\E\!\left[\bigl(\varphi\bigl(Y_{\max\{t, \rdown{r}_K\}}^{\theta,t,x}\bigr)\bigr)^{\frac{2}{p}}\right]\leq c^2 e^{4c^3(r-t)/p}(\varphi(x))^{\frac{2}{p}}.
\end{split}\end{align}
Furthermore, note  that \eqref{c09} demonstrates 
for all $t\in[0,T]$, 
$r\in[t,T]$,
 $x\in\R^d$, $\theta\in\Theta$ that
$ \sigma (\{Y_{\max\{t, \rdown{r}_K\}}^{\theta,t,x}\})\subseteq {\F_r}$.
Combining this and \eqref{r12} with the fact that for all $t\in[0,T]$,  $x\in\R^d$ we have that 
$\E\!\left[\lVert\sigma(x)W_t\rVert  ^{2} \right]
=
\lVert\sigma(x)\rVert^2{t}
$ shows for all
$\theta\in\Theta$,
 $t\in[0,T]$, $r\in[t,T]$, $x\in\R^d$ that
\begin{align}\begin{split}
&\E\!\left[\bigl\lVert\sigma(Y_{\max\{t, \rdown{r}_K\}}^{\theta,t,x})(W^\theta_{r} -W^\theta_{\max\{t, \rdown{r}_K\}} )\bigr\rVert^2\right]
\\
&
=
\E\!\left[\E\!\left[\bigl\|\sigma(y)(W^\theta_{r} -W^\theta_{\max\{t, \rdown{r}_K\}} )\bigr\|^2\right]\Bigr|_{y= Y_{t,\max\{t, \rdown{r}_K\}}^{\theta,x}}\right]
\\
&
=
\E\!\left[ 
\lVert \sigma(Y_{t,\max\{t, \rdown{r}_K\}}^{\theta,x})\rVert
^2(r- \max\{t, \rdown{r}_K\} )\right]
\\
&
\leq 
\E\!\left[ 
\lVert \sigma(Y_{t,\max\{t, \rdown{r}_K\}}^{\theta,x})\rVert
^2\frac{T}{K}\right]
\leq  c^2 e^{4c^3(r-t)/p}(\varphi(x))^{\frac{2}{p}}\frac{T}{K} .
\end{split}\end{align}
This, \eqref{c09}, the triangle inequality, and \eqref{r12}
 imply  for all
$\theta\in\Theta$, $t\in[0,T]$, $r\in[t,T]$, $x\in\R^d$ that 
\begin{align}\begin{split}
&\left(\E\!\left[ \bigl\lVert Y_{\max\{t, \rdown{r}_K\}}^{\theta,t,x}-
Y_{r}^{\theta,t,x}\bigr\rVert^2 \right]\right)^{\frac{1}{2}}\\
&\leq 
\left(\E\!\left[\bigl\lVert\mu(Y_{\max\{t, \rdown{r}_K\}}^{\theta,t,x})\bigr\rVert^2\right]\right)^{\frac{1}{2}}
(r-\max\{t, \rdown{r}_K\})\\&\quad
+\left(\E\!\left[\bigl\lVert\sigma(Y_{t,\max\{t, \rdown{r}_K\}}^{\theta,x})(W^\theta_{r} -W^\theta_{\max\{t, \rdown{r}_K\}} )\bigr\rVert^2\right]\right)^{\frac{1}{2}}\\
&\leq ce^{2c^3 (r-t)/p}(\varphi(x))^{\frac{1}{p}}
\left(\frac{T}{K}\right)^{\frac{1}{2}}
\lvert r-t\rvert^{\frac{1}{2}}
+
 ce^{2c^3 (r-t)/p}(\varphi(x))^{\frac{1}{p}}\left(\frac{T}{K}\right)^{\frac{1}{2}}
\\
&=
 c\bigl[\lvert r-t\rvert^{\frac{1}{2}}+1\bigr]
e^{2c^3 (r-t)/p}(\varphi(x))^{\frac{1}{p}}
\left(\frac{T}{K}\right)^{\frac{1}{2}}
.\label{r35}
\end{split}\end{align}
Next, note that \eqref{c03} and the fact that $c\geq 1$ assure  for all $ {z},y\in\R^d$ with $ {z}\neq y$ that
\begin{align}\label{r34}
\frac{\langle  {z}-y,\mu( {z})-\mu(y)\rangle+\frac{1}{2}\lVert\sigma( {z})-\sigma(y)\rVert^2}{\| {z}-y\|^2}+\frac{(\frac{2}{2}-1)\lVert(\sigma ( {z})-\sigma(y))^{\mathsf{T}}( {z}-y)\rVert^2}{\lVert {z}-y\rVert^4}\leq   2c^2.
\end{align}
This, \cite[Theorem 1.2]{HJ20} (applied
for every $\theta\in\Theta$,
$t\in[0,T)$, $s\in(t,T]$,
$x\in\R^d$ 
with
$
H\gets \R^d$, 
$U\gets \R^m$, 
$D\gets\R^d$,
$T\gets (s-t)$, $({\F}_r)_{r\in [0,T]}\gets (\F_{r+t})_{r\in[0,s-t]}$,
$
(W_r)_{r\in [0,T]}\gets 
(W_{t+r}^{\theta}-W_t^{\theta})_{r\in [0,s-t]}
$,
$(X_r)_{r\in[0,T]}\gets (X^{\theta,t,x}_{t+r})_{r\in [0,s-t]}$,
$(Y_r)_{r\in[0,T]}\gets (Y^{\theta,t,x}_{t+r})_{r\in [0,s-t]}$,
$(a_r)_{r\in[0,T]}\gets (\mu(Y^{\theta,x}_{t,\max\{t,\rdown{t+r}_K\}}))_{r\in [0,s-t]}$,
$(b_r)_{r\in[0,T]}\gets (\sigma(Y^{\theta,t,x}_{\max\{t,\rdown{t+r}_K\}}))_{r\in [0,s-t]}$, $\epsilon\gets 1$,
$p\gets 2$, $\tau \gets (\Omega\ni \omega\mapsto s-t\in[0,s-t])$,
$\alpha\gets 1$, $\beta \gets 1$, $r\gets 2$, 
$q\gets \infty$
in the notation of \cite[Theorem 1.2]{HJ20}), \eqref{c03}, \eqref{r35}, the fact that for all $t\in[0,\infty)$ we have that $\sqrt{t}(\sqrt{t}+1)\leq e^t$, the fact that $1\leq c$, and the fact that $p\geq 2$ imply  for all
$\theta\in\Theta$, $t\in[0,T]$, $s\in[t,T]$, $x\in\R^d$ that
\begin{equation}\begin{split}
&\left(\E\!
\left[\bigl\lVert X_{s}^{\theta,t,x}-Y_{s}^{\theta,t,x}\bigr\rVert^2\right]\right)^{\frac{1}{2}}\\
&
\leq 
\sup_{\substack{ {z},y\in\R^d,\\ {z}\neq y}}
\exp\! \left(\int_{t}^{s}\left[\frac{\langle  {z}-y,\mu( {z})-\mu(y)\rangle+
\frac{(2-1)(1+1)}{2}
\lVert\sigma( {z})-\sigma(y)\rVert^2}{\| {z}-y\|^2}+\frac{1-\frac{1}{2}}{1}+\frac{\frac{1}{2}-\frac{1}{2}}{1}\right]^+
dr
\right)\\
&\quad\cdot
\Biggl[\left(\int_{t}^{s}\E\!\left[ \bigl\lVert\mu \bigl(Y_{\max\{t,\rdown{r}_K\}}^{\theta,t,x}\bigr)-
\mu\bigl (Y_{r}^{\theta,t,x}\bigr)\bigr\rVert^2 \right]dr\right)^{\frac{1}{2}}\\
&\quad
+\sqrt{\frac{(2-1)(1+1)}{1}}\left(\int_{t}^{s}\E\!\left[ \lVert\sigma \bigl(Y_{\max\{\rdown{r}_K\}}^{\theta,t,x}\bigr)-
\sigma \bigl(Y_{r}^{\theta,t,x}\bigr)\rVert^2 \right]dr\right)^{\frac{1}{2}}
\Biggr]
\\
&\leq e^{3c^2(s-t)}3c\left(|s-t|\sup_{r\in[t,s]}\E\!\left[ \bigl\lVert Y_{\max\{t,\rdown{r}_K\}}^{\theta,t,x}-
Y_{r}^{\theta,t,x}\bigr\rVert^2 \right]\right)^{\frac{1}{2}}\\
&\leq e^{3c^2(s-t)}3c 
\lvert s-t\rvert^{\frac{1}{2}}
 c\bigl[\lvert s-t\rvert^{\frac{1}{2}}+1\bigr]
e^{2c^3 (s-t)/p}(\varphi(x))^{\frac{1}{p}}
\left(\frac{T}{K}\right)^{\frac{1}{2}}
\\&\leq 3c^2 e^{4c^2T} 
e^{2c^3 ( {s}-t)/p}(\varphi(x))^{\frac{1}{p}}
\left(\frac{T}{K}\right)^{\frac{1}{2}}.\label{r16}
\end{split}
\end{equation}
 Next, observe that
 \eqref{c04}, \eqref{r34}, and \cite[Corollary~2.26]{CoxHutzenthalerJentzen2014} (applied
for every $t\in[0,T)$, $s\in(t,T]$
with
 $T\gets s-t$,
$O\gets \R^d$,
 $(\mathcal{F}_{r})_{r\in[0,T]} \gets (\F_{t,t+r})_{r\in[0,s-t]} $,
 $(W_{r})_{r\in[0,T]} \gets (W^0_{t+r}-W^0_t)_{r\in[0,s-t]} $,
$\alpha_0\gets 0$, 
$\alpha_1\gets 0$,
$\beta_0\gets 0$,
$\beta_1\gets 0$,
$c\gets 2c^2$,
$r\gets 2$, $p\gets 2$, $q_0\gets \infty$, $q_1\gets \infty$,
$U_0\gets (\R^d\ni x\mapsto 0\in \R)$,
$U_1\gets (\R^d\ni x\mapsto 0\in [0,\infty))$,
$\overline{U}\gets (\R^d\ni x\mapsto 0\in \R)$,
$(X^x_{r})_{r\in[0,T],x\in\R^d}\gets (X_{t+r}^{0,t,x})_{r\in[0,s-t],x\in\R^d} $
in the notation of \cite[Corollary~2.26]{CoxHutzenthalerJentzen2014}) demonstrate   
for all  $t\in[0,T)$, $s\in(t,T]$, $x,y\in\R^d$  that
$
\left(\E\!\left[
\lVert X_{s}^{0,t,x}-X_{s}^{0,t,y}\rVert^2\right]\right)^{\frac{1}{2}}\leq e^{2c^2(s-t)}\lVert x-y\rVert.
$
This
and \eqref{r16}
imply 
for all $t\in[0,T]$, $s\in[t,T]$,
$r\in[s,T]$, $x,y\in\R^d$  that
\begin{align}\begin{split}
&
\left(\E\!\left[
\E\!\left[\bigl\lVert X_{r}^{0,s,\mathfrak{x}}-X_{r}^{0,s,\mathfrak{y}}\bigr\rVert^{2}\right]\Bigr|_{\substack{(\mathfrak{x},\mathfrak{y})=(X^{0,t,x}_{s},Y_{s}^{0,t,x})}}\right]\right)^{\frac{1}{2}}
\leq \left(\E\!\left[
\left[e^{2c^2(r-s)}\left\lVert X^{0,t,x}_{s}-Y_{s}^{0,t,x}\right\rVert \right]^2\right]\right)^{\frac{1}{2}}\\&\leq 
e^{2c^2(r-s)}
3c^2 e^{4c^2T} 
e^{2c^3 (s-t)/p}(\varphi(x))^{\frac{1}{p}}
\left(\frac{T}{K}\right)^{\frac{1}{2}}
\leq 
3c^2 e^{4c^2T} 
\left(\frac{T}{K}\right)^{\frac{1}{2}}\bigl[ e^{4c^3 (T-t)/p}(\varphi(x))^{\frac{2}{p}}\bigr]^{\frac{1}{2}}
.
\end{split}\label{v09c}\end{align}
Furthermore, note that \eqref{c04} and Tonelli's theorem
 ensure  for all 
$t\in[0,T]$, $s\in[t,T]$, $r\in[s,T]$, $x,y\in\R^d$ and all measurable $h\colon \R^{d}\times\R^d\to [0,\infty)$
 that
$
\R^d \times \R^d \ni (y_1,y_2) \mapsto \E\bigl[h\bigl(X^{0,s,y_1}_{r},X^{0,s,y_2}_{r}\bigr)\bigr] \in [0,\infty]
$
is measurable.
Moreover, observe that \eqref{c04} assures  for all 
$t\in[0,T]$, $s\in[t,T]$, $r\in[s,T]$, $x,y\in\R^d$  that
$X^{0,t,x}_{s} $ and 
$X^{0,s,y}_{r} $ are independent.
This and  the disintegration theorem
show that for all
$t\in[0,T]$, $s\in[t,T]$, $r\in[s,T]$, $x,y\in\R^d$ and all measurable $h\colon \R^{d}\times\R^d\to [0,\infty)$ it holds
 that
$
\E\bigl[ \E\bigr[ h\bigl(X^{0,s,\tilde{x}}_{r},X^{0,s,\tilde{y}}_{r}\bigr) \bigr] |_{\tilde{x}=X_{s}^{0,t,x},\tilde{y}=X_{s}^{0,t,y}} \bigr]
   =\E\!\left[h\bigl(X^{0,t,x}_{r},X^{0,t,y}_{r}\bigr)\right] .
$
Combining 
\eqref{c04}, 
\eqref{c09},  
\eqref{r06}, 
\eqref{v09e}, 
\eqref{v07}, 
\eqref{v09c}, 
\eqref{v15}, 
\eqref{v14}, 
\cite[Lemma~2.3]{HJKN2020} (applied with
$L\gets c$, $\rho\gets 2c^3$, 
$\eta\gets 1$,
$\delta\gets  3c^2 e^{4c^2T} 
\left(\frac{T}{K}\right)^{\frac{1}{2}}$,
$p\gets p/\beta$,
$ q\gets 2$,
$(X_{t,s}^{x,1})_{t\in[0,T],s\in[t,T],x\in\R^d}\gets (X_{s}^{0,t,x})_{t\in[0,T],s\in[t,T],x\in\R^d}$,
$(X_{t,s}^{x,2})_{t\in[0,T],s\in [t,T],x\in\R^d}\gets (Y_{s}^{0,t,x})_{t\in[0,T],s\in [t,T],x\in\R^d}$,
$V\gets b^{p/\beta}\varphi$,
$\psi \gets \bigl([0,T]\times\R^d\ni(t,x)\mapsto 
e^{
4c^3 (T-t)/p}(\varphi(x))^{\frac{2}{p}}\in(0,\infty)\bigr)$, 
$u_1\gets u_0$, $u_2\gets u_1$  in the notation of \cite[Lemma~2.3]{HJKN2020}),
the fact that
$1+cT\leq e^{cT}$, 
the fact that $c\geq1$,
the fact that $\varphi\geq1$,
the fact that $p\geq 2$,
and
the fact that $p\geq 2\beta$
hence implies  
for all $t\in[0,T]$, $x\in\R^d$ that
\begin{align}\begin{split}
&\lvert u_0(t,x)-  u_1(t,x)\rvert\\&\leq
4(1+cT)T^{-\frac{1}{2}}
e^{cT+(2c^3\beta/p+c) T}
(b^{p/\beta}\varphi(x))^{\frac{\beta }{p}}
\bigl[ e^{
4c^3 (T-t)/p}(\varphi(x))^{\frac{2}{p}}\bigr]^{\frac{1}{2}}
3c^2 e^{4c^2T} 
\left(\frac{T}{K}\right)^{\frac{1}{2}}\\
&\leq 4 e^{cT} T^{-\frac{1}{2}}e^{cT+c^3T+cT}b (\varphi(x))^{\frac{\beta}{p}}e^{c^3T}(\varphi(x))^{\frac{1}{p}}3c^2 e^{4c^2T}{\left(\frac{T}{K}\right)}^{\frac{1}{2}}
\\
&\leq 12bc^2 T^{-\frac{1}{2}}e^{{9}c^3T}
(\varphi(x))^{\frac{\beta+1}{p}}\left(\frac{T}{K}\right)^{\frac{1}{2}}.
\end{split}\label{v16}\end{align}
For the rest of this proof let $V \in C([0,T]\times\R^d,[1,\infty) )$
satisfy for all $t\in [0,T]$, $x\in \R^d$ that
\begin{align}
V(t,x)=be^\frac{2c^3\beta(T-t)}{p}(\varphi(x))^{\frac{\beta}{p}}.
\end{align}
Then \eqref{r06} and the fact that $p\geq \mathfrak{p}\beta$ show for all $t\in [0,T]$, $s\in [t,T]$, $x\in \R^d$ that
\begin{align}
\left\lVert
V(s,Y^{0,t,x}_s)\right\rVert_{\mathfrak{p}}
=b e^\frac{2c^3\beta(T-s)}{p}
\left\lVert\varphi
(Y^{0,t,x}_s)^{\frac{\beta}{p}}\right\rVert_{\mathfrak{p}}
\leq b e^\frac{2c^3\beta(T-s)}{p}
e^\frac{2c^3\beta(s-t)}{p}(\varphi (x))^\frac{\beta}{p}=V(t,x).
\end{align}
Then \cref{k09} (applied with
$d\gets d$, $p_v\gets \mathfrak{p}$, $c\gets c$, $T\gets T$,
$f\gets f$, $g\gets g$, $V\gets V$, $\Theta\gets\Theta$, $(\mathfrak{t}^\theta)_{\theta\in \Theta}\gets
(\mathfrak{t}^\theta)_{\theta\in \Theta} $,
$X\gets Y$, $(U^\theta_{n,m})_{\theta\in \Theta,n,m\in \Z}\gets (U^\theta_{n,m})_{\theta\in \Theta,n,m\in \Z}$, $p\gets \mathfrak{p}$, $q_1\gets 1$ in the notation of \cref{k09}), \eqref{t27c}, and the independence and distributional assumptions show 
for all 
$m,n\in \N$,
$t\in [0,T]$, $x\in \R^d$ that
\begin{align} \begin{split} 
\left\lVert
{U}_{n,m}^{0}(t,x)-u_1(t,x)\right\rVert_\mathfrak{p}
&\leq 2(\mathfrak{p}-1)^{\frac{n}{2}}e^{5cTn}e^{m^{\mathfrak{p}/2}/\mathfrak{p}}m^{-n/2}V(t,x)\\
&\leq 
2\mathfrak{p}^{\frac{n}{2}}e^{5cTn}e^{m^{\mathfrak{p}/2}/\mathfrak{p}}m^{-n/2}be^\frac{2c^2\beta T}{p}(\varphi(x))^{\frac{\beta}{p}}
.
\end{split}\end{align}
This, the triangle inequality, the fact that $p\geq 2\beta$, and the fact that $\varphi \geq 1$ show for all $t\in [0,T]$, $x\in \R^d$ that
\begin{align} \begin{split} 
&\left\lVert
{U}_{n,m}^{0}(t,x)-u_0(t,x)\right\rVert_\mathfrak{p}\\
&\leq 
\left\lVert
{U}_{n,m}^{0}(t,x)-u_1(t,x)\right\rVert_\mathfrak{p}
+\lvert u_0(t,x)-  u_1(t,x)\rvert\\
&\leq 2\mathfrak{p}^{\frac{n}{2}}e^{5cTn}e^{m^{\mathfrak{p}/2}/\mathfrak{p}}m^{-n/2}be^\frac{2c^2\beta T}{p}(\varphi(x))^{\frac{\beta}{p}}+
12bc^2 T^{-\frac{1}{2}}e^{{9}c^3T}
(\varphi(x))^{\frac{\beta+1}{p}}\left(\frac{T}{K}\right)^{\frac{1}{2}}\\
&\leq 12bc^2e^{9c^3T}
(\varphi(x))^{\frac{\beta+1}{p}}\left[
2\mathfrak{p}^{\frac{n}{2}}e^{5cTn}e^{m^{\mathfrak{p}/2}/\mathfrak{p}}m^{-n/2}+\frac{1}{\sqrt{K}}
\right]
\end{split}
\end{align}
This completes the proof of \cref{c02}.
\end{proof}
\subsection{Complexity analysis for MLP approximations involving Euler-Maruyama approximations}
We are now in a position to prove \cref{c40}.
\begin{proof}
[Proof of \cref{c40}]For every $d\in \N$ let $\varphi_d\in C(\R^d,\R)$ satisfy for all $x\in \R^d$ that
\begin{align}\label{t20}
\varphi_d(x)=2^\mathfrak{p}c^\mathfrak{p} d^{c\mathfrak{p}}(d^{2c}+\lVert x\rVert^2)^{\frac{\mathfrak{p}}{2}}.
\end{align}
Then \eqref{d05} shows for all $d\in \N$,  $x\in \R^d$ that
\begin{align}\label{t21} \begin{split} 
&\max
\{
\lVert \mu^d(0)\rVert+c\lVert x\rVert,
\lVert \sigma^d(0)\rVert+c\lVert x\rVert
\}\\
&\leq cd^c+ c\lVert x\rVert
=c(d^c+\lVert x\rVert)\leq 
2c(d^{2c}+\lVert x\rVert^2)^\frac{1}{2}\leq c (\varphi_d(x))^\frac{1}{\mathfrak{p}}.\end{split}
\end{align}
Next,
\cite[Lemma~2.6]{HN2022a} (applied for every $d\in \N$ with 
$d\gets d$, $m\gets d$, $a\gets d^{2c}$, $c\gets 0$,
$p\gets \mathfrak{p}/2$, 
$\mu\gets 0$, $\sigma\gets 0$, $\varphi\gets \varphi_d/(2^\mathfrak{p}c^\mathfrak{p}d^{\mathfrak{p}c})$ in the notation of \cite[Lemma~2.6]{HN2022a}) and \eqref{t20} show
for all $x,z\in \R^d$ that
\begin{align}
\lVert (\varphi_d'(x))(z) \rVert\leq \mathfrak{p} (\varphi_d(x))^{1-\frac{1}{\mathfrak{p}}} \lVert z\rVert, \quad
\lVert (\varphi_d''(x))(z,z) \rVert\leq \mathfrak{p}^2 (\varphi_d(x))^{1-\frac{2}{\mathfrak{p}}} \lVert z\rVert^2.
\end{align}
This, \eqref{t21}, and the fact that $\mathfrak{p}^2\leq c$ show for all $d\in \N$,
$\varepsilon\in (0,1)$, $x,z\in \R^d$ that
\begin{align}
\max\!\left\{
\frac{\lvert(\varphi_d'(x))(z)\rvert}{(\varphi_d(x))^\frac{\mathfrak{p}-1}{\mathfrak{p}}\lVert z\rVert},
\frac{(\varphi_d''(x))(z,z)}{(\varphi_d(x))^\frac{\mathfrak{p}-2}{\mathfrak{p}}\lVert z\rVert^2},
\frac{c\lVert x\rVert+\lVert\mu^d(0)\rVert}{(\varphi_d(x))^\frac{1}{\mathfrak{p}}},
\frac{c\lVert x\rVert+\lVert\sigma^d(0)\rVert}{(\varphi_d(x))^\frac{1}{\mathfrak{p}}}
\right\}\leq c.\label{c09b}
\end{align}
Next, \eqref{t23} and \eqref{t20} show for all $d\in \N$, $x\in \R^d$ that
\begin{align}
\max \{
\lvert Tf(0)\rvert,
\lvert g^d(x)\rvert\}
\leq  (\varphi_d(x))^\frac{1}{\mathfrak{p}}.\label{c10c}
\end{align}
Furthermore, \eqref{t24} show for all $d\in \N$, $x,y\in \R^d$ that
\begin{align}
\lvert g^d(x)-g^d(y)\rvert\leq \frac{c}{\sqrt{T}}\lVert x-y\rVert
\leq \frac{(\varphi_d(x))^\frac{1}{\mathfrak{p}}+(\varphi_d(y))^\frac{1}{\mathfrak{p}}}{2\sqrt{T}}\lVert x-y\rVert
\leq \frac{(\varphi_d(x)+\varphi_d(y))^\frac{1}{\mathfrak{p}}}{\sqrt{T}}\lVert x-y\rVert.
\end{align}
This,
\cref{c02}
(applied for every $d,K\in \N$
with 
$T\gets T$,
$\mathfrak{p}\gets \mathfrak{p}$,
$\beta\gets 1$, $b\gets 1$, $c\gets c$, 
$p\gets \beta$, $\varphi\gets \varphi_d$,
$g\gets g^d$, $f\gets f$, $\mu\gets \mu^d$, $\sigma\gets \sigma^d$,
$(\mathfrak{t}^\theta)_{\theta\in \Theta}\gets (\mathfrak{t}^\theta)_{\theta\in \Theta}$,
$(W^\theta)_{\theta\in \Theta}\gets (W^{d,\theta})_{\theta\in \Theta}$,
$(Y^{\theta,t,x})_{\theta\in \Theta,t\in [0,T], x\in \R^d}\gets
(Y^{d,\theta,K,t,x})_{\theta\in \Theta,t\in [0,T], x\in \R^d}
$, 
$(U^\theta_{n,m})_{\theta\in \Theta,n,m\in Z}\gets
(U^{d,\theta,K}_{n,m})_{\theta\in \Theta,n,m\in Z}
$ in the notation of \cref{c02})
, \eqref{c09b},   \eqref{c10c}, \eqref{t25}, and \eqref{t24}
show that the following items are true.

\begin{enumerate}[(A)]
\item For every $t\in [0,T]$, $\theta\in \Theta$, $d\in \N$ there exists an up to indistinguishability unique continuous
random field
$X^{d,\theta,t,\cdot }=(X^{d,\theta,t,x}_s)_{s\in [t,T],x\in \R^d}
\colon [t,T]\times \R^d\times \Omega\to \R^d
$ which satisfies that for all $x\in \R^d$
it holds that 
$
(X^{d,\theta,t,x}_s)_{s\in [t,T]}
$
is $(\F_s)_{s\in [t,T]}$-adapted and which satisfies that for all $s\in [t,T]$, $x\in \R^d$ it holds $\P$-a.s.\ that
\begin{align}
X^{d,\theta,t,x}_s=x+\int_{t}^{s}\mu^d (X^{d,\theta,t,x}_r)\,dr
+\int_t^s\sigma^d(X^{d,\theta,t,x}_r)\,dW^{d,\theta}_r.
\end{align}
\item For every $d\in \N$ there exists a unique
measurable
$u^d\colon [0,T]\times\R^d\to\R$ which satisfies
for all $t\in[0,T]$, $x\in \R^d$ that
\begin{align}\label{t29}
\left(\sup_{s\in[0,T],y\in\R^d} [{\lvert u^d(s,y)\rvert}{(\varphi_d(y))^{-1 /\mathfrak{p}}}]\right)+\int_{t}^{T}\E\bigl[\lvert f(u^d(s,X_{s}^{d,0,t,x}))\rvert \bigr]\,ds
+
\E\bigl[\lvert g^d(X^{d,0,t,x}_{T})\rvert\bigr]  < \infty
\end{align}
and\begin{align}
u^d(t,x)=\E\!\left[g^d(X^{d,0,t,x}_{T})\right]+\int_{t}^{T}\E\!\left[f(u^d(s,X_{s}^{0,t,x}))\right]ds .\label{t28}
\end{align}
\item For all $d,K,m\in \N$, $t\in [0,T]$, $x\in \R^d$, $n\in \N_0$ we have that $U^{d,\theta,K}_{n,m}(t,x)$ is measurable.
\item For all
$t\in [0,T]$, $x\in \R^d$, $m,n\in \N$ we have that
\begin{align} \begin{split} 
&\left\lVert
{U}_{n,m}^{d,0,K}(t,x)-u^d(t,x)\right\rVert_\mathfrak{p}\leq  12c^2e^{9c^3T}
(\varphi_d(x))^{\frac{2}{\mathfrak{p}}}\left[
2\mathfrak{p}^{\frac{n}{2}}e^{5cTn}e^{m^{\mathfrak{p}/2}/\mathfrak{p}}m^{-n/2}+\frac{1}{\sqrt{K}}
\right].
\end{split}\label{t43}
\end{align}
\end{enumerate}
Next, the triangle inequality, \eqref{t23}, \eqref{t24},  and the fact that $\forall\,x\in \R^d\colon (1+x)^2\leq 1(1+x^2)$ show 
for all $d\in \N$, $x\in \R^d$ that
\begin{align} \begin{split} 
\langle
x,\mu^d(x)
\rangle
&\leq \lVert x\rVert \left(\lVert \mu^d(x)-\mu^d(0)  \rVert
+
\lVert
\mu^d(0)  \rVert\right)\\
&\leq 
\lVert x\rVert(c\lVert x\rVert+cd^c)\\
&
\leq (1+\lVert x\rVert)^2cd^c\\
&\leq 2cd^c(1+\lVert x\rVert^2).
\end{split}\label{c19b}
\end{align}
Furthermore, the Cauchy--Schwarz inequality implies for all
$d\in \N$, $x,y\in \R^d$ that
\begin{align}
\lVert\sigma^d(x)y\rVert^2= \sum_{i=1}^d \left\lvert\sum_{j=1}^d(\sigma^d)_{ij}(x)y_j\right\rvert^2
\leq \sum_{i=1}^d\left(\sum_{j=1}^{d}\lvert(\sigma^d)_{ij}(x)\rvert^2\right)\left(\sum_{j=1}^{d}\lvert y_j\rvert^2\right)
\leq \lVert\sigma(x)\rVert^2\lVert y\rVert^2.
\end{align}
This and \eqref{t24} show for all
$d\in \N$, $x,y\in \R^d$ that
\begin{align}
\lVert\sigma^d(x)y\rVert\leq 
\lVert\sigma^d(x)\rVert
\lVert y\rVert
\leq (\lVert\sigma^d(x)-\sigma^d(0)\rVert+\lVert\sigma^d(0)\rVert)
\lVert y\rVert
\leq (c\lVert x\rVert+ cd^c)\lVert y\rVert
\leq cd^c(1+\lVert x\rVert)\lVert y\rVert.
\end{align}
This, \eqref{c19b}, 
\cite[Theorem~1.1]{beck2021nonlinear}
(applied for every $d\in \N$ with
$d\gets d$, $L\gets 2cd^c$, $T\gets T$, $\mu\gets \mu^d$, $\sigma\gets \sigma^d$,
$f\gets (\R^d\times \R \ni (x,w)\mapsto f^d(w)\in \R )$,
$g\gets g^d$, $W\gets W^{d,\theta}$ in the notation of 
\cite[Theorem~1.1]{beck2021nonlinear}),
\eqref{t24}, the fact that for every $d\in \N$, $g^d$ is polynomially growing (cf. \eqref{t23}), and the fact that for every $d\in \N$,
$u^d$ is polynomially growing (cf. \eqref{t20} and \eqref{t29})
 show for every $d\in \N$ that $u^d$ is the unique at most polynomially growing viscosity solution of 
\begin{align}
\frac{\partial u^d}{\partial t}(t,x)
+\frac{1}{2}\mathrm{trace}(\sigma^d(\sigma^d(x))^\top
(\mathrm{Hess}_x u^d(t,x) ))
+\langle \mu^d(x),(\nabla_xu^d) (t,x)\rangle
+f(u^d(t,x))=0
\end{align}
with $u^d(T,x)=g^d(x)$ for $t\in (0,T)\times \R^d$. This establishes \eqref{t30}.

Next, 
\eqref{t43} shows that there exists $\kappa\in (0,\infty)$ such that for all $d,m,n\in \N$ we have that
\begin{align} \begin{split} 
\sup_{t\in [0,T],x\in [0,\mathbf{k}]^d}\left\lVert
{U}_{n,m}^{d,0,m^n}(t,x)-u^d(t,x)\right\rVert_\mathfrak{p}
&\leq 
\sup_{x\in [0,\mathbf{k}]^d}\left(
 12c^2e^{9c^3T}
(\varphi_d(x))^{\frac{2}{\mathfrak{p}}}
3\mathfrak{p}^{\frac{n}{2}}e^{5cTn}e^{m^{\mathfrak{p}/2}/\mathfrak{p}}m^{-n/2}\right)\\
&\leq \kappa d^\kappa
\mathfrak{p}^{\frac{n}{2}}e^{5cTn}e^{m^{\mathfrak{p}/2}/\mathfrak{p}}m^{-n/2}.
\end{split}\label{t43a}
\end{align}
For every $\varepsilon\in (0,1)$ let
\begin{align}
N_\varepsilon=\inf\!\left\{
n\in \N\colon \left(
\frac{\mathfrak{p}^{\frac{1}{2}}e^{5cT}\exp ( \frac{(M_n)^{\mathfrak{p}/2}}{n})}{(M_n)^\frac{1}{2}}
\right)^n\leq \varepsilon\right\}.\label{c14c}
\end{align}
For every $\epsilon\in (0,1)$, $d\in \N$ let 
\begin{align}\label{c20b}
\varepsilon(d,\epsilon)=\frac{\epsilon}{\kappa d^\kappa},\quad n(d,\epsilon)= N_{\varepsilon(d,\epsilon)}.
\end{align}
For every $\delta\in (0,1)$ let
\begin{align}
C_\delta=\sup_{\varepsilon\in (0,1)}\left[\varepsilon^{4+\delta}(3M_{N_\varepsilon})^{2N_\varepsilon}\right].\label{c21b}
\end{align}
Next,
\cite[Lemma~4.5]{HJKP2021} and the definition of $(M_n)_{n\in \N}$ show
 that
$\liminf_{j\to\infty}M_j=\infty$, $\limsup_{j\to\infty} \frac{(M_j)^{\mathfrak{p}/2}}{j}<\infty$, and $\sup_{k\in \N}\frac{M_{k+1}}{M_k}<\infty $.
Then 
\eqref{c21b}  and
\cite[Lemma~5.1]{AJKP2024}
(applied with
$L\gets 1$, $T\gets \mathfrak{p}^{\frac{1}{2}}e^{5cT}-1$, $(m_{k})_{k\in \N}\gets (M_{k})_{k\in \N}$
in the notation of \cite[Lemma~5.1]{AJKP2024})
show for all $\delta, \varepsilon\in (0,1)$ that
$N_\varepsilon<\infty$ and
$C_\delta<\infty$.
Next, \eqref{t43a} and \eqref{c14c} 
show for all $d\in \N$, $\epsilon\in (0,1)$ that
\begin{align} \begin{split} 
&
\sup_{t\in [0,T],x\in [0,\mathbf{k}]^d}\left\lVert
{U}_{n(d,\varepsilon),M_{n(d,\varepsilon)}}^{d,0,(M_{n(d,\varepsilon)})^{{n(d,\varepsilon)}}}(t,x)-u^d(t,x)\right\rVert_\mathfrak{p}
\leq 
\kappa d^\kappa
\mathfrak{p}^{\frac{n}{2}}e^{5cT{N_{\varepsilon(d,\epsilon)}}}e^{(M_{{N_{\varepsilon(d,\epsilon)}}})^{\mathfrak{p}/2}/\mathfrak{p}}(M_{{N_{\varepsilon(d,\epsilon)}}})^{-{N_{\varepsilon(d,\epsilon)}}/2}\\
&\leq\kappa d^\kappa \varepsilon(d,\epsilon)=\epsilon. \end{split}\label{t32}
\end{align}
Next, \eqref{c29} show for all $d,K,n,m\in \N$ that
\begin{align}\label{c29a}
C_{0,m}^{d,K}=0,\quad 
C_{n,m}^{d,K}\leq 2cd^cK
m^n
+ \sum_{\ell=0}^{n-1}m^{n-\ell}\left(3cd^cK+
C_{\ell,m}^{d,K}+
C_{\ell-1,m}^{d,K}\right).
\end{align}
This and \cite[Lemma~3.14]{beck2020overcomingElliptic} show for all $d,K,m,n\in \N$ that
\begin{align}
C_{n,m}^{d,K}\leq 3cd^cK (3m)^n.
\end{align}
This, \eqref{c20b}, and \eqref{c21b}
show that for all $d\in \N$, $\epsilon\in (0,1)$ that
\begin{align} \begin{split} 
&
C^{d,(M_{n(d,\epsilon)})^{n(d,\epsilon)}}_{n(d,\epsilon),M_{n(d,\epsilon)}}
=
C_{N_{\varepsilon(d,\epsilon)},M_{N_{\varepsilon(d,\epsilon)}}}^{d,(M_{N_{\varepsilon(d,\epsilon)}})^{N_{\varepsilon(d,\epsilon)}}}\leq cd^c (3M_{N_{\varepsilon(d,\epsilon)}})^{2N_{\varepsilon(d,\epsilon)}}\\
&\leq cd^cC_{\delta}(\varepsilon(d,\varepsilon))^{-(4+\delta)}
=cd^cC_{\delta}\left(\frac{\epsilon}{\kappa d^\kappa}\right)^{-(4+\delta)}
=cd^c(\kappa d^\kappa)^{4+\delta}\epsilon^{-(4+\delta)}.
\end{split}
\end{align}
This, \eqref{t32}, \eqref{c20b}, the fact that $\forall\,\varepsilon\in(0,1)\colon N_\varepsilon<\infty$, and the fact that $\forall\,\delta\in (0,1)\colon C_\delta<\infty $ complete the proof of \cref{c40}.
\end{proof}

\section{DNNs}\label{c36}
Our main goal in this section is to prove \cref{k04c}, which states that the MLP approximations defined by \eqref{k04d} can be represented by DNNs. Furthermore, in \cref{k04c} we also bound the length and the supremum norm of the vectors of their layer dimensions. 
Note that in this paper we consider different types of activation functions than ReLU. 

\subsection{DNN representation of the one-dimensional identity}
In Lemma~\ref{a10} and~\ref{a10b} we prove that the identity in $\R$ can be represented by a DNN when considering ReLU, leaky ReLU, or softplus activation function. Later in \cref{s02} as well as in the setting of \cref{a08} we consider an
arbitrary activation function but assume that the identity can be represented by a DNN.
\begin{lemma}\label{a10}
Assume \cref{m07}.
Let $\alpha\in [0,\infty)$ satisfy for all $x\in \R^d$ that $a(x)=\max \{x,\alpha x\}$. Then 
$\mathrm{Id}_\R\in\mathcal{R}(  \{\Phi\in \mathbf{N}\colon \mathcal{D}(\Phi)=(1,2,1)\})$.
\end{lemma}
\begin{proof}
[Proof of \cref{a10}]See \cite[Lemma~3.5]{AJK+2023}.
\end{proof}
\begin{lemma}\label{a10b}
Assume \cref{m07} and assume for all $x\in \R$ that
$a(x)=\ln (1+e^x)$. Then 
$\mathrm{Id}_\R\in\mathcal{R}( \{\Phi\in \mathbf{N}\colon \mathcal{D}(\Phi)=(1,2,1)\})$.
\end{lemma}
\begin{proof}
[Proof of \cref{a10b}] See \cite[Lemma~3.8]{AJK+2023}.
\end{proof}
\subsection{DNN representation of the $d$-dimensional identity}In \cref{a01} below we prove that if the identity in $\R$ can be represented by a DNN then the identity in $\R^d$ can also be represented by a DNN. 
\begin{lemma}\label{a01}
Assume \cref{m07}.  Let $d,{\mathfrak{d}}\in \N$, $\phi \in \mathbf{N}$
satisfy for all $x\in\R$ that 
$\mathcal{D}(\phi)=(1,{\mathfrak{d}},1)$
 and $(\mathcal{R}(\phi))(x)=x$. Then there exists 
$\Phi \in \mathbf{N}$ which satisfies for all $x\in \R^d$ that
$\mathcal{D}(\Phi)=(d,{\mathfrak{d}} d,d)\in \R^{3}$ and $(\mathcal{R}(\Phi))(x)=x$.
\end{lemma}
\begin{proof}
[Proof of \cref{a01}] Let 
$W_1\in \R^{{\mathfrak{d}}\times 1}$, 
$B_1\in \R^{\mathfrak{d}}$,
$W_2\in \R^{1\times {\mathfrak{d}}}$,
$B_2\in \R$
satisfy
$\phi=((W_1,B_1),(W_2,B_2))$. Then by definition for all 
$x^0\in\R $, $x^1\in\R^{\mathfrak{d}} $
with
$x^1= \mathbf{A}_{\mathfrak{d}}(W_1x^0+B_1)$
we have that $(\mathcal{R}(\phi)) (x^0)=W_2x^1 +B_2$,
i.e.,
\begin{align}
(\mathcal{R}(\phi)) (x^0)=W_2\mathbf{A}_{\mathfrak{d}}(W_1x^0+B_1) +B_2.\label{a02}
\end{align} Now, let 
$\Phi\in \mathbf{N}$,
$\widehat{W}_1\in \R^{{\mathfrak{d}} d\times 
d}$, $\widehat{B}_1\in \R^{{\mathfrak{d}} d}$,
$\widehat{W}_{2}\in \R^{d\times {\mathfrak{d}} d}$,
$\widehat{B}_{2}\in \R^{d}$
 satisfy for all $n\in [1,H]\cap\Z$ that
$\Phi=((\widehat{W}_1,\widehat{B}_1), (\widehat{W}_2,\widehat{B}_2))$,
\begin{align}
\widehat{W}_1=\begin{pmatrix}
W_1	&		&	\\
	& 	\ddots&	\\
	&		&W_1\\
\end{pmatrix},\quad 
\widehat{B}_1=
\begin{pmatrix}
B_1\\
\vdots\\
B_1
\end{pmatrix}
,\quad
\widehat{W}_{2}=\begin{pmatrix}
W_2	&		&	\\
	& 	\ddots&	\\
	&		&W_2\\
\end{pmatrix},\quad 
\widehat{B}_{2}=
\begin{pmatrix}
B_2\\
\vdots\\
B_2
\end{pmatrix}.
\end{align}
Then $\mathcal{D}(\Phi)=(d,{\mathfrak{d}} d,d)\in \R^{3}$. Furthermore, \eqref{a02} shows that
for all $x^0=(x^0_1,\ldots,x^0_d)^\top\in\R^d $, $x^1\in \R^{{\mathfrak{d}} d}$ satisfying that
$
x^1=\mathbf{A}_{{\mathfrak{d}} d}(\widehat{W}_nx^{0}+\widehat{B}_n)
$
we have that
\begin{align}
x^1=\mathbf{A}_{{\mathfrak{d}} d}(\widehat{W}^1x^0+\widehat{B}_1)
=
\mathbf{A}_{{\mathfrak{d}} d}
\begin{pmatrix}
{W}_1x^0_1+B_1\\
\vdots
\\
{W}_1x^0_d+B_1
\end{pmatrix}
=
\begin{pmatrix}
\mathbf{A}_{\mathfrak{d}}({W}_1x^0_1+B_1)\\
\vdots
\\
\mathbf{A}_{\mathfrak{d}}(
{W}_1x^0_d+B_1)
\end{pmatrix}
\end{align}
and 
\begin{align}
(\mathcal{R}(\Phi)) (x^0)&=W_2x^1 +B_2
=
\begin{pmatrix}W_2
\mathbf{A}_{\mathfrak{d}}({W}_1x^0_1+B_1)+B_2\\
\vdots
\\W_2
\mathbf{A}_{\mathfrak{d}}(
{W}_1x^0_d+B_1)+B_2
\end{pmatrix}
=\begin{pmatrix}
x^0_1\\
\vdots\\
x^0_d
\end{pmatrix}=x^0.
\end{align}
This completes the proof of \cref{a01}.  
\end{proof}

\subsection{Approximation of one-dimensional Lipschitz functions by DNNs}In Lemmas~\ref{p10} and~\ref{p10a} we prove that one dimensional Lipschitz functions can be well approximated by DNNs. 
\begin{lemma}\label{p10}
Assume \cref{m07}. Let $\alpha\in [0,\infty)\setminus \{1\}$ and assume for all $x\in \R$ that
$a(x)= \max\{x,\alpha x\}$.
 Let $f\in C(\R,\R)$, $L\in \R$, $q\in (1,\infty)$ satisfy for all $x,y\in \R$ that
$\lvert f(x)-f(y)\rvert\leq L\lvert x-y\rvert$.
Then there exist $c\in (0,\infty)$,
$(f_\varepsilon)_{\varepsilon\in (0,1)}\subseteq C(\R,\R)$ such that for all $\varepsilon\in (0,1) $, $x,y\in \R$
we have that 
$\lvert f_\varepsilon(x)-f_\varepsilon(y)\rvert\leq L\lvert x-y\rvert$, 
$\lvert f_\varepsilon(x)-f(x)\rvert\leq \varepsilon(1+\lvert x\rvert^q)$,
 and
$ f_\varepsilon\in \mathcal{R}(\{\Phi\in \mathbf{N}\colon \dim(\mathcal{D}(\Phi ))=3
, \supnorm{\mathcal{D}(\Phi )}\leq c\varepsilon^{-c}
\})$.
\end{lemma}
\begin{proof}
[Proof of \cref{p10}]See \cite[Corollary~4.13]{AJK+2023}.
\end{proof}
\begin{lemma}\label{p10a}
Assume \cref{m07}. Assume for all $x\in \R$ that
$a(x)= \ln (1+e^x)$.
 Let $f\in C(\R,\R)$, $L\in \R$, $q\in (1,\infty)$ satisfy for all $x,y\in \R$ that
$\lvert f(x)-f(y)\rvert\leq L\lvert x-y\rvert$.
Then there exist $c\in (0,\infty)$,
$(f_\varepsilon)_{\varepsilon\in (0,1)}\subseteq C(\R,\R)$ such that for all $\varepsilon\in (0,1) $, $x,y\in \R$
we have that 
$\lvert f_\varepsilon(x)-f_\varepsilon(y)\rvert\leq L\lvert x-y\rvert$, 
$\lvert f_\varepsilon(x)-f(x)\rvert\leq \varepsilon(1+\lvert x\rvert^q)$,
 and
$ f_\varepsilon\in \mathcal{R}(\{\Phi\in \mathbf{N}\colon \dim(\mathcal{D}(\Phi ))=3
, \supnorm{\mathcal{D}(\Phi )}\leq c\varepsilon^{-c}
\})$.
\end{lemma}
\begin{proof}
[Proof of \cref{p10a}]See \cite[Corollary~4.14]{AJK+2023}.
\end{proof}

\subsection{Properties of operations associated to DNNs}
\begin{setting}\label{m08}Assume \cref{m07}.
Let $\odot \colon \mathbf{D}\times \mathbf{D} \to\mathbf{D} $ satisfy 
for all $H_1,H_2\in \N$, $ \alpha=(\alpha_0,\alpha_1,\ldots,\alpha_{H_1},\alpha_{H_1+1})\in\N^{H_1+2}$, $\beta=(\beta_0,\beta_1,\ldots,\beta_{H_2},\beta_{H_2+1})\in\N^{H_2+2}$
that $\alpha\odot \beta
=(\alpha_0,\ldots,\alpha_{H_1},\beta_1,\ldots,\beta_{H_2+1})$. Let $\boxplus \colon \mathbf{D}\times \mathbf{D} \to\mathbf{D}  $ satisfy
for all $H\in \N$, 
$\alpha= (\alpha_0,\alpha_1,\ldots,\alpha_{H},\alpha_{H+1})\in \N^{H+2}$,
$\beta= (\beta_0,\beta_1,\beta_2,\ldots,\beta_{H},\beta_{H+1})\in \N^{H+2}$
that
$
\alpha \boxplus \beta =(\alpha_0,\alpha_1+\beta_1,\ldots,\alpha_{H}+\beta_{H},\beta_{H+1})\in \N^{H+2}
$.  
\end{setting}

\begin{lemma}\label{m09}
Assume \cref{m08} and let $\alpha,\beta,\gamma\in \mathbf{D}$. Then 
$(\alpha\odot\beta)\odot\gamma
=\alpha\odot(\beta\odot\gamma)$.
\end{lemma}\begin{proof}[Proof of \cref{m09}]
Straightforward.
\end{proof}
\begin{lemma}\label{k43b}Assume \cref{m08},
let $H,k,l \in \N$, and let $\alpha,\beta,\gamma\in \left( \{k\}\times \N^{H} \times \{l\}\right)$.
Then
\begin{enumerate}[(i)]
\item we have that $\alpha\boxplus\beta\in \left(\{k\}\times \N^{H} \times \{l\}\right)$,
\item we have that $\beta\boxplus \gamma\in \left(\{k\}\times \N^{H} \times \{l\}\right)$, and 
\item we have that $(\alpha\boxplus\beta)\boxplus \gamma
= \alpha\boxplus(\beta\boxplus \gamma)$.
\end{enumerate}
\end{lemma}
\begin{proof}
[Proof of \cref{k43b}]Straightforward. We could use the proof of \cite[Lemma 3.4]{hutzenthaler2020proof}.
\end{proof}
\cref{b15} below is later important to estimate the maximum norm of the vector of layer dimensions of DNNs.
\begin{lemma}[Triangle inequality]\label{b15}
Assume \cref{m08},
let $H,k,l \in \N$, and let $\alpha,\beta\in \{k\}\times \N^{H} \times \{l\}$.
Then we have that
$\supnorm{\alpha\boxplus\beta}\leq\supnorm{\alpha}+
\supnorm{\beta} $.
\end{lemma}
\begin{proof}[Proof of \cref{b15}]We can use the proof of 
\cite[Lemma~3.5]{hutzenthaler2020proof}.
\end{proof}
\cref{p01} below shows that affine transformations of DNNs can be represented by DNNs with the same vector of layer dimensions.
\begin{lemma}[DNNs for affine transformations]\label{p01}
Assume \cref{m07} and let $d,m\in \N$, 
$\lambda\in \R$,
$b\in\R^d$, $a\in\R^m$, $\Psi\in\mathbf{N}$ satisfy that $\mathcal{R}(\Psi)\in C(\R^d,\R^m)$. Then we have that
$
\lambda\left((\mathcal{R}(\Psi))(\cdot +b)+a\right)\in \mathcal{R}(\{\Phi\in\mathbf{N}\colon \mathcal{D}(\Phi)=\mathcal{D}(\Psi)\}).
$
\end{lemma}
\begin{proof}
[Proof of \cref{p01}]We can use the proof of 
{\cite[Lemma 3.7]{hutzenthaler2020proof}}, which also works for other activation functions than only ReLU.
\end{proof}
\cref{m11b} below shows that compositions of DNN functions can be represented by DNNs.
\begin{lemma}[Composition of functions generated by DNNs]\label{m11b}
Assume~\cref{m08} and let $d_1,d_2,d_3\in\N$, $f\in C(\R^{d_2},\R^{d_3})$, $g\in C(  \R^{d_1}, \R^{d_2}) $, 
$\alpha,\beta\in \mathbf{D}$ satisfy that
$f\in \mathcal{R}(\{\Phi\in \mathbf{N}\colon \mathcal{D}(\Phi)=\alpha\})$
and
$g\in \mathcal{R}(\{\Phi\in \mathbf{N}\colon \mathcal{D}(\Phi)=\beta\})$.
Then we have
that $(f\circ g)\in \mathcal{R}(\{\Phi\in \mathbf{N}\colon \mathcal{D}(\Phi)=\alpha\odot\beta\})$.
\end{lemma}
\begin{proof}
[Proof of \cref{m11b}]See \cite[Proposition~2.1.2]{JKvW2023}, which especially works for general activation functions.
\end{proof}
\cref{b01b} below shows that sums of DNNs of the same length can be represented by DNNs. In order to represent sums of DNNs with different lengths we note that the identity function can be represented as DNNs. We then take the composition of a DNN function with the identity to change  its length. This is one of the main techniques in the proof of \cref{l01,k04c}.
\begin{lemma}[Sum of DNNs of the same length]
\label{b01b}
Assume \cref{m08} and let $p,q,M,H\in \N$,  $\alpha_1,\alpha_2,\ldots,\alpha_M\in\R$,
 $k_i\in \mathbf{D} $,
$g_i\in C(\R^{p},\R^{q})$,
$i\in [1,M]\cap\N$, satisfy 
for all $i\in [1,M]\cap\N$
that $ \dim(k_i)=H+2$ and
$g_i\in 
\mathcal{R}(\{\Phi\in\mathbf{N}\colon \mathcal{D}(\Phi)=k_i\}).
$
Then
we have that 
$
\sum_{i=1}^{M}\alpha_i g_i
\in\mathcal{R}\left(\left\{ \Phi\in\mathbf{N}\colon
\mathcal{D}(\Phi)=\boxplus_{i=1}^Mk_i\right\}\right).
$
\end{lemma}
\begin{proof}
[Proof of \cref{b01b}]We can use the proof of 
{\cite[Lemma 3.9]{hutzenthaler2020proof}}, which can be extended to other activation functions than only for ReLU. See also \cite[Lemma~2.4.11]{JKvW2023}.
\end{proof}
\subsection{DNN representation of our Euler-Maruyama approximations}
In \cref{l01} below we prove that Euler-Maruyama approximations  can be represented by DNNs if their 
 coefficients  are represented by DNNs and if the identity in $\R$ can be represented by a DNN (see \eqref{c39}).
\begin{setting}\label{s02}
 Assume \cref{m07}. 
Let ${\mathfrak{d}}\in \N$, $\mathfrak{n}_{1,{\mathfrak{d}}}=(1,{\mathfrak{d}},1)\in \mathbf{D}$ satisfy that 
\begin{align}\label{c39}\mathrm{Id}_\R\in \mathcal{R}(\{\Phi\in \N\colon \mathcal{D}(\Phi)=\mathfrak{n}_{1,{\mathfrak{d}}}\}) .
\end{align}
Let $T\in(0,\infty)$,
 $K\in \N$. Let $\rdown{\cdot}_K\colon \R\to\R$ satisfy for all $t\in \R$ that 
$\rdown{t}_K=\max( \{0,\frac{T}{K},\frac{2T}{K},\ldots,T\} \cap (    (-\infty,t)\cup\{0\} )  )$. For every $d\in \N$, $\varepsilon\in (0,1)$,
$v\in\R^d$
 let
$\mu_\varepsilon^d\in C(\R^d,\R^d)  $,
$\sigma_\varepsilon^d\in C(\R^d,\R^{d\times d})  $,
$\Phi_{\mu_\varepsilon^d},\Phi_{\sigma_\varepsilon^d,v} \in \mathbf{N}
$
satisfy that
$\mu_\varepsilon^d=\mathcal{R}(\Phi_{\mu_\varepsilon^d}) $, 
$\sigma_\varepsilon^d (\cdot)v=
\mathcal{R}(\Phi_{\sigma_\varepsilon^d,v})
$.  Assume for all
 $d\in \N$, $\varepsilon\in (0,1)$,
$v\in\R^d$ that
$\mathcal{D}(\Phi_{\sigma_\varepsilon^d,v})=\mathcal{D}(\Phi_{\sigma_\varepsilon^d,0})$.
Let $(\Omega,\mathcal{F},\P)$ be a probability space. 
For every $d\in \N$ let $W^{d,\theta}=(W^{d,\theta}_t)_{t\in[0,T]}\colon [0,T]\times \Omega\to \R^d$, $\theta\in \Theta$, be independent standard Brownian motions.
For every $ d\in\N $,
$\theta\in \Theta$,
$x\in\R^d$,
$\varepsilon\in (0,1)$, 
$t\in[0,T)$
let 
$(X^{d,\theta,K,\varepsilon,t,x}_{s})_{s\in [t,T]}$ satisfy that
$ X^{d,\theta,K,\varepsilon,t,x}_{t}=x$
and
\begin{align} \begin{split} 
X^{d,\theta,K,\varepsilon,t,x}_{s}&=x+
\int_{t}^{s}
\mu^d_\varepsilon(
X^{d,\theta,K,\varepsilon,t,x}_{\max\{t,\rdown{u}_K\}}\,
)du
+
\int_{t}^{s}
\sigma^d_\varepsilon(
X^{d,\theta,K,\varepsilon,t,x}_{\max\{t,\rdown{u}_K\}}\,
)dW_u^{d,\theta}.
\end{split}\label{l02}
\end{align}
\end{setting}
\begin{lemma}\label{l01}Assume \cref{s02}.
Let $\omega\in \Omega$.
Then there exists
$(\mathcal{X}^{d,\theta,K,\varepsilon,t}_{s})_{
  d\in\N  ,
 \theta\in \Theta ,
 \varepsilon\in (0,1) ,
  t\in [0,T) ,
 s\in (t,T] 
}\subseteq \mathbf{N}$ such that the following items are true.
\begin{enumerate}[(i)]
\item For all $ d\in\N $,
$\theta\in \Theta$,
$\varepsilon\in (0,1)$,
 $t\in [0,T)$,
$s\in (t,T]$, $x\in \R^d$ we have that
$\mathcal{R}(\mathcal{X}^{d,\theta,K,\varepsilon,t}_{s})\in C(\R^d,\R^d)$
and
$(\mathcal{R}(\mathcal{X}^{d,\theta,K,\varepsilon,t}_{s}))(x)
=X^{d,\theta,K,\varepsilon,t,x}_{s}(\omega)$.
\item For all $ d\in\N $,
$\theta\in \Theta$,
$\varepsilon\in (0,1)$,
 $t_1\in [0,T)$,
$s_1\in (t_1,T]$, 
 $t_2\in [0,T)$,
$s_2\in (t_2,T]$, 
$x\in \R^d$ we have that
$\mathcal{D}(\mathcal{X}^{d,\theta_1,K,\varepsilon,t_1}_{s_1})=
\mathcal{D}(
\mathcal{X}^{d,\theta_2,K,\varepsilon,t_2}_{s_2})
$.
\item For all $ d\in\N $,
$\theta\in \Theta$,
$\varepsilon\in (0,1)$,
 $t\in [0,T)$,
$s\in (t,T]$ we have that
$
\dim(\mathcal{D}(\mathcal{X}^{d,\theta,K,\varepsilon,t}_{s}))=K(
\max\{\dim(\mathcal{D}(\Phi_{\mu^d_\varepsilon})) ,\dim(\mathcal{D}(\Phi_{\sigma^d_\varepsilon,0}))\}
-2)+2
$. 
\item For all $ d\in\N $,
$\theta\in \Theta$,
$\varepsilon\in (0,1)$,
 $t\in [0,T)$,
$s\in (t,T]$ we have that
$ \supnorm{\mathcal{D}(\mathcal{X}^{d,\theta,K,\varepsilon,t}_{s})}\leq 3\max\{ d{\mathfrak{d}},\supnorm{\mathcal{D}(\Phi_{\mu^d_\varepsilon})}
,
\supnorm{
\mathcal{D}(\Phi_{\sigma^d_\varepsilon,0})}\}
.$
\end{enumerate}

\end{lemma}

\begin{proof}[Proof of \cref{l01}]
Throughout this proof let the notation in \cref{m08} be given. Moreover, for every $d,n\in \N$ let
$\mathfrak{n}_{d,{\mathfrak{d}}}=(d,{\mathfrak{d}} d,d)\in \mathbf{D}$ and
$\mathfrak{n}_{d,{\mathfrak{d}}}^{\odot n}= \mathfrak{n}_{d,{\mathfrak{d}}}\odot\ldots\odot\mathfrak{n}_{d,{\mathfrak{d}}}$ ($n$ times). 
Lemmas \ref{a01}, \ref{m11b}, and a simple induction argument show for all $d,n\in \N$ that
\begin{align}
\mathrm{Id}_{\R^d}\in \mathcal{R}(\{\Phi\in \N\colon \mathcal{D}(\Phi)= \mathfrak{n}_{d,{\mathfrak{d}}}^{\odot n}\}), \quad
\mathfrak{n}_{d,{\mathfrak{d}}}^{\odot n}=(d,{\mathfrak{d}} d,\ldots,{\mathfrak{d}} d, d)\in \R^{n+2}.
\label{a09}
\end{align}
This, \cref{m11b}, and the definition of $\odot$ show for all $d,n\in \N$, $\varepsilon\in (0,1)$ that
\begin{align}
\mu_\varepsilon^d\in \mathcal{R}(\{\Phi\in \N\colon\mathcal{D}(\Phi )= \mathcal{D}(\Phi_{\mu^d_\varepsilon})\odot \mathfrak{n}_{d,{\mathfrak{d}}}^{\odot n} \})\label{e01}
\end{align}
and 
\begin{align}
 \begin{split} \dim ( \mathcal{D}(\Phi_{\mu^d_\varepsilon})\odot \mathfrak{n}_{d,{\mathfrak{d}}}^{\odot n})&=
\dim (\mathcal{D}(\Phi_{\mu^d_\varepsilon}))
+ \dim (\mathfrak{n}_{d,{\mathfrak{d}}}^{\odot n})-2
\\
&=\dim (\mathcal{D}(\Phi_{\mu^d_\varepsilon}))+ n+2-2\\
&
=\dim (\mathcal{D}(\Phi_{\mu^d_\varepsilon}))+ n.
\end{split}
\end{align}
Similarly, for all $d,n\in \N$, $\varepsilon\in (0,1)$, $v\in \R^d$ we have that
\begin{align}
\sigma_\varepsilon^d(\cdot )v\in \mathcal{R}(\{\Phi\in \N\colon\mathcal{D}(\Phi )= \mathcal{D}(\Phi_{\sigma^d_\varepsilon,0})\odot \mathfrak{n}_{d,{\mathfrak{d}}}^{\odot n} \})\label{e02}
\end{align}
and 
\begin{align}
 \begin{split} \dim ( \mathcal{D}(\Phi_{\sigma^d_\varepsilon,0})\odot \mathfrak{n}_{d,{\mathfrak{d}}}^{\odot n})
=\dim (\mathcal{D}(\Phi_{\sigma^d_\varepsilon,0}))+ n.
\end{split}
\end{align}
This and 
\eqref{a09}--\eqref{e02} prove that we can assume without lost of generality that
\begin{align}
\dim (\mathcal{D}(\Phi_{\mu^d_\varepsilon}))=\dim (\mathcal{D}(\Phi_{\sigma^d_\varepsilon,0}) ) \label{k02b} 
\end{align}
since otherwise we could change their lengths by taking compositions with identities.
Next, observe that
for all 
$ d\in\N $,
$\theta\in \Theta$,
$x\in\R^d$,
$\varepsilon\in (0,1)$,
$k\in [1,K]\cap\Z$, $t\in [0,T)$,
$s\in [\frac{kT}{K},\frac{(k+1)T}{K}]$ we have that
\begin{align} \label{d01}\begin{split} 
X^{d,\theta,K,\varepsilon,t,x}_{s}(\omega)
&=X^{d,\theta,K,\varepsilon,t,x}_{\max\{t,\frac{kT}{K}\}}(\omega)
+\mu^d_\varepsilon\left(X^{d,\theta,K,\varepsilon,t,x}_{\max\{t,\frac{kT}{K}\}}(\omega)\right)\left(s-\max\{t,\frac{kT}{K}\}\right)\\&\quad+
\sigma^d_\varepsilon\left(X^{d,\theta,K,\varepsilon,t,x}_{\max\{t,\frac{kT}{K}\}}(\omega)\right) \left(W^{d,\theta}_s(\omega)-
W^{d,\theta}_{\max\{t,\frac{kT}{K}\}}(\omega)\right).\end{split}
\end{align}
Next, for every 
$ d\in\N $,
$\theta\in \Theta$,
$x\in\R^d$,
$\varepsilon\in (0,1)$,
$k\in [1,K]\cap\Z$, $t\in [0,T)$,
$s\in (t,T]$
let $ J_k(s)\in \R$,
$
\phi^{d,\theta,K,\varepsilon}_{t,s,k}(x)\in \R^d$
 satisfy that
\begin{align} \begin{split} 
J_k(s)&=\max\{ t,\frac{(k-1)T}{K} \}\1_{ [0,\max\{ t,\frac{(k-1)T}{K} \}    ] }(s)\\&\quad 
+s\1_{ (\max\{ t,\frac{(k-1)T}{K} \},  \max\{ t,\frac{kT}{K} \} 
 ] }(s)+ \max\{ t,\frac{kT}{K} \}
\1_{ (  \max\{ t,\frac{kT}{K} \},T ]}(s)\end{split}
\end{align} 
and
\begin{align} \label{d02}\begin{split} 
\phi^{d,\theta,K,\varepsilon}_{t,s,k}(x)
&=x
+\mu^d_\varepsilon(x)\left(J_k(s)-\max\{t,\frac{(k-1)T}{K}\}\right)+
\sigma^d_\varepsilon(x)\left(W^{d,\theta}_{J_k(s)}(\omega)-
W^{d,\theta}_{\max\{t,\frac{(k-1)T}{K}\}}(\omega)\right).
\end{split}
\end{align}
Next, for every 
$ d\in\N $,
$\theta\in \Theta$,
$\varepsilon\in (0,1)$,
$k\in [1,K]\cap\Z$, $t\in [0,T)$,
$s\in (t,T]$ let 
\begin{align}\label{d03}
\psi^{d,\theta,K,\varepsilon}_{t,s,k}=
\phi^{d,\theta,K,\varepsilon}_{t,s,k}
\circ
\phi^{d,\theta,K,\varepsilon}_{t,s,k-1}\circ\ldots
\circ
\phi^{d,\theta,K,\varepsilon}_{t,s,1}
.
\end{align}
Note that for all 
$ d\in\N $,
$\theta\in \Theta$,
$\varepsilon\in (0,1)$, 
$k\in [1,K-1]\cap\Z$, $s\in [0,\max\{t,\frac{(k-1)T}{K}\}]$ we have that $\phi^{d,\theta,K,\varepsilon}_{t,s,k}=\mathrm{Id}_{\R^d} $. This ensures for all
$ d\in\N $,
$\theta\in \Theta$,
$\varepsilon\in (0,1)$, 
 $k\in [1,K-1]\cap\Z$, $n\in [k+1,K]\cap\Z$,
$s\in [0,\max\{t,\frac{kT}{K}\}]$ that 
$\psi^{d,\theta,K,\varepsilon}_{t,s,k}=
\psi^{d,\theta,K,\varepsilon}_{t,s,n}$ and in particular 
$\psi^{d,\theta,K,\varepsilon}_{t,s,k}=
\psi^{d,\theta,K,\varepsilon}_{t,s,K}$.
Observe that for all 
$ d\in\N $,
$\theta\in \Theta$,
$\varepsilon\in (0,1)$, 
$k\in [1,K]\cap\Z$, $s\in [0,\max\{t,\frac{kT}{K}\}]$, $x\in \R^d$ that 
$\psi^{d,\theta,K,\varepsilon}_{t,s,k}(x)=X^{d,\theta,K,\varepsilon,t,x}_{s}(\omega)$.
Therefore, for all
$ d\in\N $,
$\theta\in \Theta$,
$\varepsilon\in (0,1)$,
$s\in[0,T]$, $x\in \R^d$ we have that
$\psi^{d,\theta,K,\varepsilon}_{t,s,K}(x)=X^{d,\theta,K,\varepsilon,t,x}_{s}$, i.e.,
\begin{align}\label{k03b}
X^{d,\theta,K,\varepsilon,t,x}_{s}(\omega)=
\phi^{d,\theta,K,\varepsilon}_{t,s,K}
\circ
\phi^{d,\theta,K,\varepsilon}_{t,s,K-1}\circ\ldots
\circ
\phi^{d,\theta,K,\varepsilon}_{t,s,1} (x).
\end{align}
Next, 
\eqref{a09},
\eqref{d02},
\eqref{k02b}, and \cref{b01b}
show for all $ d\in\N $,
$\theta\in \Theta$,
$\varepsilon\in (0,1)$,
$k\in [1,K]\cap\Z$, $t\in [0,T)$,
$s\in (t,T]$ that
\begin{align}
\phi^{d,\theta,K,\varepsilon}_{t,s,k}(\cdot)\in \mathcal{R}\left(\left\{\Phi\in \mathbf{N}\colon 
\mathcal{D}(\Phi)=
\mathfrak{n}_{d,{\mathfrak{d}}}^{\odot\dim(\Phi_{\mu^d_\varepsilon}) -2 }\boxplus \mathcal{D}(\Phi_{\mu^d_\varepsilon})
\boxplus \mathcal{D}(\Phi_{\sigma^d_\varepsilon,0})
\right\}\right).
\end{align}
This, \eqref{k03b}, and \cref{m11b} show that there exists
$(\mathcal{X}^{d,\theta,K,\varepsilon,t}_{s})_{
  d\in\N  ,
 \theta\in \Theta ,
 \varepsilon\in (0,1) ,
  t\in [0,T) ,
 s\in (t,T] 
}\subseteq \mathbf{N}$
such that
for all
$ d\in\N $,
$\theta\in \Theta$,
$\varepsilon\in (0,1)$, 
 $t\in [0,T)$,
$s\in (t,T]$, $x\in \R^d$ we have that
\begin{align} \begin{split} 
&
\mathcal{D}(\mathcal{X}^{d,\theta,K,\varepsilon,t}_{s})=
\operatorname*{\odot}_{k=1}^K\left[
\mathfrak{n}_{d,{\mathfrak{d}}}^{\odot\dim(\Phi_{\mu^d_\varepsilon}) -2 }\boxplus \mathcal{D}(\Phi_{\mu^d_\varepsilon})
\boxplus \mathcal{D}(\Phi_{\sigma^d_\varepsilon,0})
\right],\\
& (\mathcal{R}(\mathcal{X}^{d,\theta,K,\varepsilon,t}_{s}))(x)
=X^{d,\theta,K,\varepsilon,t,x}_{s}(\omega).
\label{k08b}\end{split}
\end{align}
This, the definition of $\odot$, and an induction argument show that for all
$ d\in\N $,
$\theta\in \Theta$,
$\varepsilon\in (0,1)$,
 $t\in [0,T)$,
$s\in (t,T]$, $x\in \R^d$ we have that
\begin{align}
\dim(\mathcal{D}(\mathcal{X}^{d,\theta,K,\varepsilon,t}_{s}))=K(\dim(\mathcal{D}(\Phi_{\mu^d_\varepsilon}))-2)+2.
\end{align}
Next, \eqref{k08b}, the definition of $\odot$, the triangle inequality (cf. \cref{b15}), and \eqref{a09}
show that
  for all
$ d\in\N $,
$\theta\in \Theta$,
$\varepsilon\in (0,1)$,
 $t\in [0,T)$,
$s\in (t,T]$, $x\in \R^d$ we have that
\begin{align} \begin{split} \supnorm{\mathcal{D}(\mathcal{X}^{d,\theta,K,\varepsilon,t}_{s})}&=
\supnorm{\operatorname*{\odot}_{k=1}^K\left[
\mathfrak{n}_{d,{\mathfrak{d}}}^{\odot\dim(\Phi_{\mu^d_\varepsilon}) -2 }\boxplus \mathcal{D}(\Phi_{\mu^d_\varepsilon})
\boxplus \mathcal{D}(\Phi_{\sigma^d_\varepsilon,0})
\right]
}\\
&\leq
\supnorm{
\mathfrak{n}_{d,{\mathfrak{d}}}^{\odot\dim(\Phi_{\mu^d_\varepsilon}) -2 }\boxplus \mathcal{D}(\Phi_{\mu^d_\varepsilon})
\boxplus \mathcal{D}(\Phi_{\sigma^d_\varepsilon,0})
}
\\
&\leq 3\max\left\{d{\mathfrak{d}}, \supnorm{\mathcal{D}(\Phi_{\mu^d_\varepsilon})}
,
\supnorm{
\mathcal{D}(\Phi_{\sigma^d_\varepsilon,0})}
\right\}.
\end{split}
\end{align}
The proof of  \cref{l01} is thus completed.
\end{proof}
\begin{lemma}\label{k04c}
Assume \cref{s02}. For every
$d\in \N$, $\varepsilon\in (0,1)$
let $f_\varepsilon\in C(\R,\R)$,
$g^d_\varepsilon\in C(\R^d,\R)$,
$\Phi_{f_\varepsilon}, \Phi_{g^d_\varepsilon}\in \mathbf{N}$
satisfy that
$ \mathcal{R}(\Phi_{f_\varepsilon})=f_\varepsilon $ and
$ \mathcal{R}(\Phi_{g^d_\varepsilon})=g^d_\varepsilon $.
Let $\mathfrak{t}^\theta\colon [0,1]\to\R^d$, $\theta\in \Theta$, be independent random variables which satisfy for all $t\in [0,1]$ that
$\P(\mathfrak{t}^0\leq t)=t$. Assume that $(W^{d,\theta})_{d\in \N,\theta\in \Theta}$ and $(\mathfrak{t}^\theta)_{\theta\in\Theta}$ are independent.
For every $d\in \N$, $\varepsilon\in (0,1)$
let 
$U^{d,\theta,K,\varepsilon}_{n,m}\colon [0,T]\times\R^d\times \Omega\to \R$, 
$\theta\in \Theta$,  
$n,m\in \Z$, 
satisfy for all 
$\theta\in \Theta$,  
$n\in \N_0$, $m\in \N$, $t\in [0,T]$, $x\in \R^d$ that
\begin{align} \begin{split} &U^{d,\theta,K,\varepsilon}_{n,m}(t,x)=
\frac{\1_{\N}(n)}{m^n}\sum_{i=1}^{m^n}g^d_\varepsilon\Bigl( X^{ d,(\theta,0,-i), K,\varepsilon,t,x}_{T} \Bigr)\\
&
+\sum_{\ell=0}^{n-1}
\frac{(T-t)}{m^{n-\ell}}
\sum_{i=1}^{m^{n-\ell}}\Bigl(
f_\varepsilon\circ
U^{d,(\theta,\ell,i),K,\varepsilon}_{\ell,m}
-\1_{\N}(\ell)
f_\varepsilon\circ
U^{d,(\theta,-\ell,i),K,\varepsilon}_{\ell-1,m}\Bigr)
\Bigl(t+(T-t)\mathfrak{t}^{(\theta,\ell,i)}, 
 X^{ d,(\theta,\ell,i), K,\varepsilon,t,x}_{t+(T-t)\mathfrak{t}^{(\theta,\ell,i)}} 
\Bigr).\end{split}\label{k04d}
\end{align}
For every $d\in \N$, $\varepsilon\in (0,1)$ let 
\begin{align}
L_{d,\varepsilon}=K(
\max\{\dim(\mathcal{D}(\Phi_{\mu^d_\varepsilon})) ,\dim(\mathcal{D}(\Phi_{\sigma^d_\varepsilon,0}))\}
-2)+2.\label{k10}
\end{align}
Let
$(c_{d,\varepsilon})_{d\in\N,\varepsilon\in (0,1)}\subseteq \R $ satisfy 
for all 
$d\in\N$, $\varepsilon\in (0,1)$ 
that
\begin{align}
c_{d,\varepsilon} \geq 
3\max\!\left\{ d\mathfrak{d},
\supnorm{\mathcal{D}(\Phi_{f_\varepsilon})},
\supnorm{\mathcal{D}(\Phi_{g^d_\varepsilon})},
\supnorm{\mathcal{D}(\Phi_{\mu^d_\varepsilon})}
,
\supnorm{
\mathcal{D}(\Phi_{\sigma^d_\varepsilon,0})}\right\}
.\label{c01}
\end{align}
Let $\omega\in \Omega$.
Then for all
$m\in \N$,
$d\in \N$,
$n\in \N_0$,
$\varepsilon\in (0,1)$ there exists
$ (\Phi^{d,\theta,K,\varepsilon}_{n,m,t} )_{t\in[0,T],\theta\in \Theta}\subseteq \mathbf{N}$ such that the following items are true.
\begin{enumerate}[(i)]
\item\label{h01} We have for all $t_1,t_2\in [0,T]$, $\theta_1,\theta_2\in \Theta $
that $\mathcal{D}(\Phi^{d,\theta_1,K,\varepsilon}_{n,m,t_1})
=\mathcal{D}(\Phi^{d,\theta_2,K,\varepsilon}_{n,m,t_2})
$.
\item\label{h02} We have for all
$t\in [0,T]$, $\theta\in \Theta$ that
\begin{align} \begin{split} 
\dim(\mathcal{D}( 
\Phi^{d,\theta,K,\varepsilon}_{n,m,t}
 ))
&=n\left(\dim\!\left(\mathcal{D}(\Phi_{f_\varepsilon})\right)
+L_{d,\varepsilon}-4 \right)
+\dim \! \left(\mathcal{D}(\Phi_{g^d_\varepsilon})\right)
+L_{d,\varepsilon}-2.
\end{split}\end{align}
\item\label{h03} We have for all
$t\in [0,T]$, $\theta\in \Theta$ that
$\supnorm{\mathcal{D}(\Phi^{d,\theta,K,\varepsilon}_{n,m,t})}\leq c_{d,\varepsilon}(3m)^n$.
\item \label{h04} We have for all
$t\in [0,T]$, $\theta\in \Theta$, $x\in \R^d$ that
$U^{d,\theta,K,\varepsilon}_{n,m}(t,x,\omega)=(\mathcal{R}(\Phi^{d,\theta,K,\varepsilon}_{n,m,t}))(x)$.
\end{enumerate}
\end{lemma}
\begin{proof}[Proof of \cref{k04c}]
Throughout this proof let the notation in \cref{m08} be given and let $d,m\in \N$, 
$\varepsilon\in (0,1)$ be fixed.
Moreover, for every $n\in \N$ let
$\mathfrak{n}_{1,{\mathfrak{d}}}^{\odot n}= \mathfrak{n}_{1,{\mathfrak{d}}}\odot\ldots\odot\mathfrak{n}_{1,{\mathfrak{d}}}$ ($n$ times). 
Lemmas \ref{a01} and \ref{m11b} and a simple induction argument show for all $n\in \N$ that
\begin{align}
\mathrm{Id}_{\R}\in \mathcal{R}(\{\Phi\in \N\colon \mathcal{D}(\Phi)= \mathfrak{n}_{1,{\mathfrak{d}}}^{\odot n}\}), \quad
\mathfrak{n}_{1,{\mathfrak{d}}}^{\odot n}=(1,{\mathfrak{d}} ,\ldots,{\mathfrak{d}} , 1)\in \R^{n+2}.
\label{a09b}
\end{align}
Furthermore, \cref{l01} shows that  there exists
$(\mathcal{X}^{d,\theta,K,\varepsilon,t}_{s})_{
  d\in\N  ,
 \theta\in \Theta ,
 \varepsilon\in (0,1) ,
  t\in [0,T) ,
 s\in (t,T] 
}\subseteq \mathbf{N}$ such that the following items are true.
\begin{enumerate}[(A)]
\item For all $ d\in\N $,
$\theta\in \Theta$,
$\varepsilon\in (0,1)$,
 $t\in [0,T)$,
$s\in (t,T]$, $x\in \R^d$ we have that
$\mathcal{R}(\mathcal{X}^{d,\theta,K,\varepsilon,t}_{s})\in C(\R^d,\R^d)$
and
\begin{align}
(\mathcal{R}(\mathcal{X}^{d,\theta,K,\varepsilon,t}_{s}))(x)
=X^{d,\theta,K,\varepsilon,t,x}_{s}(\omega).\label{k11}
\end{align}
\item For all $ d\in\N $,
$\theta\in \Theta$,
$\varepsilon\in (0,1)$,
 $t_1\in [0,T)$,
$s_1\in (t_1,T]$, 
 $t_2\in [0,T)$,
$s_2\in (t_2,T]$, 
$x\in \R^d$ we have that
\begin{align}
\mathcal{D}(\mathcal{X}^{d,\theta_1,K,\varepsilon,t_1}_{s_1})=
\mathcal{D}(
\mathcal{X}^{d,\theta_2,K,\varepsilon,t_2}_{s_2})
.\label{k11b}
\end{align}
\item For all $ d\in\N $,
$\theta\in \Theta$,
$\varepsilon\in (0,1)$,
 $t\in [0,T)$,
$s\in (t,T]$ we have that
\begin{align}
\dim(\mathcal{D}(\mathcal{X}^{d,\theta,K,\varepsilon,t}_{s}))=K(
\max\{\dim(\mathcal{D}(\Phi_{\mu^d_\varepsilon})) ,\dim(\mathcal{D}(\Phi_{\sigma^d_\varepsilon,0}))\}
-2)+2
.\label{k10b}
\end{align} 
\item For all $ d\in\N $,
$\theta\in \Theta$,
$\varepsilon\in (0,1)$,
 $t\in [0,T)$,
$s\in (t,T]$ we have that
\begin{align}
 \supnorm{\mathcal{D}(\mathcal{X}^{d,\theta,K,\varepsilon,t}_{s})}\leq 3\max\!\left\{ d{\mathfrak{d}},\supnorm{\mathcal{D}(\Phi_{\mu^d_\varepsilon})}
,
\supnorm{
\mathcal{D}(\Phi_{\sigma^d_\varepsilon,0})}\right\}
.\label{k16}
\end{align}
\end{enumerate}
By \eqref{k10} and \eqref{k10b}
for all 
$\theta\in \Theta$,
$\varepsilon\in (0,1)$,
 $t\in [0,T)$,
$s\in (t,T]$ we have that
\begin{align}
\dim(\mathcal{D}(\mathcal{X}^{d,\theta,K,\varepsilon,t}_{s}))=L_{d,\varepsilon}
.\label{k10c}
\end{align} 
We will prove the result by induction. First, the base case is true since the zero function can be represented by DNN with arbitrary number of hidden layer.
For the induction step 
$\N_0\ni n\mapsto n+1\in \N$
let 
$n\in \N_0$ and assume that there exists
$ (\Phi^{d,\theta,K,\varepsilon}_{\ell,m,t} )_{t\in[0,T],\theta\in \Theta}\subseteq \mathbf{N}$, $\ell\in [0,n]\cap\Z$, such that the following items are true.
\begin{enumerate}[(A)]
\item We have for all $t_1,t_2\in [0,T]$, $\theta_1,\theta_2\in \Theta $, $\ell\in [0,n]\cap\Z$
that 
\begin{align}\label{h01b}
\mathcal{D}(\Phi^{d,\theta_1,K,\varepsilon}_{\ell,m,t_1})
=\mathcal{D}(\Phi^{d,\theta_2,K,\varepsilon}_{\ell,m,t_2})
.
\end{align}
\item We have for all
$t\in [0,T]$, $\theta\in \Theta$, $\ell\in [0,n]\cap\Z$ that
\begin{align}\label{h02b} \begin{split} 
\dim(\mathcal{D}( 
\Phi^{d,\theta,K,\varepsilon}_{\ell,m,t}
 ))
&=\ell\left(\dim\!\left(\mathcal{D}(\Phi_{f_\varepsilon})\right)
+L_{d,\varepsilon}-4 \right)
+\dim \! \left(\mathcal{D}(\Phi_{g^d_\varepsilon})\right)
+L_{d,\varepsilon}-2.
\end{split}\end{align}
\item We have for all
$t\in [0,T]$, $\theta\in \Theta$, $\ell\in [0,n]\cap\Z$ that
\begin{align}\label{h03b}
\supnorm{\mathcal{D}(\Phi^{d,\theta,K,\varepsilon}_{\ell,m,t})}\leq c_{d,\varepsilon}(3m)^\ell.
\end{align}\item  We have for all
$t\in [0,T]$, $\theta\in \Theta$, $x\in \R^d$, $\ell\in [0,n]\cap\Z$ that
\begin{align}\label{h04b}
U^{d,\theta,K,\varepsilon}_{\ell,m}(t,x,\omega)=(\mathcal{R}(\Phi^{d,\theta,K,\varepsilon}_{\ell,m,t}))(x).
\end{align}\end{enumerate}
Next, \cref{m11b}, \eqref{a09b}, the fact that $g^d_\varepsilon=\mathcal{R}(\Phi_{g^d_\varepsilon})$
prove for all
$\theta\in \Theta$, $i\in [1,m^{n+1}]\cap\Z$,
$t\in [0,T]$ that
\begin{align} \begin{split} 
&g^d_\varepsilon\Bigl( X^{ d,(\theta,0,-i), K,\varepsilon,t,\cdot}_{T} \Bigr)
=\mathrm{Id}_\R \!\left(
g^d_\varepsilon\Bigl( X^{ d,(\theta,0,-i), K,\varepsilon,t,\cdot}_{T} \Bigr)\right)\\
&
\in \mathcal{R}\!\left(\left\{
\Phi\in \mathbf{N}\colon 
\mathcal{D}(\Phi)=
\mathfrak{n}_{1,{\mathfrak{d}}}^{\odot (n+1)(\dim(\mathcal{D}(\Phi_{f_\varepsilon}))+{L}-4)}
\odot\mathcal{D} (\Phi_{g^d_\varepsilon})
\odot\mathcal{D}(\mathcal{X}^{d,0,K,\varepsilon,0}_{T})
\right\}\right)
\end{split}\label{k12}
\end{align}
In addition, the definition of $\odot$, \eqref{a09b}, and \eqref{k10c}
imply that
\begin{align} \begin{split} 
&\dim \!\left(\mathfrak{n}_{1,{\mathfrak{d}}}^{\odot (n+1)(\dim(\mathcal{D}(\Phi_{f_\varepsilon}))+{L_{d,\varepsilon}}-4)}
\odot\mathcal{D} (\Phi_{g^d_\varepsilon})
\odot\mathcal{D}(\mathcal{X}^{d,0,K,\varepsilon,0}_{T})
\right)\\
&=\dim \!\left(\mathfrak{n}_{1,{\mathfrak{d}}}^{\odot (n+1)(\dim(\mathcal{D}(\Phi_{f_\varepsilon}))+{L_{d,\varepsilon}}-4)}\right)
+\dim \!\left(\mathcal{D} (\Phi_{g^d_\varepsilon})\right)
+\dim \!\left(\mathcal{D}(\mathcal{X}^{d,0,K,\varepsilon,0}_{T})\right)-4\\
&=(n+1)(\dim(\mathcal{D}(\Phi_{f_\varepsilon}))+{L_{d,\varepsilon}}-4)+2+
\dim \!\left(\mathcal{D} (\Phi_{g^d_\varepsilon})\right)+L_{d,\varepsilon}-4
\\
&=(n+1)(\dim(\mathcal{D}(\Phi_{f_\varepsilon}))+{L_{d,\varepsilon}}-4)+
\dim \!\left(\mathcal{D} (\Phi_{g^d_\varepsilon})\right)+L_{d,\varepsilon}-2.
\end{split}
\end{align}
Furthermore, \cref{m11b}, the fact that
$f_\varepsilon^d=\mathcal{R}(\Phi_{f^d_\varepsilon})$,  \eqref{h01b},
\eqref{h04b}, \eqref{k11b}, and \eqref{k11}
show for all $i\in [1,m]$, $\theta\in \Theta$, $t\in [0,T]$ that
\begin{align} \begin{split} 
&\Bigl(
f_\varepsilon\circ
U^{d,(\theta,n,i),K,\varepsilon}_{n,m}
\Bigr)
\Bigl(t+(T-t)\mathfrak{t}^{(\theta,\ell,i)}(\omega), 
 X^{ d,(\theta,\ell,i), K,\varepsilon,t,\cdot }_{t+(T-t)\mathfrak{t}^{(\theta,\ell,i)}(\omega)} (\omega)
\Bigr)\\
&\in 
\mathcal{R}
\left(\left\{
\Phi\in\mathbf{N}\colon \mathcal{D}(\Phi)=
\mathcal{D}(\Phi_{f_\varepsilon}) \odot 
\mathcal{D}(\Phi^{d,0,K,\varepsilon}_{n,m,0})
\odot
\mathcal{D}(\mathcal{X}^{d,0,K,\varepsilon,0}_{T})
\right\}\right).\end{split}
\end{align}
Moreover,
the definition of $\odot$, \eqref{h02b}, and \eqref{k10c}
show that
\begin{align} \begin{split} 
&\dim \!\left(
\mathcal{D}(\Phi_{f_\varepsilon}) \odot 
\mathcal{D}(\Phi^{d,0,K,\varepsilon}_{n,m,0})
\odot
\mathcal{D}(\mathcal{X}^{d,0,K,\varepsilon,0}_{T})\right)
\\
&=\dim\!\left(\mathcal{D}(\Phi_{f_\varepsilon})\right)
+\dim\!\left(
\mathcal{D}(\Phi^{d,0,K,\varepsilon}_{n,m,0})\right)
+
\dim\!\left(
\mathcal{D}(\mathcal{X}^{d,0,K,\varepsilon,0}_{T})\right)-4\\
&=\dim\!\left(\mathcal{D}(\Phi_{f_\varepsilon})\right)
+
n\left(\dim\!\left(\mathcal{D}(\Phi_{f_\varepsilon})\right)+L_{d,\varepsilon}-4 \right)
+\dim \! \left(\mathcal{D}(\Phi_{g^d_\varepsilon})\right)
+L_{d,\varepsilon}-2
+
L_{d,\varepsilon}-4\\
&=
(n+1)\left(\dim\!\left(\mathcal{D}(\Phi_{f_\varepsilon})\right)+L_{d,\varepsilon}-4 \right)
+\dim \! \left(\mathcal{D}(\Phi_{g^d_\varepsilon})\right)
+L_{d,\varepsilon}-2
.
\end{split}
\end{align}
Furthermore, \cref{m11b},
 the fact that
$f_\varepsilon=\mathcal{R}(\Phi_{f_{\varepsilon}})$, 
\eqref{a09b},
\eqref{h04b}, \eqref{h01b}, \eqref{k11}, and \eqref{k11b}
show for all $\ell\in [0,n-1]\cap\Z$,
$\theta\in \Theta$, $i\in [1,m^{n+1-\ell}]\cap\Z$,
$t\in [0,T]$  that
\begin{align} \begin{split} 
&
\Bigl(
f_\varepsilon\circ
U^{d,(\theta,\ell,i),K,\varepsilon}_{\ell,m}
\Bigr)
\Bigl(t+(T-t)\mathfrak{t}^{(\theta,\ell,i)}(\omega), 
 X^{ d,(\theta,\ell,i), K,\varepsilon,t,\cdot }_{t+(T-t)\mathfrak{t}^{(\theta,\ell,i)}(\omega)} ,\omega
\Bigr)\\&=
\Bigl(
f_\varepsilon\circ\mathrm{Id}_\R\circ
U^{d,(\theta,\ell,i),K,\varepsilon}_{\ell,m}
\Bigr)
\Bigl(t+(T-t)\mathfrak{t}^{(\theta,\ell,i)}(\omega), 
 X^{ d,(\theta,\ell,i), K,\varepsilon,t,\cdot }_{t+(T-t)\mathfrak{t}^{(\theta,\ell,i)}(\omega)} ,\omega
\Bigr)\\
&\in 
\mathcal{R}
\left(\left\{
\Phi\in\mathbf{N}\colon \mathcal{D}(\Phi)=
\mathcal{D}(\Phi_{f_\varepsilon}) \odot 
\mathfrak{n}_{1,{\mathfrak{d}}}^{\odot (n-\ell )(\dim (\mathcal{D}(\Phi_{f_\varepsilon}))
+L_{d,\varepsilon}-4 ) }
\odot 
\mathcal{D}(\Phi^{d,0,K,\varepsilon}_{\ell,m,0})
\odot
\mathcal{D}(\mathcal{X}^{d,0,K,\varepsilon,0}_{T})
\right\}\right).\end{split}
\end{align}
Next, the definition of $\odot$, \eqref{a09b}, \eqref{h02b}, and \eqref{k10c} show for all $\ell \in [0,n-1]\cap\Z$ that
\begin{align} \begin{split} 
&
\dim \!\left(
\mathcal{D}(\Phi_{f_\varepsilon}) \odot 
\mathfrak{n}_{1,{\mathfrak{d}}}^{\odot (n-\ell )(\dim (\mathcal{D}(\Phi_{f_\varepsilon}))
+L_{d,\varepsilon}-4 ) }
\odot 
\mathcal{D}(\Phi^{d,0,K,\varepsilon}_{\ell,m,0})
\odot
\mathcal{D}(\mathcal{X}^{d,0,K,\varepsilon,0}_{T})
\right)\\
&
=\dim\!\left(
\mathcal{D}(\Phi_{f_\varepsilon})\right)
+\dim\!\left(
\mathfrak{n}_{1,{\mathfrak{d}}}^{\odot (n-\ell )(\dim (\mathcal{D}(\Phi_{f_\varepsilon}))
+L_{d,\varepsilon}-4 ) }
\right)
+
\dim \!\left(
\mathcal{D}(\Phi^{d,0,K,\varepsilon}_{\ell,m,0})\right)\\
&\quad 
+\dim\!\left(
\mathcal{D}(\mathcal{X}^{d,0,K,\varepsilon,0}_{T})\right)-6\\
&=
\dim\!\left(
\mathcal{D}(\Phi_{f_\varepsilon})\right)
+
 (n-\ell )(\dim (\mathcal{D}(\Phi_{f_\varepsilon}))
+L_{d,\varepsilon}-4 )+2 
\\
&\quad 
+
\ell\left(\dim\!\left(\mathcal{D}(\Phi_{f_\varepsilon})\right)
+L_{d,\varepsilon}-4 \right)
+\dim \! \left(\mathcal{D}(\Phi_{g^d_\varepsilon})\right)
+L_{d,\varepsilon}-2
+L_{d,\varepsilon}-6\\
&=
\dim\!\left(
\mathcal{D}(\Phi_{f_\varepsilon})\right)
+
 n(\dim (\mathcal{D}(\Phi_{f_\varepsilon}))
+L_{d,\varepsilon}-4 )
\\
&\quad 
+\dim \! \left(\mathcal{D}(\Phi_{g^d_\varepsilon})\right)
+L_{d,\varepsilon}-2
+L_{d,\varepsilon}-4\\
&=
(n+1)\left(\dim\!\left(\mathcal{D}(\Phi_{f_\varepsilon})\right)+L_{d,\varepsilon}-4 \right)
+\dim \! \left(\mathcal{D}(\Phi_{g^d_\varepsilon})\right)
+L_{d,\varepsilon}-2.
\end{split}
\end{align}
Similarly, for all $\ell\in [1,n]\cap\Z$, $\theta\in \Theta$,
$i\in m^{n+1-\ell}$, $t\in [0,T]$ we have that
\begin{align} \begin{split} 
&
\Bigl(
f_\varepsilon\circ
U^{d,(\theta,-\ell,i),K,\varepsilon}_{\ell-1,m}
\Bigr)
\Bigl(t+(T-t)\mathfrak{t}^{(\theta,\ell,i)}(\omega), 
 X^{ d,(\theta,\ell,i), K,\varepsilon,t,\cdot }_{t+(T-t)\mathfrak{t}^{(\theta,\ell,i)}(\omega)}(\omega) ,\omega
\Bigr)\\&=
\Bigl(
f_\varepsilon\circ\mathrm{Id}_\R\circ
U^{d,(\theta,-\ell,i),K,\varepsilon}_{\ell-1,m}
\Bigr)
\Bigl(t+(T-t)\mathfrak{t}^{(\theta,\ell,i)}(\omega), 
 X^{ d,(\theta,\ell,i), K,\varepsilon,t,\cdot }_{t+(T-t)\mathfrak{t}^{(\theta,\ell,i)}(\omega)}(\omega) ,\omega
\Bigr)\\
&\in 
\mathcal{R}
\left(\left\{
\Phi\in\mathbf{N}\colon \mathcal{D}(\Phi)=
\mathcal{D}(\Phi_{f_\varepsilon}) \odot 
\mathfrak{n}_{1,{\mathfrak{d}}}^{\odot (n-\ell+1 )(\dim(\mathcal{D}(\Phi_{f_\varepsilon}))
+L_{d,\varepsilon}-4 )}
\odot 
\mathcal{D}(\Phi^{d,0,K,\varepsilon}_{\ell-1,m,0})
\odot
\mathcal{D}(\mathcal{X}^{d,0,K,\varepsilon,0}_{T})
\right\}\right).\end{split}
\end{align}
and
\begin{align} \begin{split} 
&
\dim \!\left(
\mathcal{D}(\Phi_{f_\varepsilon}) \odot 
\mathfrak{n}_{1,{\mathfrak{d}}}^{\odot (n-\ell+1 )(\dim(\mathcal{D}(\Phi_{f_\varepsilon}))
+L_{d,\varepsilon}-4 )}
\odot 
\mathcal{D}(\Phi^{d,0,K,\varepsilon}_{\ell-1,m,0})
\odot
\mathcal{D}(\mathcal{X}^{d,0,K,\varepsilon,0}_{T})\right)\\
&= (n+1)\left(\dim\!\left(\mathcal{D}(\Phi_{f_\varepsilon})\right)+L_{d,\varepsilon}-4 \right)
+\dim \! \left(\mathcal{D}(\Phi_{g^d_\varepsilon})\right)
+L_{d,\varepsilon}-2.\end{split}\label{k13}
\end{align}
Now,
\eqref{k12}--\eqref{k13} and \cref{b01b}
show that there exists
$(\Phi^{d,\theta,K,\varepsilon}_{n+1,m,t})_{t\in[0,T], \theta\in \Theta}$ such that
$t\in[0,T]$, $\theta\in \Theta$, $x\in \R^d$
we have that
\begin{align} \begin{split} 
&(\mathcal{R}(\Phi^{d,\theta,K,\varepsilon}_{n+1,m,t}))(x)
\\&=
\frac{1}{m^{n+1}}\sum_{i=1}^{m^{n+1}}g^d_\varepsilon\Bigl( X^{ d,(\theta,0,-i), K,\varepsilon,t,x}_{T} (\omega)\Bigr)\\
&\quad+
\frac{1}{m}
\sum_{i=1}^{m}
\Bigl(
f_\varepsilon\circ
U^{d,(\theta,n,i),K,\varepsilon}_{n,m}
\Bigr)
\Bigl(\mathfrak{T}_t^{(\theta,\ell,i)}(\omega), 
 X^{ d,(\theta,\ell,i), K,\varepsilon,t,x}_{\mathfrak{T}_t^{(\theta,\ell,i)}(\omega)} (\omega),\omega
\Bigr)
\\
&\quad+
\sum_{\ell=0}^{n-1}
\frac{(T-t)}{m^{n+1-\ell}}
\sum_{i=1}^{m^{n+1-\ell}}\Bigl(
f_\varepsilon\circ
U^{d,(\theta,\ell,i),K,\varepsilon}_{\ell,m}
\Bigr)
\Bigl(t+(T-t)\mathfrak{t}^{(\theta,\ell,i)}(\omega), 
 X^{ d,(\theta,\ell,i), K,\varepsilon,t,x}_{t+(T-t)\mathfrak{t}^{(\theta,\ell,i)}(\omega)}(\omega) ,\omega
\Bigr)
\\
&\quad-
\sum_{\ell=1}^{n}
\frac{(T-t)}{m^{n+1-\ell}}
\sum_{i=1}^{m^{n+1-\ell}}\Bigl(
f_\varepsilon\circ
U^{d,(\theta,-\ell,i),K,\varepsilon}_{\ell-1,m}
\Bigr)
\Bigl(t+(T-t)\mathfrak{t}^{(\theta,\ell,i)}(\omega), 
 X^{ d,(\theta,\ell,i), K,\varepsilon,t,x}_{t+(T-t)\mathfrak{t}^{(\theta,\ell,i)}(\omega)}(\omega), \omega
\Bigr)\\
&=U^{d,\theta,K,\varepsilon}_{n+1,m}(t,x),
\end{split}\label{k17}
\end{align}
\begin{align}
\dim \!\left(\mathcal{D}(\Phi^{d,\theta,K,\varepsilon}_{n+1,m,t})\right)
(n+1)\left(\dim\!\left(\mathcal{D}(\Phi_{f_\varepsilon})\right)+L_{d,\varepsilon}-4 \right)
+\dim \! \left(\mathcal{D}(\Phi_{g^d_\varepsilon})\right)
+L_{d,\varepsilon}-2,\label{k20}
\end{align}
and
\begin{align} \begin{split} 
&\mathcal{D}(\Phi^{d,\theta,K,\varepsilon}_{n+1,m,t})\\
&=\left[
\operatorname*{\boxplus}_{i=1}^{m^{n+1}}\left[
\mathfrak{n}_{1,{\mathfrak{d}}}^{\odot (n+1)(\dim(\mathcal{D}(\Phi_{f_\varepsilon}))+{L}-4)}
\odot\mathcal{D} (\Phi_{g^d_\varepsilon})
\odot\mathcal{D}(\mathcal{X}^{d,0,K,\varepsilon,0}_{T})\right]\right]\\
&\quad \boxplus
\left[\operatorname*{\boxplus}_{i=1}^m
\left[
\mathcal{D}(\Phi_{f_\varepsilon}) \odot 
\mathcal{D}(\Phi^{d,0,K,\varepsilon}_{n,m,0})
\odot
\mathcal{D}(\mathcal{X}^{d,0,K,\varepsilon,0}_{T})
\right]\right]\\
&\quad
\boxplus
\left[
\operatorname*{\boxplus}_{\ell=0}^{n-1}
\operatorname*{\boxplus}_{i=1}^{m^{n+1-\ell}}
\left[\mathcal{D}(\Phi_{f_\varepsilon}) \odot 
\mathfrak{n}_{1,{\mathfrak{d}}}^{\odot (n-\ell )(\dim (\mathcal{D}(\Phi_{f_\varepsilon}))
+L-4 ) }
\odot 
\mathcal{D}(\Phi^{d,0,K,\varepsilon}_{\ell,m,0})
\odot
\mathcal{D}(\mathcal{X}^{d,0,K,\varepsilon,0}_{T})
\right]
\right]\\
&\quad
\boxplus
\left[
\operatorname*{\boxplus}_{\ell=1}^{n}
\operatorname*{\boxplus}_{i=1}^{m^{n+1-\ell}}
\left[\mathcal{D}(\Phi_{f_\varepsilon}) \odot 
\mathfrak{n}_{1,{\mathfrak{d}}}^{\odot (n-\ell+1 )(\dim(\mathcal{D}(\Phi_{f_\varepsilon}))
+L-4 )}
\odot 
\mathcal{D}(\Phi^{d,0,K,\varepsilon}_{\ell-1,m,0})
\odot
\mathcal{D}(\mathcal{X}^{d,0,K,\varepsilon,0}_{T})
\right]
\right].\end{split}\label{k14}
\end{align}
This shows for all
$t_1,t_2\in [0,T]$, $\theta_1,\theta_2\in \Theta$ that
\begin{align}
\mathcal{D}(\Phi^{d,\theta_1,K,\varepsilon}_{n+1,m,t_1})=
\mathcal{D}(\Phi^{d,\theta_2,K,\varepsilon}_{n+1,m,t_2}).\label{k15}
\end{align}
Next, the definition of $\odot$, \eqref{a09b}, \eqref{k16}, and \eqref{c01} prove that
\begin{align} \begin{split} 
&\supnorm{\mathfrak{n}_{1,{\mathfrak{d}}}^{\odot (n+1)(\dim(\mathcal{D}(\Phi_{f_\varepsilon}))+{L}-4)}
\odot\mathcal{D} (\Phi_{g^d_\varepsilon})
\odot\mathcal{D}(\mathcal{X}^{d,0,K,\varepsilon,0}_{T})}
\\
&\leq \max\!\left\{d{\mathfrak{d}}, \supnorm{\mathcal{D} (\Phi_{g^d_\varepsilon})}, \supnorm{\mathcal{D}(\mathcal{X}^{d,0,K,\varepsilon,0}_{T})}
\right\}
\\
&\leq \max\!\left\{d{\mathfrak{d}}, \supnorm{\mathcal{D} (\Phi_{g^d_\varepsilon})}, 3\max\!\left\{ d{\mathfrak{d}},\supnorm{\mathcal{D}(\Phi_{\mu^d_\varepsilon})}
,
\supnorm{
\mathcal{D}(\Phi_{\sigma^d_\varepsilon,0})}\right\}
\right\}\\
&\leq 
3\max\!\left\{ d{\mathfrak{d}},\supnorm{\mathcal{D}(\Phi_{g^d_\varepsilon})},
\supnorm{\mathcal{D}(\Phi_{\mu^d_\varepsilon})}
,
\supnorm{
\mathcal{D}(\Phi_{\sigma^d_\varepsilon,0})}\right\}\\&\leq c_{d,\varepsilon}.
\end{split}\label{k18}\end{align}
Furthermore, the definition of $\odot$, \eqref{h03b}, \eqref{k16},  and \eqref{c01} prove that
\begin{align}
 \begin{split} 
&\supnorm{
\mathcal{D}(\Phi_{f_\varepsilon}) \odot 
\mathcal{D}(\Phi^{d,0,K,\varepsilon}_{n,m,0})
\odot
\mathcal{D}(\mathcal{X}^{d,0,K,\varepsilon,0}_{T})
}\\
&
\leq \max \!\left\{
\supnorm{\mathcal{D}(\Phi_{f_\varepsilon})},
\supnorm{\mathcal{D}(\Phi^{d,0,K,\varepsilon}_{n,m,0})}
,\supnorm{\mathcal{D}(\mathcal{X}^{d,0,K,\varepsilon,0}_{T})}
\right\}\\
&\leq 
\max \!\left\{
\supnorm{\mathcal{D}(\Phi_{f_\varepsilon})},
c_{d,\varepsilon}(3m)^n
,3\max\!\left\{ d{\mathfrak{d}},\supnorm{\mathcal{D}(\Phi_{\mu^d_\varepsilon})}
,
\supnorm{
\mathcal{D}(\Phi_{\sigma^d_\varepsilon,0})}\right\}
\right\}\\
&\leq c_{d,\varepsilon}(3m)^n.
\end{split}
\end{align}
In addition, the definition of $\odot$, \eqref{a09b}, \eqref{h03b},
\eqref{k16}, and \eqref{c01} show for all $\ell\in [0,n-1]\cap\Z$ that
\begin{align} \begin{split} 
&
\supnorm{\mathcal{D}(\Phi_{f_\varepsilon}) \odot 
\mathfrak{n}_{1,{\mathfrak{d}}}^{\odot (n-\ell )(\dim (\mathcal{D}(\Phi_{f_\varepsilon}))
+L-4 ) }
\odot 
\mathcal{D}(\Phi^{d,0,K,\varepsilon}_{\ell,m,0})
\odot
\mathcal{D}(\mathcal{X}^{d,0,K,\varepsilon,0}_{T})}\\
&
\leq \max \!\left\{
\supnorm{\mathcal{D}(\Phi_{f_\varepsilon})},d{\mathfrak{d}},
\supnorm{\mathcal{D}(\Phi^{d,0,K,\varepsilon}_{\ell,m,0})}
,\supnorm{\mathcal{D}(\mathcal{X}^{d,0,K,\varepsilon,0}_{T})}
\right\}\\
&\leq 
\max \!\left\{
\supnorm{\mathcal{D}(\Phi_{f_\varepsilon})},d{\mathfrak{d}},
c_{d,\varepsilon}(3m)^\ell
,3\max\!\left\{ d{\mathfrak{d}},\supnorm{\mathcal{D}(\Phi_{\mu^d_\varepsilon})}
,
\supnorm{
\mathcal{D}(\Phi_{\sigma^d_\varepsilon,0})}\right\}
\right\}\\
&\leq c_{d,\varepsilon}(3m)^\ell.
\end{split}\label{k19}
\end{align}
Similarly, we have for all $\ell\in [1,n]\cap\Z$ that
\begin{align}
 \begin{split} 
&
\supnorm{\mathcal{D}(\Phi_{f_\varepsilon}) \odot 
\mathfrak{n}_{1,{\mathfrak{d}}}^{\odot (n-\ell+1 )(\dim(\mathcal{D}(\Phi_{f_\varepsilon}))
+L-4 )}
\odot 
\mathcal{D}(\Phi^{d,0,K,\varepsilon}_{\ell-1,m,0})
\odot
\mathcal{D}(\mathcal{X}^{d,0,K,\varepsilon,0}_{T})}\\
&\leq c_{d,\varepsilon}(3m)^{\ell-1}.
\end{split}
\end{align}
This, \eqref{k14}, \eqref{k18}--\eqref{k19}, and the triangle inequality (cf.\ \cref{b15}) show that for all
$\theta\in \Theta$, $t\in [0,T]$ that
\begin{align}
\begin{split}
&\supnorm{\mathcal{D}(\Phi^{d,\theta,K,\varepsilon}_{n+1,m,t})}\\
&\leq \left[
\sum_{i=1}^{m^{n+1}}c_{d,\varepsilon} \right]+
\left[ \sum_{i=1}^{m}c_{d,\varepsilon}({3} m)^n\right]+
\left[ \sum_{\ell=0}^{n-1}\sum_{i=1}^{m^{n+1-\ell}}c_{d,\varepsilon}({3} m)^\ell\right]
+\left[ \sum_{\ell=1}^{n}\sum_{i=1}^{m^{n+1-\ell}}c_{d,\varepsilon}({3} m)^{\ell-1}\right]
\\
&= 
m^{n+1}c_{d,\varepsilon}+m c_{d,\varepsilon} (3m)^{n}+\left[\sum_{\ell=0}^{n-1}m^{n+1-\ell}c_{d,\varepsilon}(3m)^\ell\right]+\left[
\sum_{\ell=1}^{n}m^{n+1-\ell}c_{d,\varepsilon}(3m)^{\ell-1}\right]
\\
&= 
m^{n+1}c_{d,\varepsilon}\left[1+3^n+\sum_{\ell=0}^{n-1}3^\ell+\sum_{\ell=1}^{n}3^{\ell-1}\right]=
m^{n+1}c_{d,\varepsilon}\left[1+\sum_{\ell=0}^{n}3^\ell+\sum_{\ell=1}^{n}3^{\ell-1}\right]
\\
&\leq 
 cm^{n+1}\left[1+2\sum_{\ell=0}^{n} {3} ^\ell\right]= cm^{n+1}\left[1+2\frac{{3}^{n+1}-1}{{3}-1}\right]
= c_{d,\varepsilon}({3} m)^{n+1}.
\end{split}
\label{b18}
\end{align}
This, \eqref{k15}, \eqref{k17}, and \eqref{k20} complete the induction step. The proof of \cref{k04c} is thus completed.
\end{proof}
\section{DNN approximations for PDEs}\label{c37}
\begin{theorem}\label{a08}
Assume \cref{m07}. 
Let ${\mathfrak{d}}\in \N$, $\mathfrak{n}_{1,{\mathfrak{d}}}=(1,{\mathfrak{d}},1)\in \mathbf{D}$ satisfy that 
\begin{align}\mathrm{Id}_\R\in \mathcal{R}(\{\Phi\in \N\colon \mathcal{D}(\Phi)=\mathfrak{n}_{1,{\mathfrak{d}}}\}) .\label{c38}
\end{align}
Let $\beta,\mathfrak{p}\in [2,\infty) $, $c\in [\max\{3\mathfrak{d},\beta^2\mathfrak{p}^2\},\infty)$.
For every $d\in \N$, $\varepsilon\in (0,1)$, $v\in\R^d$
let
$\Phi_{f_\varepsilon}$,
$\Phi_{\mu_\varepsilon^d},\Phi_{\sigma_\varepsilon^d,v}, \Phi_{g^d_\varepsilon}
\in \mathbf{N}
$,
$f,f_\varepsilon\in C(\R,\R)$, $g^d,g_\varepsilon^d\in C(\R^d,\R) $,
$\mu^d,\mu^d_\varepsilon\in C(\R^d,\R^d)$,
$\sigma^d,\sigma^d_\varepsilon\in C(\R^{d\times d},\R^d)$
satisfy 
for all $v\in \R^d$
that
$f_\varepsilon= \Phi_{f_\varepsilon}$,
$\mu_\varepsilon^d=\mathcal{R}(\Phi_{\mu_\varepsilon^d}) $, 
$\sigma_\varepsilon^d (\cdot)v=\mathcal{R}(\Phi_{\sigma_\varepsilon^d,v})$,
$g^d_\varepsilon=\mathcal{R}(\Phi_{g^d_\varepsilon})$. Assume for all
 $d\in \N$, $\varepsilon\in (0,1]$,
$v\in\R^d$ that
$\mathcal{D}(\Phi_{\sigma_\varepsilon^d,v})=\mathcal{D}(\Phi_{\sigma_\varepsilon^d,0})$.
Assume for all 
$d\in \N$, $\varepsilon\in (0,1)$,
$v,w\in \R$,
$x,y\in \R^d$ that
\begin{align}
\max\{
\lVert
\mu_\varepsilon^d(x)-\mu_\varepsilon^d(y)
\rVert,
\lVert
\sigma_\varepsilon^d(x)-\sigma_\varepsilon^d(y)
\rVert
\}
\leq c\lVert x-y\rVert,\label{d04}
\end{align}
\begin{align}
\lvert f_\varepsilon(w)-f_\varepsilon(v)\rvert\leq c\lvert w-v\rvert, \quad 
\lvert
g_\varepsilon^d(x)-g_\varepsilon^d(y)
\rvert\leq c\frac{(d^c+\lVert x\rVert)^\beta+(d^c+\lVert y\rVert)^\beta}{2\sqrt{T}}\lVert x-y\rVert,\label{d09}
\end{align}
\begin{align}
\lvert g^d_\varepsilon(x)\rvert \leq c(d^c +\lVert x\rVert)^\beta,\quad \max\!\left \{
\lVert\mu^d_\varepsilon(0)\rVert,
\lVert\sigma^d_\varepsilon(0)\rVert,
\lvert Tf_\varepsilon(0)\rvert, \lvert g^d_\varepsilon(0)\rvert
\right\}\leq cd^c,\label{d05}
\end{align}
\begin{align}
\max\{
\lVert
\mu^d_\varepsilon(x)-\mu^d(x)
\rVert,
\lVert
\sigma^d_\varepsilon(x)-\sigma^d(x)
\rVert,
\lVert
g^d_\varepsilon(x)-g^d(x)
\rVert
\}\leq \varepsilon cd^c (d^c+\lVert x\rVert)^\beta,\label{d08}
\end{align}
\begin{align}
\lvert
f_\varepsilon(w)-f(w)
\rvert\leq \varepsilon (1+\lvert w\rvert^\beta),\label{d10}
\end{align}
\begin{align}
3\max\!\left\{
\supnorm{\mathcal{D}(\Phi_{f_\varepsilon})},
\supnorm{\mathcal{D}(\Phi_{g^d_\varepsilon})},
\supnorm{\mathcal{D}(\Phi_{\mu^d_\varepsilon})}
,
\supnorm{
\mathcal{D}(\Phi_{\sigma^d_\varepsilon,0})}\right\}\leq cd^c\varepsilon^{-c},\label{c01d}
\end{align}
\begin{align}
\max\!\left\{
\dim (
\mathcal{D}(\Phi_{f_\varepsilon})),
\dim (\mathcal{D}(\Phi_{g^d_\varepsilon})),
\dim (\mathcal{D}(\Phi_{\mu^d_\varepsilon}))
,
\dim (
\mathcal{D}(\Phi_{\sigma^d_\varepsilon,0}))\right\}\leq cd^c\varepsilon^{-c}
.\label{c01e}
\end{align}

Then the following items are true.
\begin{enumerate}[(i)]
\item\label{c11} For every $d\in \N$ there exists a unique at most polynomially growing viscosity solution $u^d$ of 
\begin{align}
\frac{\partial u^d}{\partial t}(t,x)
+\frac{1}{2}\mathrm{trace}(\sigma^d(x)(\sigma^d(x))^\top
(\mathrm{Hess}_x u^d(t,x) ))
+\langle \mu^d(x),(\nabla_xu^d) (t,x)\rangle
+f(u^d(t,x))=0
\end{align}
with $u^d(T,x)=g^d(x)$ for $(t,x)\in (0,T)\times \R^d$.
\item There exists $(C_{\delta})_{\delta\in (0,1)}\subseteq(0,\infty)$, $\eta\in (0,\infty)$, $(\Psi_{d,\epsilon})_{d\in \N, \epsilon\in(0,1)}\subseteq \mathbf{N}$ such that for all 
$d\in \N$, $\epsilon\in(0,1)$ we have that 
$\mathcal{R}(\Psi_{d,\epsilon})\in C(\R^d,\R)$,  
\begin{align}
\mathcal{P}(\Psi_{d,\epsilon})\leq C_\delta \eta d^\eta\epsilon^{-(4+\delta)-6c}, \quad \text{and}
\quad
\left(
\int_{[0,1]^d}
\left\lvert
(\mathcal{R}(\Psi_{d,\epsilon}))(x)
-
u^d(0,x)
\right\rvert^\mathfrak{p}
dx\right)^\frac{1}{\mathfrak{p}}<\epsilon.
\end{align}

\end{enumerate}

\end{theorem}

\begin{proof}
[Proof of \cref{a08}]
Let $p\in [3,\infty)$ satisfy that $p=\beta\mathfrak{p}$.
For every $d\in \N$ let $\varphi_d\in C(\R^d,[1,\infty))$ satisfy for all $x\in \R^d$ that
\begin{align}
\varphi_d (x)= 2^p c^p d^{pc}(d^{2c}+\lVert x\rVert^2)^{\frac{p}{2}}.\label{d06}
\end{align}
Then \eqref{d05} shows for all $d\in \N$, $\varepsilon\in(0,1)$, $x\in \R^d$ that
\begin{align} \begin{split} 
&\max
\{
\lVert \mu^d_\varepsilon(0)\rVert+c\lVert x\rVert,
\lVert \sigma^d_\varepsilon(0)\rVert+c\lVert x\rVert
\}\\
&\leq cd^c+ c\lVert x\rVert
=c(d^c+\lVert x\rVert)\leq 
2c(d^{2c}+\lVert x\rVert^2)^\frac{1}{2}\leq c (\varphi_d(x))^\frac{1}{p}.\end{split}\label{d07}
\end{align}
Next,
\cite[Lemma~2.6]{HN2022a} (applied for every $d\in \N$ with 
$d\gets d$, $m\gets d$, $a\gets d^{2c}$, $c\gets 0$,
$p\gets p/2$, 
$\mu\gets 0$, $\sigma\gets 0$, $\varphi\gets \varphi_d/(2^pc^pd^{pc})$ in the notation of \cite[Lemma~2.6]{HN2022a}) and \eqref{d06} show
for all $x,z\in \R^d$ that
\begin{align}
\lVert (\varphi_d'(x))(z) \rVert\leq p (\varphi_d(x))^{1-\frac{1}{p}} \lVert z\rVert, \quad
\lVert (\varphi_d''(x))(z,z) \rVert\leq p^2 (\varphi(x))^{1-\frac{2}{p}} \lVert z\rVert^2.
\end{align}
This, \eqref{d07}, and the fact that $p^2=\beta^2\mathfrak{p}^2\leq c$ show for all $d\in \N$,
$\varepsilon\in (0,1)$, $x,z\in \R^d$ that
\begin{align}
\max\!\left\{
\frac{\lvert(\varphi_d'(x))(z)\rvert}{(\varphi_d(x))^\frac{p-1}{p}\lVert z\rVert},
\frac{(\varphi_d''(x))(z,z)}{(\varphi_d(x))^\frac{p-2}{p}\lVert z\rVert^2},
\frac{c\lVert x\rVert+\lVert\mu^d_\varepsilon(0)\rVert}{(\varphi_d(x))^\frac{1}{p}},
\frac{c\lVert x\rVert+\lVert\sigma^d_\varepsilon(0)\rVert}{(\varphi_d(x))^\frac{1}{p}}
\right\}\leq c,\label{c08b}
\end{align}
This, \eqref{d04},  and \eqref{d08} we have for all 
$d\in \N$,
 $x,z\in \R^d$
that
\begin{align}
\max\!\left\{
\frac{\lvert(\varphi_d'(x))(z)\rvert}{(\varphi_d(x))^\frac{p-1}{p}\lVert z\rVert},
\frac{(\varphi_d''(x))(z,z)}{(\varphi_d(x))^\frac{p-2}{p}\lVert z\rVert^2},
\frac{c\lVert x\rVert+\lVert\mu^d(0)\rVert}{(\varphi_d(x))^\frac{1}{p}},
\frac{c\lVert x\rVert+\lVert\sigma^d(0)\rVert}{(\varphi_d(x))^\frac{1}{p}}
\right\}\leq c\label{c08c}
\end{align} and
\begin{align}
\max\{
\lVert
\mu^d(x)-\mu^d(y)
\rVert,
\lVert
\sigma^d(x)-\sigma^d(y)
\rVert
\}
\leq c\lVert x-y\rVert.\label{d04a}
\end{align}
Furthermore, \eqref{d09}, \eqref{d05}, and \eqref{d08}
prove for all $d\in \N$, $w,v\in \R$, $x,y\in \R^d$ that
\begin{align}
\lvert f(w)-f(v)\rvert\leq c\lvert w-v\rvert, \quad 
\lvert
g^d(x)-g^d(y)
\rvert\leq c\frac{(d^c+\lVert x\rVert)^\beta+(d^c+\lVert y\rVert)^\beta}{2\sqrt{T}}\lVert x-y\rVert,\label{d09a}
\end{align}
\begin{align}
\max\!\left \{
\lVert\mu^d(0)\rVert,
\lVert\sigma^d(0)\rVert,
\lvert f(0)\rvert, \lvert g^d(0)\rvert
\right\}\leq cd^c.\label{d05a}
\end{align}
Let $(\Omega,\mathcal{F},\P, (\F_t)_{t\in[0,T]})$ be a filtered probability space which satisfies the usual conditions.
Let 
$  \Theta = \bigcup_{ n \in \N }\! \Z^n$.
Let $\mathfrak{t}^\theta\colon \Omega\to[0,1]$, $\theta\in \Theta$, be identically distributed and independent random variables. Assume for all $t\in(0,1)$ that $\P(\mathfrak{t}^0\leq t)=t$. For every $d\in \N$ let $W^{d,\theta}\colon [0,T]\times\Omega \to \R^{d}$, $\theta\in\Theta$, be independent standard $(\F_{t})_{t\in[0,T]}$-Brownian motions. 
Assume that
$(\mathfrak{t}^\theta)_{\theta\in\Theta}$ and
$(W^{d,\theta})_{\theta\in\Theta,d\in \N}$ are independent. For every $K\in \N$ 
let
$\rdown{\cdot}_K\colon \R\to \R$ satisfy for all $t\in \R$
that $\rdown{t}_K=\max ( \{0,\frac{T}{K},\ldots,\frac{(K-1)T}{T},T\}\cap ((-\infty,t)\cup\{0\}) )$. 
For every $\theta\in \Theta$, 
$d,K\in \N$, $\varepsilon\in (0,1)$, 
$t\in[0,T]$, $x\in \R^d$
let $X^{d,\theta,K,\varepsilon,t,x}= (X^{d,\theta,K,\varepsilon,t,x}_s)_{s\in [t,T]}\colon [t,T]\times \Omega\to \R^d$ satisfy for all $s\in [t,T]$ that
$X^{d,\theta,K,\varepsilon,t,x}_t=x$ and
\begin{align} \begin{split} 
X^{d,\theta,K,\varepsilon,t,x}_s&=X^{d,\theta,K,\varepsilon,t,x}_{\max\{t,\rdown{s}_K\}}
+\mu^d_\varepsilon(X^{d,\theta,K,\varepsilon,t,x}_{\max\{t,\rdown{s}_K\}})
(s-\max \{t,\rdown{s}_K\})\\
&
\quad \qquad\qquad
+
\sigma^d_\varepsilon(X^{d,\theta,K,\varepsilon,t,x}_{\max\{t,\rdown{s}_K\}})
(W^\theta_s-W^\theta_{\max \{t,\rdown{s}_K\}}).\label{c09a}\end{split}
\end{align}
Let $U^{d,\theta,K,\varepsilon}_{n,m}\colon [0,T]\times \R^d\times \Omega\to \R$, $n\in \Z$, $K,d,m\in \N$, $\theta\in \Theta$, 
$\varepsilon\in (0,1)$,
satisfy for all
$\theta\in \Theta$, $K,d,m\in \N$, $n\in \N_0$, $t\in [0,T]$, $x\in \R^d$, $\varepsilon\in (0,1)$ that
$
{U}_{-1,m}^{d,\theta,K,\varepsilon}(t,x)={U}_{0,m}^{d,\theta,K,\varepsilon}(t,x)=0$ and
\begin{align} \begin{split} 
\label{t27d}
&  {U}_{n,m}^{d,\theta,K,\varepsilon}(t,x)
=  \frac{1}{m^n}\sum_{i=1}^{m^n}
g^d_\varepsilon(X^{d,(\theta,0,-i),K,\varepsilon,t,x}_T)\\
&+
\sum_{\ell=0}^{n-1}\frac{T-t}{m^{n-\ell}}
    \sum_{i=1}^{m^{n-\ell}}
     \bigl( f_\varepsilon\circ {U}_{\ell,m}^{d,(\theta,\ell,i),K,\varepsilon}
-\1_{\N}(\ell)
f_\varepsilon\circ {U}_{\ell-1,m}^{d,(\theta,-\ell,i),K,\varepsilon}
\bigr)\!
\left(t+(T-t)\mathfrak{t}^{(\theta,\ell,i)},X_{t+(T-t)\mathfrak{t}^{(\theta,\ell,i)}}^{d,(\theta,\ell,i),K,\varepsilon,t,x}\right).\end{split}
\end{align}
Next, \eqref{d04a} and \eqref{d04} prove for all $d\in \N$, $\varepsilon\in (0,1)$, $t\in [0,T]$, $\theta\in \Theta$ that there exist up to indistinguishability unique continuous
random fields
$X^{d,\theta,\varepsilon,t,\cdot }=(X^{d,\theta,t,x}_s)_{s\in [t,T],x\in \R^d},
X^{d,\theta,t,\cdot }=(X^{d,\theta,t,x}_s)_{s\in [t,T],x\in \R^d}
\colon [t,T]\times \R^d\times \Omega\to \R^d
$ which satisfy  for all $x\in \R^d$
 that 
$
(X^{d,\theta,\varepsilon,t,x}_s)_{s\in [t,T]}
$,
$
(X^{d,\theta,t,x}_s)_{s\in [t,T]}
$
are $(\F_s)_{s\in [t,T]}$-adapted and which satisfy  for all $s\in [t,T]$, $x\in \R^d$  that  $\P$-a.s.\
\begin{align}
X^{d,\theta,\varepsilon,t,x}_s=x+\int_{t}^{s}\mu^d_\varepsilon (X^{d,\theta,\varepsilon,t,x}_r)\,dr
+\int_t^s\sigma^d_\varepsilon(X^{d,\theta,\varepsilon,t,x}_r)\,dW^{d,\theta}_r,\label{c33}
\end{align}
\begin{align}\label{c32}
X^{d,\theta,t,x}_s=x+\int_{t}^{s}\mu^d (X^{d,\theta,t,x}_r)\,dr
+\int_t^s\sigma^d(X^{d,\theta,t,x}_r)\,dW^{d,\theta}_r.
\end{align}
Hence,
\cite[Lemma~2.1]{CHW2022} (applied for every $\theta\in \Theta$, $d\in \N$, $\varepsilon\in (0,1)$ with
$d\gets d$, $m\gets m$, $c\gets c$, $\kappa\gets 1$, $p\gets p$, $\varphi\gets \varphi_d$, $\mu\gets \mu_\varepsilon^d $,
$\sigma\gets \sigma_\varepsilon^d $
and applied for every $\theta\in \Theta$, $d\in \N$ with
$d\gets d$, $m\gets m$, $c\gets c$, $\kappa\gets 1$, $p\gets p$, $\varphi\gets \varphi_d$, $\mu\gets \mu^d $,
$\sigma\gets \sigma^d $ in the notation of 
\cite[Lemma~2.1]{CHW2022}), \eqref{d04},
and \eqref{d04a} prove
for all $d\in \N$, $\varepsilon\in (0,1)$, $x\in \R^d$ that
\begin{align}
\langle 
(\nabla\varphi_d )(x),\mu^d_\varepsilon(x)\rangle +\frac{1}{2}\mathrm{trace}
(\sigma^d_\varepsilon (\sigma^d_\varepsilon )^\top \mathrm{Hess}\varphi_d(x)  )\leq 1.5c^3\varphi_d(x),
\end{align}
\begin{align}
\langle 
(\nabla\varphi_d )(x),\mu^d(x)\rangle +\frac{1}{2}\mathrm{trace}
(\sigma^d (\sigma^d )^\top (\mathrm{Hess}\varphi_d)(x)  )\leq 1.5c^3\varphi_d(x),
\end{align}
\begin{align}
\max\!\left \{
\E \!\left[\varphi_d(X^{d,\theta,\varepsilon,t,x}_s )\right],
\E \!\left[\varphi_d(X^{d,\theta,t,x}_s )\right]
\right\}
\leq e^{1.5c^3(s-t)}\varphi(x),
\end{align}
This,
\cite[Proposition 2.2]{HJKN2020}
(applied for every $d\in \N$ with 
$d\gets d$, $L\gets c$, $T\gets T$, $\mathcal{O}\gets \R^d$, $\lVert\cdot \rVert\gets \lVert\cdot\rVert$,
$f\gets ([0,T]\times \R^d\times \R \ni (t,x,w)\mapsto f^d(w)\in \R )$, $g\gets g^d$,
$(X_{t,s}^x)_{t\in [0,T], s\in [t,T], x\in \R^d}\gets 
(X_{s}^{d,0,t,x})_{t\in [0,T], s\in [t,T], x\in \R^d}
$, $V\gets ([0,T]\times\R^d\ni (s,x)\mapsto e^{1.5c^3(T-s)}\varphi_d(s,x)\in (0,\infty) ) $
and 
applied for every $d\in \N$, $\varepsilon\in (0,1)$ with 
$d\gets d$, $L\gets c$, $T\gets T$, $\mathcal{O}\gets \R^d$, $\lVert\cdot \rVert\gets \lVert\cdot\rVert$,
$f\gets ([0,T]\times \R^d\times \R \ni (t,x,w)\mapsto f^d_\varepsilon(w)\in \R )$, $g\gets g^d_\varepsilon$,
$(X_{t,s}^x)_{t\in [0,T], s\in [t,T], x\in \R^d}\gets 
(X_{s}^{d,0,\varepsilon,t,x})_{t\in [0,T], s\in [t,T], x\in \R^d}
$, $V\gets ([0,T]\times\R^d\ni (s,x)\mapsto e^{1.5c^3(T-s)}\varphi_d(s,x)\in (0,\infty) ) $
in the notation of \cite[Proposition~2.2]{HJKN2020}), 
\eqref{d09}, \eqref{d05}, \eqref{d06}, \eqref{d09a}, and \eqref{d05a}
show for all $d\in \N$, $\varepsilon\in (0,1)$ that there exist unique measurable functions $u^d, u^{d,\varepsilon}\colon [0,T]\times \R^d \to \R$ such that for all $t\in [0,T]$, $x\in \R^d$ we have that
$
\sup_{s\in [0,T],y\in \R^d}
\frac{\lvert u^d(s,y)\rvert+\lvert u^{d,\varepsilon}(s,y)\rvert}{\varphi_d(s,y)}<\infty
$,
$\E \!\left[\lvert g^d(X^{d,0,t,x}_T)\rvert\right]
+\int_{t}^{T}\E \!\left[\lvert f( u^d(s,X^{d,0,t,x}_s) )\rvert\right]ds
+
\E \!\left[\lvert g^d(X^{d,0,\varepsilon,t,x}_T)\rvert\right]
+\int_{t}^{T}\E \!\left[\lvert f( u^d(s,X^{d,0,\varepsilon,t,x}_s) )\rvert\right]ds<\infty
$,
\begin{align}\label{c31}
u^d(t,x)=\E \!\left[g^d(X^{d,0,t,x}_T)\right]
+\int_{t}^{T}\E \!\left[f( u^d(s,X^{d,0,t,x}_s) )\right]ds
\end{align}
and
\begin{align}\label{c30}
u^{d,\varepsilon}(t,x)=\E \!\left[g^d_\varepsilon(X^{d,0,\varepsilon,t,x}_T)\right]
+\int_{t}^{T}\E \!\left[f_\varepsilon( u^d(s,X^{d,0,\varepsilon,t,x}_s) )\right]ds.
\end{align}
Next, the triangle inequality, \eqref{d04}, \eqref{d05},  and the fact that $\forall\,x\in \R\colon (1+x)^2\leq 2(1+x^2)$ show 
for all $d\in \N$, $x\in \R^d$, $\varepsilon\in (0,1)$ that
\begin{align} \begin{split} 
\langle
x,\mu^d_\varepsilon(x)
\rangle
&\leq \lVert x\rVert \left(\lVert \mu^d_\varepsilon(x)-\mu^d_\varepsilon(0)  \rVert
+
\lVert
\mu^d_\varepsilon(0)  \rVert\right)\\
&\leq 
\lVert x\rVert(c\lVert x\rVert+cd^c)\\
&
\leq (1+\lVert x\rVert)^2cd^c\\
&\leq 2cd^c(1+\lVert x\rVert^2).
\end{split}\label{c19}
\end{align}
Furthermore, the Cauchy--Schwarz inequality implies for all
$d\in \N$, $x,y\in \R^d$, $\varepsilon\in (0,1)$ that
\begin{align}
\lVert\sigma^d_\varepsilon(x)y\rVert^2= \sum_{i=1}^d \left\lvert\sum_{j=1}^d(\sigma^d_\varepsilon)_{ij}(x)y_j\right\rvert^2
\leq \sum_{i=1}^d\left(\sum_{j=1}^{d}\lvert(\sigma^d_\varepsilon)_{ij}(x)\rvert^2\right)\left(\sum_{j=1}^{d}\lvert y_j\rvert^2\right)
\leq \lVert\sigma(x)\rVert^2\lVert y\rVert^2.
\end{align}
This and \eqref{d04} show for all
$d\in \N$, $x,y\in \R^d$, $\varepsilon\in (0,1)$ that
\begin{align}
\lVert\sigma^d_\varepsilon(x)y\rVert\leq 
\lVert\sigma^d_\varepsilon(x)\rVert
\lVert y\rVert
\leq (\lVert\sigma^d_\varepsilon(x)-\sigma^d_\varepsilon(0)\rVert+\lVert\sigma^d_\varepsilon(0)\rVert)
\lVert y\rVert
\leq (c\lVert x\rVert+ cd^c)\lVert y\rVert
\leq cd^c(1+\lVert x\rVert)\lVert y\rVert.
\end{align}
This, \eqref{c19}, and \eqref{d08} prove for all
$d\in \N$, $x,y\in \R^d$, $\varepsilon\in (0,1)$ that
\begin{align}
\langle
x,\mu^d(x)
\rangle\leq 2cd^c(1+\lVert x\rVert^2),\quad 
\lVert\sigma^d(x)y\rVert \leq cd^c(1+\lVert x\rVert)\lVert y\rVert.
\end{align}
This, 
\cite[Theorem~1.1]{beck2021nonlinear}
(applied with
$d\gets d$, $L\gets 2cd^c$, $T\gets T$, $\mu\gets \mu^d$, $\sigma\gets \sigma^d$,
$f\gets (\R^d\times \R \ni (x,w)\mapsto f^d(w)\in \R )$,
$g\gets g^d$, $W\gets W^{d,\theta}$ in the notation of 
\cite[Theorem~1.1]{beck2021nonlinear}),
\eqref{d04a}, the fact that $g$ is polynomially growing (cf. \eqref{d09}--\eqref{d05}), and the fact that 
$u^d$ is polynomially growing
 show that $u^d$ is the unique at most polynomially growing viscosity solution of 
\begin{align}
\frac{\partial u^d}{\partial t}(t,x)
+\frac{1}{2}\mathrm{trace}(\sigma^d(\sigma^d(x))^\top
(\mathrm{Hess}_x u^d(t,x) ))
+\langle \mu(x),(\nabla_xu^d) (t,x)\rangle
+f(u^d(t,x))=0
\end{align}
with $u^d(T,x)=g(x)$ for $(t,x)\in (0,T)\times \R^d$. This establishes \eqref{c11}.

Next, \eqref{d04}--\eqref{d10}
show for all $d\in \N$, $\varepsilon\in (0,1)$, $x,y\in \R^d$ that
\begin{align}
\lvert Tf_\varepsilon(0)\rvert
\leq (\varphi_d (x))^\frac{\beta}{p}
,\quad 
\lvert
g^d_\varepsilon(x)
-g^d_\varepsilon(y)
\rvert\leq \frac{(\varphi_d(x) + \varphi_d(y))^\frac{\beta}{p}}{\sqrt{T}}\lVert x-y\rVert,\quad 
\lvert g^d_\varepsilon(x)\rvert\leq (\varphi_d(x))^\frac{\beta}{p},\label{d11}
\end{align}
\begin{align}
\lvert Tf (0)\rvert
\leq (\varphi_d (x))^\frac{\beta}{p}
,\quad 
\lvert
g^d (x)
-g^d (y)
\rvert\leq \frac{(\varphi_d(x) + \varphi_d(y))^\frac{\beta}{p}}{\sqrt{T}}\lVert x-y\rVert,\quad 
\lvert g^d (x)\rvert\leq (\varphi_d(x))^\frac{\beta}{p},
\end{align}
\begin{align}
\max\{\lvert f_\varepsilon(v)-f(v)\rvert,
\lVert
\mu^d_\varepsilon(x)-\mu^d(x)
\rVert,
\lVert
\sigma^d_\varepsilon(x)-\sigma^d(x)
\rVert,
\lVert
g^d_\varepsilon(x)-g^d(x)
\rVert
\}\leq \varepsilon((\varphi_d(x))^\beta + \lvert v\rvert^\beta) .\label{d08b}
\end{align}
This, 
\eqref{c08b}, \eqref{c08c}, \eqref{d04}, \eqref{d04a}, 
and
\cite[Lemma~2.4]{CHW2022}
(applied
for every $d\in \N$, $\varepsilon\in (0,1)$
 with 
$d\gets d$, $m\gets d$, $\delta\gets \varepsilon$,
$\beta\gets \beta$, $b\gets 1$, $c\gets c$, $q\gets \beta$, $p\gets p$,
$\varphi\gets \varphi_d$,
$g_1\gets g^d_\varepsilon$,
$\mu_1\gets \mu^d_\varepsilon$,
$\sigma_1\gets \sigma^d_\varepsilon$,
$f_1\gets ([0,T]\times\R^d\times\R\ni (t,x,w)\mapsto f_\varepsilon(w)\in \R)$,
$W\gets W^{d,0}$,
$(X^{x,1}_{t,s})_{x\in \R^d,t\in [0,T],s\in [t,T]}
\gets 
(X^{d,0,\varepsilon,t,x}_s)_{x\in \R^d,t\in [0,T],s\in [t,T]}
$,
$g_2\gets g^d$,
$\mu_2\gets \mu^d$,
$\sigma_2\gets \sigma^d$,
$f_2\gets ([0,T]\times\R^d\times\R\ni (t,x,w)\mapsto f(w)\in \R)$,
$(X^{x,2}_{t,s})_{x\in \R^d,t\in [0,T],s\in [t,T]}
\gets 
(X^{d,0,t,x}_s)_{x\in \R^d,t\in [0,T],s\in [t,T]}
$ in the notation of 
\cite[Lemma~2.4]{CHW2022})
show for all $d\in \N$, $\varepsilon\in (0,1)$, $t\in [0,T]$, $x\in \R^d$ that
\begin{align}
\lvert
u^{d,\varepsilon}(t,x)-u^d(t,x)\rvert
\leq \varepsilon 
2^{\beta+2}e^T e^{5\beta^2c^4 + 2^\beta c^\beta T^\beta+4c^2}
(\varphi(x))^{\beta+0.5}.\label{c12}
\end{align}
Next, \cref{c02}
(applied for all $d,K\in \N$, $\varepsilon\in (0,1)$ with
$d\gets d$, $K\gets K$, $T\gets T$, $\mathfrak{p}\gets \mathfrak{p}$,
$\beta\gets \beta$, $b\gets 1$, $c\gets c$, $p\gets p$, 
$\varphi \gets \varphi_d$,
$g\gets g^d_\varepsilon$,
$f\gets f_\varepsilon$,
$\mu \gets \mu^d_\varepsilon$,
$\sigma \gets \sigma^d_\varepsilon$,
$(\mathfrak{t}^\theta)_{\theta\in \Theta}
\gets (\mathfrak{t}^\theta)_{\theta\in \Theta}
$,
$(W^\theta)_{\theta\in \Theta}\gets 
(W^{d,\theta})_{\theta\in \Theta}
$,
$(Y^{\theta,t,x}_s)_{\theta\in \Theta, t\in [0,T], s\in [t,T],x\in \R^d }
\gets 
(X^{d,\theta,K,\varepsilon,t,x}_s)_{\theta\in \Theta, t\in [0,T], s\in [t,T],x\in \R^d }
$
in the notation of \cref{c02}), 
\eqref{c08b}, \eqref{d11}, \eqref{d09}, \eqref{d08}, \eqref{c09a}, \eqref{t27d}, and
the independence and distributional properties imply 
that for all $d,K,m,n\in \N$, $\varepsilon\in (0,1)$, 
$t\in [0,T]$, $x\in \R^d$ we have that
\begin{align} \begin{split} 
&\left\lVert
{U}_{n,m}^{d,0,K,\varepsilon}(t,x)-u^{d,\varepsilon}(t,x)\right\rVert_\mathfrak{p}\leq  12c^2e^{9c^3T}
(\varphi_d(x))^{\frac{\beta+1}{p}}\left[
2\mathfrak{p}^{\frac{n}{2}}e^{5cTn}e^{m^{\mathfrak{p}/2}/\mathfrak{p}}m^{-n/2}+\frac{1}{\sqrt{K}}
\right].
\end{split}\label{c34}
\end{align}
This, the triangle inequality, and \eqref{c12} prove for all $d,K,m,n\in \N$, $\varepsilon\in (0,1)$, 
$t\in [0,T]$, $x\in \R^d$ that
\begin{align} \begin{split} 
&
\left\lVert
{U}_{n,m}^{d,0,K,\varepsilon}(t,x)-
u^d(t,x)
\right\rVert_\mathfrak{p}\leq 
\left\lVert
{U}_{n,m}^{d,0,K,\varepsilon}(t,x)-u^{d,\varepsilon}(t,x)\right\rVert_\mathfrak{p}
+\lvert
u^{d,\varepsilon}(t,x)-u^d(t,x)\rvert\\
&\leq 12^{\beta+2}c^2
e^{9c^3T+T+5\beta^2c^4+2^\beta c^\beta T^\beta+4c^2}(\varphi_d(x))^{\beta+0.5}
\left[
2\mathfrak{p}^{\frac{n}{2}}e^{5cTn}e^{m^{\mathfrak{p}/2}/\mathfrak{p}}m^{-n/2}+\frac{1}{\sqrt{K}}+
\varepsilon\right].
\end{split}
\end{align}
Hence, 
\eqref{d06} shows that
there exists $\kappa\in (0,\infty)$ such that for all
$ d,K,n,m\in \N$, $\varepsilon\in (0,1)$ we have that 
\begin{align} \begin{split} 
&
\left(
\int_{[0,1]^d}
\left\lVert
{U}_{n,m}^{d,0,K,\varepsilon}(0,x)-
u^d(0,x)
\right\rVert_\mathfrak{p}^\mathfrak{p}\,dx
\right)^\frac{1}{\mathfrak{p}}\\
&
\leq 
12^{\beta+2}c^2
e^{9c^3T+T+5\beta^2c^4+2^\beta c^\beta T^\beta+4c^2}
\left(\int_{[0,1]^d}
(\varphi_d(x))^{\mathfrak{p}(\beta+0.5)}\,dx
\right)^\frac{1}{\mathfrak{p}}
\left[
2\mathfrak{p}^{\frac{n}{2}}e^{5cTn}e^{m^{\mathfrak{p}/2}/\mathfrak{p}}m^{-n/2}+\frac{1}{\sqrt{K}}+
\varepsilon\right]\\
&\leq \kappa d^\kappa
\left[
2\mathfrak{p}^{\frac{n}{2}}e^{5cTn}e^{m^{\mathfrak{p}/2}/\mathfrak{p}}m^{-n/2}+\frac{1}{\sqrt{K}}+
\varepsilon\right]\\
&\leq 
\kappa d^\kappa
\left[\left(\frac{
2\mathfrak{p}^{\frac{1}{2}}e^{5cT}\exp (\frac{m^{\mathfrak{p}/2}}{n})}{m^\frac{1}{2}}\right)^n+\frac{1}{\sqrt{K}}+
\varepsilon\right].
\end{split}\label{c13}
\end{align}
For the next step let $(M_n)_{n\in \N}\colon \N\to \N$ satisfy  that
$\liminf_{j\to\infty}M_j=\infty$, $\limsup_{j\to\infty} \frac{(M_j)^{\mathfrak{p}/2}}{j}<\infty$, and $\sup_{k\in \N}\frac{M_{k+1}}{M_k}<\infty $ (see, e.g., \cite[Lemma~4.5]{HJKP2021} for an example). For every $\varepsilon\in (0,1)$, $d\in \N$ let
\begin{align}\label{c22}
K_\varepsilon= \inf\{k\in \N\colon 1/\sqrt{k}\leq \varepsilon\},
\end{align}
\begin{align}
N_\varepsilon=\inf\!\left\{
n\in \N\colon \left(
\frac{2\mathfrak{p}^{\frac{1}{2}}e^{5cT}\exp ( \frac{(M_n)^{\mathfrak{p}/2}}{n})}{(M_n)^\frac{1}{2}}
\right)^n\leq \varepsilon\right\},\label{c14}
\end{align}
\begin{align}
L_{d,\varepsilon}=K_\varepsilon(
\max\{\dim(\mathcal{D}(\Phi_{\mu^d_\varepsilon})) ,\dim(\mathcal{D}(\Phi_{\sigma^d_\varepsilon,0}))\}
-2)+2,\label{c15}
\end{align}
\begin{align}
c_{d,\varepsilon}=
3\max\!\left\{ d\mathfrak{d},
\supnorm{\mathcal{D}(\Phi_{f_\varepsilon})},
\supnorm{\mathcal{D}(\Phi_{g^d_\varepsilon})},
\supnorm{\mathcal{D}(\Phi_{\mu^d_\varepsilon})}
,
\supnorm{
\mathcal{D}(\Phi_{\sigma^d_\varepsilon,0})}\right\}
.\label{c16}
\end{align}
For every $\epsilon\in (0,1)$, $d\in \N$ let 
\begin{align}\label{c20}
\varepsilon(d,\epsilon)=\frac{\epsilon}{3\kappa d^\kappa}
\end{align}
For every $\delta\in (0,1)$ let
\begin{align}
C_\delta=\sup_{\varepsilon\in (0,1)}\left[\varepsilon^{4+\delta}N_\varepsilon(3M_{N_\varepsilon})^{2N_\varepsilon}\right].\label{c21}
\end{align}
Then \cite[Lemma~5.1]{AJKP2024}
(applied with
$L\gets 1$, $T\gets 2\mathfrak{p}^{\frac{1}{2}}e^{5cT}-1$, $(m_{k})_{k\in \N}\gets (M_{k})_{k\in \N}$
in the notation of \cite[Lemma~5.1]{AJKP2024})
implies for all $\delta\in (0,1)$ that
$C_\delta<\infty$. Furthermore,
\eqref{c22} prove for all $\varepsilon\in (0,1)$ that
$\frac{1}{\sqrt{K_\varepsilon-1}}>\varepsilon$, i.e.,
\begin{align}
K_\varepsilon=K_\varepsilon-1+1<\varepsilon^{-2}+1<2\varepsilon^{-2}.\label{c23}
\end{align}
Next,  Tonelli's theorem, \eqref{c13}--\eqref{c14}, and \eqref{c20} show for all
$d\in \N$, $\epsilon\in (0,1)$ that
\begin{align} \begin{split} 
&\E \!\left[
\int_{[0,1]^d}
\left\lvert
{U}_{N_{\varepsilon(d,\epsilon)},M_{N_{\varepsilon(d,\epsilon)}}}^{d,0,K_{\varepsilon(d,\epsilon)},\varepsilon(d,\epsilon)}(t,x)-
u^d(t,x)
\right\rvert^\mathfrak{p}
dx\right]\\
&=
\int_{[0,1]^d}
\E \!\left[
\left\lvert
{U}_{N_{\varepsilon(d,\epsilon)},M_{N_{\varepsilon(d,\epsilon)}}}^{d,0,K_{\varepsilon(d,\epsilon)},\varepsilon(d,\epsilon)}(t,x)-
u^d(t,x)
\right\rvert^\mathfrak{p}
\right]dx\\
&\leq \left(\kappa d^\kappa
\left[ \left(
\frac{2\mathfrak{p}^{\frac{1}{2}}e^{5cT}\exp\!\left ( \frac{\left(M_{N_{\varepsilon(d,\epsilon)}}\right)^{\mathfrak{p}/2}}{N_{\varepsilon(d,\epsilon)}}\right)}{(M_{N_{\varepsilon(d,\epsilon)}})^\frac{1}{2}}
\right)^n+\frac{1}{\sqrt{K_{\varepsilon(d,\epsilon)}}}+
\varepsilon(d,\epsilon)\right]\right)^\mathfrak{p}\\
&\leq \left(\kappa d^\kappa\left[\varepsilon(d,\epsilon)+\varepsilon(d,\epsilon)+\varepsilon(d,\epsilon)\right] \right)^\mathfrak{p}=\epsilon^\mathfrak{p}.
\end{split}
\end{align}
Therefore, for every $d\in \N$, $\epsilon\in (0,1)$ there exists $\omega(d,\epsilon)\in \Omega$ such that
\begin{align}
\int_{[0,1]^d}
\left\lvert
{U}_{N_{\varepsilon(d,\epsilon)},M_{N_{\varepsilon(d,\epsilon)}}}^{d,0,K_{\varepsilon(d,\epsilon)},\varepsilon(d,\epsilon)}(t,x, \omega(d,\epsilon))-
u^d(t,x)
\right\rvert^\mathfrak{p}
dx<\epsilon^\mathfrak{p}.\label{c26}
\end{align}
Furthermore, \cref{k04c} and \eqref{c15} prove that there exists
$(\Phi^\omega_{d,\varepsilon})_{d\in \N,\varepsilon \in (0,1), \omega\in \Omega}\subseteq \mathbf{N}$ such that for all
$d\in \N $, $\varepsilon\in (0,1)$, $\omega\in \Omega$
we have that
\begin{align} \begin{split} 
\dim(\mathcal{D}(\Phi_{d,\varepsilon} ))
&=N_\varepsilon\left(\dim\!\left(\mathcal{D}(\Phi_{f_\varepsilon})\right)
+L_{d,\varepsilon}-4 \right)
+\dim \! \left(\mathcal{D}(\Phi_{g^d_\varepsilon})\right)
+L_{d,\varepsilon}-2,
\end{split}\label{c18}\end{align}
\begin{align}
\supnorm{\mathcal{D}(\Phi^\omega_{d,\varepsilon})}
\leq c_{d,\varepsilon}(3M_{N_\varepsilon})^{N_\varepsilon}
,\quad {U}_{N_\varepsilon,M_{N_\varepsilon}}^{d,0,K_\varepsilon,\varepsilon}(0,x,\omega)=(\mathcal{R}(\Phi^\omega_{d,\varepsilon}))(x).\label{c19a}\end{align}
Next, \eqref{c15} and \eqref{c01e} show for all
$d\in \N$, $\varepsilon\in (0,1)$ that
$
L_{d,\varepsilon}\leq K_\varepsilon cd^c \varepsilon^{-c}
$. This,  
\eqref{c18}, and \eqref{c01e} show for all
$d\in \N$, $\varepsilon\in (0,1)$, $\omega\in \Omega$ that
\begin{align} \begin{split} 
\dim(\mathcal{D}(\Phi^\omega_{d,\varepsilon} ))
&\leq 
4
N_\varepsilon L_{d,\varepsilon}\max \{ \dim(\mathcal{D}(\Phi_{f_\varepsilon})), 
\dim(\mathcal{D}(\Phi_{g^d_\varepsilon}))
\}\\
&\leq 4N_\varepsilon K_\varepsilon cd^c\varepsilon^{-c}cd^c\varepsilon^{-c}
=4N_\varepsilon K_\varepsilon c^2d^{2c}\varepsilon^{-2c}
\end{split}\label{c17}
\end{align}
Next, \eqref{c16}, \eqref{c01d}, and the fact that $c\geq 3\mathfrak{d}$ prove for all
$d\in \N$, $\varepsilon\in (0,1)$ that
\begin{align}
c_{d,\varepsilon}=
3\max\!\left\{ d\mathfrak{d},
\supnorm{\mathcal{D}(\Phi_{f_\varepsilon})},
\supnorm{\mathcal{D}(\Phi_{g^d_\varepsilon})},
\supnorm{\mathcal{D}(\Phi_{\mu^d_\varepsilon})}
,
\supnorm{
\mathcal{D}(\Phi_{\sigma^d_\varepsilon,0})}\right\}
\leq cd^c \varepsilon^{-c}
.\label{c01b}
\end{align}
This and \eqref{c19a} imply for all
$d\in \N$, $\varepsilon\in (0,1)$, $\omega\in \Omega$ that
\begin{align}
\supnorm{\mathcal{D}(\Phi^\omega_{d,\varepsilon})}
\leq c_{d,\varepsilon}(3M_{N_\varepsilon})^{N_\varepsilon}
\leq cd^c\varepsilon^{-c}(3M_{N_\varepsilon})^{N_\varepsilon}.
\end{align}
This, the fact that 
$\forall\, \Phi\in\mathbf{N}\colon \mathcal{P}(\Phi)\leq 2\dim (\mathcal{D}(\Phi))\supnorm{\mathcal{D}(\Phi)}^2$,  \eqref{c17}, 
\eqref{c23},
\eqref{c21}, and the fact that
$\forall\,\delta\in (0,1)\colon C_\delta<\infty$ prove for all
$d\in \N$, $\varepsilon,\delta\in (0,1)$ that
\begin{align} \begin{split} 
\mathcal{P}(\Phi_{d,\varepsilon}^\omega)
&\leq 2\cdot 4N_\varepsilon K_\varepsilon c^2d^{2c}\varepsilon^{-2c}
\left(cd^c\varepsilon^{-c}(3M_{N_\varepsilon})^{N_\varepsilon}\right)^2\\
&= 8  K_\varepsilon c^4 d^{4c}\varepsilon^{-4c}
N_\varepsilon
(3M_{N_\varepsilon})^{2N_\varepsilon}\\
&=
8  K_\varepsilon c^4 d^{4c}\varepsilon^{-4c}
\varepsilon^{4+\delta}N_\varepsilon
(3M_{N_\varepsilon})^{2N_\varepsilon}
\varepsilon^{-(4+\delta)}
\\
&\leq 8\cdot 2\varepsilon^{-2}c^4d^{4c}\varepsilon^{-4c}
C_\delta \varepsilon^{-(4+\delta)}\\
&=16 C_\delta c^4d^{4c}\varepsilon^{-(6+\delta)-4c}<\infty.\end{split}\label{c24}
\end{align}
Next, for every $d\in \N$, $\epsilon\in (0,1)$ let \begin{align}\label{c25}
\Psi_{d,\epsilon}= \Phi^{\omega(d,\epsilon)}_{d,\varepsilon(d,\epsilon)}.
\end{align} Then
\eqref{c24} and \eqref{c20} show for all $d\in \N$, $\epsilon\in (0,1)$ that
\begin{align} \begin{split} 
\mathcal{P}(
\Psi_{d,\epsilon})=\mathcal{P}( \Phi^{\omega(d,\epsilon)}_{d,\varepsilon(d,\epsilon)})
&\leq 
16 C_\delta c^4d^{4c}(\varepsilon(d,\epsilon))^{-(6+\delta)-4c}\\
&=
16 C_\delta c^4d^{4c}\left(\frac{\epsilon}{3\kappa d^{\kappa}}\right)^{-(6+\delta)-4c}\\
&=16 C_\delta(3\kappa)^{(6+\delta)+4c} c^4d^{4c+\kappa((6+\delta)+4c)}\epsilon^{-(6+\delta)-4c}.
\end{split}\label{c27}
\end{align}
Furthermore, 
\eqref{c19a},
\eqref{c25},   and \eqref{c26} imply for all $d\in \N$, $\epsilon\in (0,1)$ that 
$
{U}_{N_{\varepsilon(d,\epsilon)},M_{N_{\varepsilon(d,\epsilon)}}}^{d,0,K_{\varepsilon(d,\epsilon)},\varepsilon(d,\epsilon)}(0,x, \omega(d,\epsilon))
=(\mathcal{R}(\Phi^{\omega(d,\epsilon)}_{d,\varepsilon(d,\epsilon)}))(x)=(\mathcal{R}(\Psi_{d,\epsilon}))(x)
$
and 
\begin{align}
\int_{[0,1]^d}
\left\lvert
(\mathcal{R}(\Psi_{d,\epsilon}))(x)
-
u^d(0,x)
\right\rvert^\mathfrak{p}
dx<\epsilon^\mathfrak{p}.
\end{align}
This, \eqref{c27}, and the fact that $\forall\,\delta\in (0,1)\colon C_\delta<\infty$ complete the proof of \cref{a08}.
\end{proof}
\begin{proof}
[Proof of \cref{a08b}]
The definitions of $\mathfrak{a}_0,\mathfrak{a}_1$, the fact that $a\in \{\mathfrak{a}_0,\mathfrak{a}_1\}$, and Lemmas~\ref{a10} and \ref{a10b} show that there exists
 ${\mathfrak{d}}\in \N$, $\mathfrak{n}_{1,{\mathfrak{d}}}=(1,{\mathfrak{d}},1)\in \mathbf{D}$ such that
\begin{align}
\mathrm{Id}_\R\in \mathcal{R}(\{\Phi\in \N\colon \mathcal{D}(\Phi)=\mathfrak{n}_{1,{\mathfrak{d}}}\}) .\label{t02}
\end{align}
Next, the definitions of $\mathfrak{a}_0,\mathfrak{a}_1$, the fact that $a\in \{\mathfrak{a}_0,\mathfrak{a}_1\}$, and Lemmas~\ref{p10} and \ref{p10a} prove that there exists
$\tilde{c}\in (0,\infty)$,
$(f_\varepsilon)_{\varepsilon\in (0,1)}\subseteq C(\R,\R)$ such that for all $\varepsilon\in (0,1) $, $x,y\in \R$
we have that 
\begin{align}
\lvert f_\varepsilon(x)-f_\varepsilon(y)\rvert\leq c\lvert x-y\rvert, \quad 
\lvert f_\varepsilon(x)-f(x)\rvert\leq \varepsilon(1+\lvert x\rvert^\beta),\label{t01}
\end{align}
 \begin{align}
 f_\varepsilon\in \mathcal{R}(\{\Phi\in \mathbf{N}\colon \dim(\mathcal{D}(\Phi ))=3
, \supnorm{\mathcal{D}(\Phi )}\leq \tilde{c}\varepsilon^{-\tilde{c}}
\}).\label{c28}
\end{align} Then \eqref{d05b} implies for all $\varepsilon\in (0,1)$ that \begin{align}
T\lvert f_\varepsilon(0)\rvert\leq T(\lvert f(0)\rvert+\varepsilon)\leq c+T.
\end{align} Furthermore,
\eqref{c01eb} proves that there exists $\bar{c}\in (0,\infty)$ such that for all $d\in \N$, $\varepsilon\in (0,1)$ that
\begin{align}
\max\!\left\{
\supnorm{\mathcal{D}(\Phi_{f_\varepsilon})},
\supnorm{\mathcal{D}(\Phi_{g^d_\varepsilon})},
\supnorm{\mathcal{D}(\Phi_{\mu^d_\varepsilon})}
,
\supnorm{
\mathcal{D}(\Phi_{\sigma^d_\varepsilon,0})}\right\}\leq \bar{c}d^{\bar{c}}\varepsilon^{-\bar{c}},\label{c01g}
\end{align}
\begin{align}
\max\!\left\{
\dim (
\mathcal{D}(\Phi_{f_\varepsilon})),
\dim (\mathcal{D}(\Phi_{g^d_\varepsilon})),
\dim (\mathcal{D}(\Phi_{\mu^d_\varepsilon}))
,
\dim (
\mathcal{D}(\Phi_{\sigma^d_\varepsilon,0}))\right\}\leq \bar{c}d^{\bar{c}}\varepsilon^{-\bar{c}}
.\label{c01f}
\end{align}
Combining \eqref{t02}--\eqref{c01f}, the assumptions of \cref{a08b}, and \cref{a08} (applied with $c$ replaced by a suitable large constant) we complete the proof of \cref{a08b}.
\end{proof}

\bibliographystyle{acm}
\bibliography{References}
\appendix
\section{Code for Example~\ref{e02a}}
\begin{listing}\label{l03}
Here comes the code for Example~\ref{e02a}.
The code was written in
Julia (see \url{https://julialang.org}). We used a laptop with 16GB RAM, 12th Gen Intel Core i5-1240P
x 16, Operating System: Ubuntu 22.04.4 LTS 64 bit.
Note that the global variable \texttt{count} is introduced to count
the number of real-valued random variables needed for the MLP
approximations. 
The following code should be saved
under the name \texttt{example.jl}. To run the code we type \texttt{julia example.jl}. The outputs will be contained in \texttt{example.csv}  and \texttt{example.png}. 
Note that before
running the code we may need to first install the packages used here.

\begin{lstlisting}
using Plots,LaTeXStrings,
DataFrames,CSV, LinearAlgebra,Distributions,Random

function Y(t,s,x;d,mu,sigma,N)
	global count
	y=x; cur=t; h=T/N;
	nex=ceil(t/h)*h
	if (nex==cur)
		nex= min(cur+h,s)
    end
	if (nex==s)
		count=count+d
   		y= y+ mu(d,y)*(nex-cur)\
            + sigma(d,y)*rand(Normal(),d)*sqrt(nex-cur)
		return y
    end
	while (nex<s)
		count=count+d
   		y= y+ mu(d,y)*(nex-cur)
            + sigma(d,y)*rand(Normal(),d)*sqrt(nex-cur)
		cur=nex
		nex= min(cur+h,s)
    end
	return y
end

function Uref(t,x;d,mu,sigma,f,g,N,n,m)
	if (n==0)
		return 0.0
    end
	s=0.0;
	for i in 1:m^n
        s=s+ g(d,Y(t,T,x;d=d,mu=mu,sigma=sigma,N=N))/(m^n)
    end
	for l in 0:(n-1)
		for i in 1:m^(n-l)
			r=t+(T-t)*rand(Uniform(0,1))
            y=Y(t,r,x,d=d,mu=mu,sigma=sigma,N=N)
			if (l>=1)
                s=s+ (f(Uref(r,y,d=d,mu=mu,
                sigma=sigma,f=f,g=g,N=N,n=l,m=m))
                    -f(Uref(r,y,d=d,mu=mu,
                    sigma=sigma,f=f,g=g,N=N,n=l-1,m=m)))/m^(n-l)
			else
                s=s+ f(Uref(r,y,d=d,mu=mu,
                sigma=sigma,f=f,g=g,N=N,n=l,m=m))/m^(n-l)
            end
        end
    end
	return s;
end

function U(t,x;d,mu,sigma,f,g,n,m)
	if (n==0)
		return 0.0
    end
	s=0.0;
    for l in 0:(n-1)
    	for i in 1:m^n
            if (l>=1)
                s=s+ (g(d,Y(t,T,x,d=d,mu=mu,sigma=sigma,N=m^l))                   
                    -g(d,Y(t,T,x,d=d,mu=mu,sigma=sigma,N=m^(l-1))) )
                    /(m^n);
            else
                s=s+ g(d,Y(t,T,x,d=d,mu=mu,sigma=sigma,N=m^l))/(m^n);
            end
        end
    end
	for l in 0:(n-1)
		for i in 1:(m^(n-l))
            r=t+(T-t)*rand(Uniform(0,1))
            y=Y(t,r,x,d=d,mu=mu,sigma=sigma,N=m^m)
            if (l>=1)
                s=s+ (f(U(r,y,d=d,mu=mu,
                sigma=sigma,f=f,g=g,n=l,m=m))
                    -f(U(r,y,d=d,mu=mu,
                    sigma=sigma,f=f,g=g,n=l-1,m=m)))/m^(n-l)
			else
                s=s+ f(U(r,y,d=d,mu=mu,sigma=sigma,f=f,g=g,n=l,m=m))
                /m^(n-l)
            end
        end
    end
	return s;
end

function M(n)    
    return  floor(exp(sqrt(log(n))))
end

function mu(d,x)
	global count
	count=count+1
	return cos(norm(d,x))*x
end
function norm(d,x)
	s=0.0; for i in 1:d s=s+x[i]^2 end; return sqrt(s)
end
function sigma(d,x)
	global count
	count=count+1
	return  Matrix(1.0I, d, d) 
end
function f(y)
	global count;
	count=count+1;
	return  sin(y);
end
function g(d,x)
	global count;
	count=count+1;
	return (2+2/5*norm(d,x)^2)^(-1);
end


T=1.0;
d=100;
count=0;
n_max=4;
diff=zeros(n_max);
effort=zeros(n_max);
runtime=zeros(n_max);
u1=Uref(0.0,zeros(d);d=d,mu=mu,sigma=sigma,f=f,g=g,N=10000,n=4,m=4)
U(0.0,zeros(d),d=d,mu=mu,sigma=sigma,f=f,g=g,n=0,m=0)

for n in 1:n_max
    m=Int(M(n));
    global count;count=0
    ti=time()
    s=0.0
    for i in 1:5
        u2= U(0.0,zeros(d),d=d,mu=mu,sigma=sigma,f=f,g=g,n=n,m=m);
        s=s+abs(u1-u2)^4;
    end    
    runtime[n]=time()-ti
    diff[n]=s^(1/4);
    effort[n]=count
end

df=DataFrame( error=diff,runtime=runtime,effort=effort)
CSV.write("example.csv", df)
println(df)
plot(effort,[diff effort.^(-1/2) effort.^(-1/4)], 
xaxis=:log10, yaxis=:log10, 
    label=["error" "line "*L"y=x^{-1/2}" "line "*L"y=x^{-1/4}"], 
    ls=[:solid :dash :dot], 
    xlabel="computational effort", ylabel="error",
    title=L"d=%$d")
savefig("example.png") 

\end{lstlisting}
\end{listing}
\end{document}